\documentclass{article} 
\usepackage{iclr2024_conference,times}

\usepackage{hyperref}
\usepackage{url}
\usepackage{amsthm}
\usepackage{caption}
\usepackage{subcaption}
\usepackage{wrapfig}

\usepackage[utf8]{inputenc} 
\usepackage[T1]{fontenc}    
\usepackage{hyperref}       
\usepackage{url}            
\usepackage{booktabs}       
\usepackage{amsfonts}       
\usepackage{nicefrac}       
\usepackage{microtype}      
\usepackage{xcolor}         
\usepackage{hyperref}
\usepackage{nccmath}
\usepackage{makecell}
\usepackage{multirow}
\usepackage{booktabs}

\usepackage{algorithm}
\usepackage[noend]{algorithmic}

\usepackage{threeparttable, tablefootnote}

\newtheorem{innercustomthm}{Theorem}
\newenvironment{customthm}[1]
  {\renewcommand\theinnercustomthm{#1}\innercustomthm}
{\endinnercustomthm}
  
\newtheorem{innercustomlemma}{Lemma}
\newenvironment{customlemma}[1]
  {\renewcommand\theinnercustomlemma{#1}\innercustomlemma}
{\endinnercustomlemma}

\DeclareMathOperator{\Span}{Span}

\title{Advancing the lower bounds:\\
An accelerated, stochastic, second-order method with optimal adaptation to inexactness}

\newtheorem{theorem}{Theorem}[section]

\newtheorem{lemma}[theorem]{Lemma}
\newtheorem{corollary}[theorem]{Corollary}
\newtheorem{definition}[theorem]{Definition}
\newtheorem{assumption}[theorem]{Assumption}

\usepackage{kamzolov}
\usepackage{graphicx}
\usepackage{enumitem}
\usepackage{refcount}

\newcommand{\eqdef}{:=}

\def\bk{\bar \kappa}

\usepackage{pifont}
\usepackage{multirow}
\newcommand{\xmark}{\ding{55}}%

\newcommand{\CRN}{\text{CRN }}

\def\aa#1{{\color{black}#1}} 
\def\aaa#1{{\color{black}#1}} 
\def\dk#1{{\color{black}#1}}

\author{%
A. Agafonov \textsuperscript{1, 2} \And D. Kamzolov\textsuperscript{1}\And A. Gasnikov \textsuperscript{2,4,5} \And A. Kavis \textsuperscript{3} \And K. Antonakopoulos \textsuperscript{3} \And V.  Cevher\textsuperscript{3} \And M. Tak\'a\v{c}\textsuperscript{2}\\ \\ 
\textbf{\textsuperscript{1}}The University of Hong Kong \quad \textbf{\textsuperscript{2}}ByteDance Inc. \quad \textbf{\textsuperscript{3}}Apple Inc. \\
\texttt{\{lzheng2,lpk\}@cs.hku.hk} \\
\texttt{jianbo.yuan@bytedance.com} \quad \texttt{mr.chongwang@apple.com}
}

\author{Artem Agafonov \textsuperscript{1, 2}, Dmitry Kamzolov\textsuperscript{1}, Alexander Gasnikov \textsuperscript{2,4,5}, Ali Kavis \textsuperscript{3}, 
\And Kimon Antonakopoulos \textsuperscript{3}, Volkan  Cevher\textsuperscript{3}, Martin Tak\'a\v{c}\textsuperscript{1}  \\ \\
\textbf{\textsuperscript{1}} Mohamed bin Zayed University of Artificial Intelligence, Abu Dhabi, UAE\\ \textbf{\textsuperscript{2}} Moscow Institute of Physics and Technology, Dolgoprudny, Russia\\
\textbf{\textsuperscript{3}} Laboratory for Information and Inference Systems, IEM STI EPFL, Lausanne, Switzerland\\
\textbf{\textsuperscript{4}} Innopolis University, Kazan, Russia\\
\textbf{\textsuperscript{5}} Skoltech, Moscow, Russia
\vspace{-0.15in} 
}

\usepackage{footmisc}

\usepackage{array}
\usepackage{arydshln}
\iclrfinalcopy
\begin{document}

\maketitle
\begin{abstract}

We present a new accelerated stochastic second-order method that is robust to both gradient and Hessian inexactness, which occurs typically in machine learning. We establish theoretical lower bounds and prove that our algorithm achieves optimal convergence in both gradient and Hessian inexactness in this key setting.  We further introduce a tensor generalization for stochastic higher-order derivatives. When the oracles are non-stochastic, the proposed tensor algorithm matches the global convergence of Nesterov Accelerated Tensor method. Both algorithms allow for approximate solutions of their auxiliary subproblems with verifiable conditions on the accuracy of the solution.

\if
    We present the first accelerated stochastic second-order method that achieves optimal convergence in both gradient and Hessian inexactness.
    Compared to previous state-of-the-art algorithms, our proposed method improves the convergence rate for inaccuracies in gradients and Hessians, while matching their convergence rate in the term corresponding for the exact convergence.
    We provide lower bounds on the corresponding convergence terms for stochastic gradient and inexact Hessian. 
    We further introduce a tensor generalization for inexact higher-order derivatives. 
    Both algorithms allow for approximate solutions of their auxiliary subproblems with verifiable conditions on the accuracy of the solution.
\fi
\end{abstract}
 \vspace{-13pt}
\section{Introduction}
 \vspace{-8pt}

In this paper, we consider the following general convex optimization problem:
\begin{equation}\label{eq:main_problem}
    \min\limits_{x\in \R^d} f(x),
\end{equation}
where $f(x)$ is a convex and sufficiently smooth function.
We assume that a solution $x^{\ast}\in \R^d$ exists and we denote $f^{\ast}:=f(x^{\ast})$. We define $R = \|x_0 - x^{\ast}\|$ as a distance to the solution.
\begin{assumption}\label{as:lip}
The function $f(x)\in C^2$ has $L_2$-Lipschitz-continuous Hessian if for any $x,y \in \R^d$
\begin{equation*}
\label{eq:L2}
    \|\nabla^2 f(x) - \nabla^2 f(y) \| \leq L_2\|x-y\|.
\end{equation*}
\end{assumption}
\vspace{-0.3 cm}
Since the calculation of an exact gradient is very expensive or impossible in many applications in several domains, including machine learning, statistics, and signal processing, efficient methods that can work with inexact stochastic gradients are of great interest. 
\begin{assumption}
    \label{as:1ord_stoch}
        For all $x \in \R^d$, we assume that stochastic gradients $g(x, \xi)\in \R^d$ satisfy
        \begin{equation}
        \label{eq:stoch_grad_def}
                \mathbb{E}[g(x, \xi) \mid x] = \nabla f(x), \quad
                \mathbb{E}\left[\|g(x, \xi)-  \nabla f(x)\|^2 \mid x\right] \leq \sigma_1^2.
        \end{equation}
\end{assumption}
\vspace{-0.3 cm}
Extensive research has been conducted on first-order methods, both from a theoretical and practical perspective. For $L_1$-smooth functions with stochastic gradients characterized by a variance of $\sigma_1^2$, lower bound $\Omega \ls \tfrac{\sigma_1 R}{\sqrt{T}} + \tfrac{L_1R^2}{T^2}\rs$ has been established by \cite{nemirovski1983problem}. For $f(x)=\E [F(x,\xi)]$, the stochastic approximation (SA) was developed, starting from the pioneering paper by \cite{robbins1951stochastic}. Important improvements of the SA were developed by \citet{polyak1990new, polyak1992acceleration, nemirovski2009robust}, where longer stepsizes with iterate averaging and proper step-size modifications were proposed, obtaining the rate $O \ls \tfrac{\sigma_1 R}{\sqrt{T}} + \tfrac{L_1R^2}{T}\rs$. The optimal method matching the lower bounds has been developed by \cite{lan2012optimal} with a convergence rate $O \ls \tfrac{\sigma_1 R}{\sqrt{T}} + \tfrac{L_1R^2}{T^2}\rs$. However, the literature on second-order methods is significantly limited for the study of provable, globally convergent stochastic second-order methods for convex minimization. 
\\[2pt]
{\bf Second-order methods.} Although second-order methods have been studied for centuries \citep{newton1687philosophiae, raphson1697analysis, simpson1740essays, kantorovich1949newton,more1977levenberg, griewank1981modification}, most of the results are connected with local quadratic convergence. The significant breakthroughs regarding global convergence have been achieved only recently, starting from the paper on Cubic Regularized Newton (CRN) method by \cite{nesterov2006cubic}, the first second-order method with a global convergence rate $O \ls \tfrac{L_2R^3}{T^2} \rs$. Following this work, \cite{nesterov2008accelerating} proposes an acceleration mechanism on top of CRN and achieves the convergence rate of $O \ls \tfrac{L_2R^3}{T^3} \rs$, going beyond the $\dk{\Omega}(1/T^2)$ lower bound for first-order methods. Another cornerstone in the field is the work by \cite{monteiro2013accelerated}, which achieves lower complexity bound $\Omega \ls\tfrac{L_2R^3}{T^{7/2}}\rs$ \citep{pmlr-v75-agarwal18a,arjevani2019oracle} up to a logarithmic factor, for the first time in the literature. The gap between upper and lower bounds was closed only in 2022 in subsequent works of \cite{kovalev2022first, carmon2022optimal}. \\[2pt]
One of the main limitations of the second-order methods is a high per-iteration cost as they require computation of exact Hessian. Therefore, it is natural to use approximations of derivatives instead of their exact values. 
In \citep{ghadimi2017second}, \CRN method with $\delta_2$-inexact Hessian information and its accelerated version were proposed, achieving convergence rate $O\ls \frac{\delta_2R^2}{T^2} + \frac{L_2R^3}{T^3} \rs$. This algorithm was later extended by \citet{agafonov2023inexact} to handle $\delta_1$-inexact gradients (and high-order derivatives) with a resulting convergence rate of $O\ls \delta_1R + \frac{\delta_2R^2}{T^2} + \frac{L_2R^3}{T^3} \rs$. A recent paper by \citet{antonakopoulos2022extra} proposes a stochastic \textit{adaptive} second-order method based on the extragradient method without line-search and with the convergence rate $O\ls \frac{\sigma_1R}{\sqrt{T}} + \frac{\sigma_2R^2}{T^{3/2}} + \frac{L_2R^3}{T^3} \rs$ when gradients and Hessians are noisy with variances $\sigma_1^2$ and $\sigma_2^2$. In the light of these results, we identify several shortcomings and open questions: 
\begin{center}
    \textit{What are the lower bounds for inexact second-order methods?\\What is the optimal trade-off between inexactness in the gradients and the Hessian?}
\end{center}

In this work, we attempt to answer these questions in a systematic manner. Detailed descriptions of other relevant studies can be found in Appendix \ref{app:rel_works}. 
\begin{table}[h]
    \centering
    \caption{
    Comparison of existing results for second-order methods under inexact feedback.
    $T$ denotes the number of iterations, and $L_2$ represents the Lipschitz constant of the Hessian.
    }
    \resizebox{\columnwidth}{!}{
        \begin{tabular}{c c c c c }
            \toprule
            Algorithm & Inexactness & \makecell{Gradient \\ convergence} & \makecell{Hessian \\ convergence} & \makecell{Exact \\ convergence} \\
            \midrule
            
            \makecell{Accelerated Inexact Cubic Newton \\ \citep{ghadimi2017second}}
            & 
            \makecell{
            exact gradient \\ 
            $\delta_2$-inexact Hessian 
            \tablefootnote{
                $\delta_2$-inexact Hessian: $\frac{\delta_2}{2} I  \preceq H_x - \nabla^2 f(x)  \preceq \delta_2 I$
            }
            }  
            & 
            \xmark 
            & 
            $O\ls \frac{\delta_2 R^2}{T^2} \rs$ & $O\ls \frac{L_2 R^3}{T^3} \rs$ 
            \\
            \hdashline
             \makecell{Accelerated Inexact Tensor Method  
             \tablefootnote{
                  It is worth noting that the Accelerated Inexact Tensor Method can also be applied to the case of stochastic derivatives. Specifically, when $p=2$, the total number of stochastic gradient computations is on the order of $O(\e^{-7/3})$, while the total number of stochastic Hessian computations is on the order of $O(\e^{-2/3})$ \cite{agafonov2023inexact}. In our work, we propose an algorithm that achieves the same number of stochastic Hessian computations but significantly improves the number of stochastic gradient computations to $O(\e^{-2})$.
             }\\
            \citep{agafonov2023inexact}
            } 
            & 
            \makecell{
            $\delta_1$-inexact gradient 
            \tablefootnote{
                $\delta_1$-inexact gradient: $\|g_x - \nabla f(x)\| \leq \delta_1$
            } 
            \\
            $\delta_2$-inexact Hessian 
            \tablefootnote{ 
                $\delta_2$-inexact Hessian: $\|\left(H_{x} - \nabla^2 f(x)\right)(y - x)\| \leq \delta_2 \|y-x\|$}
            }  
            & 
            $O(\delta_1R)$
            & 
            $O\ls \frac{\delta_2 R^2}{T^2} \rs$ 
            & 
            $O\ls \frac{L_2 R^3}{T^3} \rs$ \\
            \hdashline
            \makecell{Extra-Newton \\ \citep{antonakopoulos2022extra}} & 
            \makecell{
             stochastic gradient \eqref{eq:stoch_grad_def}
            \\
            unbiased stochastic Hessian 
            \tablefootnote{ 
                Unbiased stochasic Hessian: $ \mathbb{E}[H(x, \xi) \mid x]=\nabla^2 f(x)$, ~~
                $\mathbb{E}\left[\left\|H(x, \xi)-\nabla^2 f(x)\right\|^2 \mid x\right] \leq \sigma_2^2$ 
            }
            }
            & 
            $O\ls\frac{\sigma_1 R}{\sqrt{T}}\rs$ 
            & 
            $O\ls \frac{\sigma_2 R^2}{T^{3/2}} \rs$ 
            & 
            $O\ls \frac{L_2 R^3}{T^3} \rs$ \\
            \hdashline
            \makecell{Accelerated Stochastic  Second-order method \\{[This Paper]}} &
            \makecell{
            stochastic gradient \eqref{eq:stoch_grad_def}\\
            stochastic Hessian \eqref{eq:stoch_hessian}
            }
            & 
            $O\ls\frac{\sigma_1 R}{\sqrt{T}}\rs$ 
            & 
            $O\ls \frac{\sigma_2 R^2}{T^2} \rs$
            \tablefootnote{
                Under assumption of $\delta_2$-inexact Hessian the convergence is $O\ls \frac{\delta_2 R^2}{T^2} \rs$ \label{fn:other}
            } 
            & 
            $O\ls \frac{L_2 R^3}{T^3} \rs$ 
            \\
            \hdashline
            \makecell{Lower bound \\ {[This Paper]}} 
            &
            \makecell{
            stochastic gradient \eqref{eq:stoch_grad_def}\\
            stochastic Hessian \eqref{eq:stoch_hessian}
            } 
            & 
            $\Omega\ls\frac{\sigma_1 R}{\sqrt{T}}\rs$ 
            & 
            \makecell{$\Omega\ls \frac{\sigma_2 R^2}{T^2} \rs $}
            & 
            \makecell{$\Omega\ls \frac{L_2 R^3}{T^{7/2}} \rs $}  
            \\
            \bottomrule
        \end{tabular}
    }
    \label{tab:comparison}
\end{table}

{\bf Contributions.} We summarize our contributions as follows:
\begin{enumerate}[noitemsep,topsep=0pt,leftmargin=12pt]
    \item 
    We propose an accelerated second-order algorithm that achieves the convergence rate of $O\ls \frac{\sigma_1R}{\sqrt{T}} + \frac{\sigma_2 R^2}{T^2} + \frac{L_2 R^3}{T^3} \rs$ for stochastic Hessian with variance $\sigma_2^2$ and $O\ls \frac{\sigma_1 R}{\sqrt{T}} + \frac{\delta_2 R^2}{T^2} + \frac{L_2 R^3}{T^3} \rs$ for $\delta_2$-inexact Hessian, improving the existing results~\citep{agafonov2023inexact, antonakopoulos2022extra} (see Table~\ref{tab:comparison}).
    
    \item 
    We prove that the above bounds are tight with respect to the variance of the gradient and the Hessian by developing a matching \aaa{theoretical complexity} lower bound (see Table~\ref{tab:comparison}).
    \item 
    Our algorithm involves solving a cubic subproblem that arises in several globally convergent second-order methods \citep{nesterov2006cubic, nesterov2008accelerating}. To address this, we propose a criterion based on the accuracy of the subproblem's gradient, along with a dynamic strategy for selecting the appropriate level of inexactness. This ensures an efficient solution of the subproblems without sacrificing the fast convergence of the initial method.
    \item 
    We extend our method for higher-order minimization with stochastic/inexact oracles. We achieve the $O\ls \frac{\sigma_1 R}{\sqrt{T}} + \sum \limits_{i=2}^p \frac{\delta_iR^i}{T^{i}} + \frac{L_{p}R^{p+1}}{T^{p+1}} \rs$ rate with $\delta_i$-inexact $i$-th derivative.
    \item 
    We propose a restarted version of our algorithm for strongly convex minimization, which exhibits a linear rate. Via a mini-batch strategy, we demonstrate that the total number of Hessian computations scales linearly with the desired accuracy $\e$.
\end{enumerate}

\vspace{-13pt}
\section{Problem statement and preliminaries} 
\vspace{-8pt}

\label{sec:problem}

{\bf Taylor approximation and oracle feedback.}
Our starting point for constructing second-order method is based primarily on the second-order Taylor approximation of the function $f(x)$ 
\begin{equation*}\label{eq:2ord_taylor}
    \Phi_{x}(y) \stackrel{\text { def }}{=} f(x)+\la \nabla f(x), y - x\ra + \tfrac{1}{2}\la y - x, \nabla^2 f(x)(y-x) \ra, \quad y \in \mathbb{R}^d.
\end{equation*}

In particular, since the exact computation of the Hessians can be a quite tiresome task, we attempt to employ more tractable inexact estimators $g(x)$ and $H(x)$ for the gradient and Hessian. These estimators are going to be the main building blocks for the construction of the "inexact" second-order Taylor approximation. Formally, this is given by:
 \begin{equation}\label{eq:2ord_approx_taylor}
    \phi_{x}(y)=f\left(x\right)+\la g(x), y - x\ra + \tfrac{1}{2}\la y - x,  H(x)(y-x) \ra, \quad y \in \mathbb{R}^d.
\end{equation}

Therefore, by combining Assumption \ref{as:lip} with the aforementioned estimators, we readily get the following estimation:
\begin{lemma}[\aaa{{\citep[Lemma 2]{agafonov2023inexact}}}]
    \label{lm:2ord_bound}
     Let Assumption \ref{as:lip} hold. Then, for any $x,y \in \mathbb{R}^d$, we have
    \begin{equation*}
    \begin{split}\label{eq:func_bnd}
         |f( y) - \phi_{{x}}(y)|  
         \leq 
         \ls
         \|g(x) - \nabla f(x)\|
         + 
         \tfrac{1}{2}\|\ls H(x) - \nabla^2 f(x)\rs (y-x)\|
         \rs\|y-x\| 
         +
         \tfrac{L_2}{6}\|y - x\|^{3}.
     \end{split}
\end{equation*}

\begin{equation*}
    \begin{split}\label{eq:1deriv_bnd}
         \|\nabla f(y) - \nabla \phi_{x}(y)\| 
         \leq  \|g(x) - \nabla f(x)\| + \|\ls H(x) - \nabla^2 f(x)\rs (y-x)\| + \tfrac{L_2}{3}\|y - x\|^{2}.
    \end{split}
\end{equation*}
\end{lemma}

Now, having established the main toolkit concerning the approximation of $f$ in the rest of this section, we introduce the blanket assumptions regarding the inexact gradients and Hessians (for a complete overview, we refer to Table~\ref{tab:comparison}). In particular, we assume that our estimators satisfy the following statistical conditions.
\begin{assumption}
    [Unbiased stochastic gradient with bounded variance and \aaa{stochastic Hessian with bounded variance}] 
    \label{as:2ord_stoch}
        For all $x \in \R^d$, stochastic gradient $g(x, \xi)$ satisfies ~\eqref{eq:stoch_grad_def} and stochastic Hessian $H(x, \xi)$ satisfies
        \begin{equation}
        \label{eq:stoch_hessian}
            \mathbb{E}\left[\|H(x, \xi) -  \nabla^2 f(x)\|_2^2 \mid x\right] \leq \sigma_2^2. 
        \end{equation}
\end{assumption}

\begin{assumption}
    [Unbiased stochastic gradient with bounded variance and inexact Hessian]
    \label{as:2ord_stoch_grad_inexact_hess} 
        For all $x \in \R^d$ stochastic gradient $g(x, \xi)$ satisfies~\eqref{eq:stoch_grad_def}. For given 
        $x,~y \in \R^d$ inexact Hessian $H(x)$ satisfies
        \begin{gather}  
                \|(H(x) - \nabla^2 f({x}))[y - x]\| \le \delta_2^{x, y} \|y - x\| \label{eq:inexact_hess_def}.
        \end{gather}
\end{assumption}

Assumptions~\ref{as:2ord_stoch} and ~\ref{as:2ord_stoch_grad_inexact_hess} differ from Condition 1 in \citep{agafonov2023inexact}  by the unbiasedness of the gradient. An unbiased gradient allows us to attain optimal convergence in the corresponding term $ O(1/\sqrt{T}) $, while an inexact gradient slows down the convergence to $O(1)$ since a constant error can misalign the gradient. Note, that we do not assume the unbiasedness of the Hessian in all assumptions.
Finally, note that Assumption~\ref{as:2ord_stoch_grad_inexact_hess} does not require~\eqref{eq:inexact_hess_def} to be met for all $x, ~ y \in \R^d$. Instead, we only consider inexactness along the direction $y - x$, which may be significantly less than the norm of the difference between Hessian and its approximation $H(x)$.
\\ 
{\bf Auxiliary problem.} Most second-order methods with global convergence require solving an auxiliary subproblem at each iteration. However, to the best of our knowledge, existing works that consider convex second-order methods under inexact derivatives do not account for inexactness in the solution of the subproblem. To address this gap, we propose incorporating a gradient criteria for the subproblem solution, given by
\begin{gather}
    \textstyle{\min} _{y \in \R^d} \omega_{x}(y)~~\text{such that}~~\|\nabla \omega_{x}(y)\| \leq \tau,
\end{gather}
where $\omega_{x}(y)$ is the objective of subproblem and $\tau \geq 0$ is a tolerance parameter. We highlight that this criterion is verifiable at each step of the algorithm, which facilitates determining when to stop. By setting a constant tolerance parameter $\tau$, we get the following relationship between the absolute accuracy $\epsilon$ required for the initial problem and $\tau$: $\tau = O\ls\epsilon^{\frac{5}{6}}\rs$. In practice, it may not be necessary to use a very small accuracy in the beginning. Later, we will discuss strategies for choosing the sequence of $\tau_t$ based on the number of iterations $t$.

\vspace{-13pt}
\section{The method}
\vspace{-8pt}
\label{sec:method}
In this section, we present our proposed method, dubbed as Accelerated Stochastic Cubic Regularized Newton's method. In particular, extending on recent accelerated second-order algorithms \citep{nesterov2021implementable, ghadimi2017second, agafonov2023inexact}, we propose a new variant of the accelerated cubic regularization method with stochastic gradients that achieves optimal convergence in terms corresponding to gradient and Hessian inexactness. Moreover, the proposed scheme allows for the approximate solution of the auxiliary subproblem, enabling a precise determination of the required level of subproblem accuracy.\\[2pt]
We begin the algorithm description by introducing the main step. Given constants $\bar{\delta} > 0$ and $M \geq L_2$, we define a model of the objective 
\begin{equation*}
    \omega^{M, \bar{\delta}}_x(y) \eqdef \phi_{x}(y) + \tfrac{\bar{\delta}}{2}\|x-y\|^2 + \tfrac{M}{6}\|x-y\|^{3}.
\end{equation*}
At each step of the algorithm, we aim to find $u \in \argmin_{y \in \R^d} \omega_{x}^{M, \bar{\delta}}(y)$. However, finding the respective minimizer is a separate challenge. Instead of computing the exact minimum, we aim to find a point $s \in \R^d$ with a small norm of the gradient.
\begin{definition} \label{def:inexact_sub}
    Denote by $s^{M, \bar \delta, \tau} (x)$ a $\tau$-inexact solution of subproblem, i.e. a point $s \eqdef s^{M, \bar \delta, \tau} (x)$ such that
    \begin{equation*}
        \|\nabla \omega^{M, \bar{\delta}}_{x}(s)\| \leq \tau.
    \end{equation*}
\end{definition}
\vspace{-7pt}
Next, we employ the technique of estimating sequences to propose the Accelerated Stochastic Cubic Newton method.
Such acceleration is based on aggregating stochastic linear models given by
\begin{equation*}
    l(x,y)=f(y)+\la g(y, \xi), x-y\ra
\end{equation*}
in function $\psi_t(x)$ \aaa{\eqref{eq:psi},~\eqref{eq:subproblem_2ord}}.
The method is presented in detail in Algorithm \ref{alg:inexact_acc_detailed}.

\begin{algorithm}
  \caption{Accelerated Stochastic Cubic Newton}\label{alg:inexact_acc_detailed}
  \begin{algorithmic}[1]
      \STATE \textbf{Input:} $y_0 = x_0$ is starting point; constants $M \geq 2L_2$; 
     non-negative non-decreasing sequences $\{\bar{\delta}_t\}_{t \geq 0}, \{\lambda_t\}_{t \geq 0}, \{\bk_2^t\}_{t \geq 0}$, $\{\bk_3^t\}_{t \geq 0}$, and
      \vspace{-6pt}
      \begin{equation}
        \label{eq:alphas}
          \alpha_t = \tfrac{3}{t + 3}, ~~~ A_t = \textstyle{\prod \limits_{j=1}^t}(1 -\alpha_j), ~~~ A_0 = 1,
      \end{equation}
      \vspace{-10pt}
      \begin{equation}\label{eq:psi}
        \psi_{0}(x):= \tfrac{\bk_2^{0}+\lambda_0}{2}\|x - x_0\|^2+\tfrac{\bk_3^{0}}{3}\|x - x_0\|^3 .
    \end{equation}
    \vspace{-0.1cm}
    \FOR{$t \geq 0$} 
        \STATE 
                \[v_t = (1 - \alpha_t)x_t + \alpha_t y_t, \quad x_{t+1} = s^{M, \bar{\delta}_t, \tau}(v_{t})\]
                \vspace{-0.3cm}
        \STATE Compute 
            \begin{equation}
            \label{eq:subproblem_2ord}
            \begin{aligned}
                y_{t+1}=\arg \min _{x \in \mathbb{R}^{n}}&\left\{\psi_{t+1}(x):= \psi_{t}(x)+ \tfrac{\lambda_{t+1} - \lambda_{t}}{2}\|x - x_0\|^2 
                \right.\\
                &\left.
                + \textstyle{\sum \limits_{i = 2}^{3} }
                \tfrac{\bk^{t+1}_i - \bk^{t}_i}{i}\|x - x_0\|^i
                +\tfrac{\alpha_{t}}{A_{t}} 
                l(x,x_{t+1}) \right\}.
                \end{aligned}
            \end{equation}
    \ENDFOR
  \end{algorithmic}
\end{algorithm}

\begin{theorem}\label{thm:acc_convergence}
    Let Assumption \ref{as:lip} hold and $M \geq 2L_2$. 
    \begin{itemize}[leftmargin=10pt,nolistsep]
        \item 
            Let Assumption \ref{as:2ord_stoch} hold. After $T \geq 1$ with parameters
            \begin{gather} \label{eq:params}
                \bk_2^{t+1} =  \tfrac{2\bar{\delta}_t\al_t^2}{A_t}, ~  \bk_{3}^{t+1} = \tfrac{8M}{3}\tfrac{\alpha_{t+1}^{3}}{A_{t+1}}, ~
                \lambda_t = \tfrac{\sigma_1}{R}(t+3)^{\frac{5}{2}},~
                \bar{\delta}_t = 2\sigma_2 + \tfrac{\sigma_1 + \aa{\tau}}{R}(t+3)^{\frac{3}{2}},
            \end{gather}
            we get the following bound 
            \begin{equation}
            \label{eq:acc_convergence_stoch}
            \begin{aligned}
               \E \left[ f(x_{T}) - f(x^{\ast}) \right]  
               &\leq 
               O \ls \tfrac{\tau R}{\sqrt{T}} + \tfrac{\sigma_1 R}{\sqrt{T}} + \tfrac{\sigma_2 R^{2}}{T^{2}}  
                + \tfrac{MR^{3}}{T^{3}}\rs.
            \end{aligned}
            \end{equation}
        \item 
            Let Assumption \ref{as:2ord_stoch_grad_inexact_hess} hold. After $T \geq 1$ with parameters defined in  \eqref{eq:params} and \\
            $\sigma_2 = \delta_2 = \max \limits_{t=1, \ldots, T} \delta_t^{v_{t-1}, x_{t}}$,
            we get the following bound 
            \begin{equation}
            \label{eq:acc_convergence_stoch_grad_inexact_hess}
            \begin{aligned}
               \E  [ f(x_{T}) - f(x^{\ast})   ]  
               &\leq 
               O \ls \tfrac{\tau R}{\sqrt{T}} + \tfrac{\sigma_1 R}{\sqrt{T}} + \tfrac{\delta_2 R^{2}}{T^{2}}  
                + \tfrac{M R^{3}}{T^{3}}\rs.
            \end{aligned}
            \end{equation}
    \end{itemize}    
\end{theorem}
\vspace{-0.3cm}

This result provides an upper bound for the objective residual after $T$ iterations of Algorithm \ref{alg:inexact_acc_detailed}. The last term in the RHS of \eqref{eq:acc_convergence_stoch} and \eqref{eq:acc_convergence_stoch_grad_inexact_hess} corresponds to the case of exact Accelerated Cubic Newton method \citep{nesterov2008accelerating}. The remaining terms reveal how the convergence rate is affected by the imprecise calculation of each derivative and by inexact solution of subproblem.
We provide sufficient conditions for the inexactness in the derivatives to ensure that the method can still obtain an objective residual smaller than $\e$.
Specifically, this result addresses the following question: given that the errors are controllable and can be made arbitrarily small, how small should each derivative's error be to achieve an $\e$-solution?

\begin{corollary}\label{cor:2ord_inexactenss_acc}
    Let assumptions of Theorem \ref{thm:acc_convergence} hold and let $\e > 0$ be the desired solution accuracy.
    \begin{itemize}[leftmargin=10pt,nolistsep]
        \item Let the levels of inexactness in Assumption \ref{as:2ord_stoch} be:
        \begin{gather*}
            \tau = O\ls\e^{\frac{5}{6}}\ls \tfrac{M}{R^3} \rs^{\frac{1}{6}} \rs, 
            \quad 
            \sigma_1 = O\ls\e^{\frac{5}{6}}\ls \tfrac{M}{R^3} \rs^{\frac{1}{6}} \rs,
            \quad 
            \sigma_2 = O\ls\e^{\frac{1}{3}}M^\frac{2}{3}\rs
        \end{gather*}
        \item Let the levels of inexactness in Assumption \ref{as:2ord_stoch_grad_inexact_hess} be:
        \begin{gather*}
            \tau = O\ls\e^{\frac{5}{6}}\ls \tfrac{M}{R^3} \rs^{\frac{1}{6}} \rs, 
            \quad 
            \sigma_1 = O\ls\e^{\frac{5}{6}}\ls \tfrac{M}{R^3} \rs^{\frac{1}{6}} \rs,
            \quad 
            \delta_2 = O\ls\e^{\frac{1}{3}}M^\frac{2}{3}\rs
        \end{gather*}
    \end{itemize}
    \vspace{-0.3cm}
    And let the number of iterations of Algorithm \ref{alg:inexact_acc_detailed} satisfy $T = O \ls \frac{MR^3}{\e} \rs^{\frac{1}{3}}$. 
    Then $x_{T}$ is an $\e$-solution of problem \eqref{eq:main_problem}, i.e. $f(x_{T})-f(x^{\ast})\leq \e$. 
\end{corollary}

In practice, achieving an excessively accurate solution for the subproblem on the initial iterations is not essential. Instead, a dynamic strategy can be employed to determine the level of accuracy required for the subproblem. Specifically, we can choose a dynamic precision level according to \aaa{$\tau_t = \frac{c}{t^{5/2}}$}, where $c>0$. As a result, the convergence rate term associated with the inexactness of the subproblem becomes $O \ls \frac{c}{T^{3}} \rs$, which matches the convergence rate of the Accelerated Cubic Newton method.
 \vspace{-13pt}
\section{\aaa{Theoretical complexity lower bound}}
 \vspace{-8pt}
\label{sec:lb}
In this section, we present a novel \aaa{theoretical complexity} lower bound for inexact second-order methods with stochastic gradient and inexact (stochastic) Hessian. The proof technique draws inspiration from the works \citep{devolder2014first,nesterov2021implementable,nesterov2018lectures}. For this section, we assume that the function $f(x)$ is convex and has $L_1$-Lipschitz-continuous gradient and $L_2$-Lipschitz-continuous Hessian.
\\
To begin, we describe the information and structure of stochastic second-order methods. At each point $x_t$, the oracle provides us with an unbiased stochastic gradient $g_t=g(x_t,\xi)$ and an inexact (stochastic) Hessian $H_t = H(x_t,\xi)$. The method can compute the minimum of the following models:
\begin{equation*}
  h_{t+1} = \textstyle{\argmin}_h \lb\phi_{x_t}(h) = a_1 \la g_t, h\ra + a_2 \la H_t h, h\ra + b_1 \|h\|^2 + b_2 \|h\|^3\rb.
\end{equation*}
Now, we formulate the main assumption regarding the method's ability to generate new points.
\begin{assumption}
\label{as:lower_bound_sequence}
    The method generates a recursive sequence of test points ${x_t}$ that satisfies the following condition
    \begin{equation*}
        x_{t+1} \in x_0 + \Span \lb h_1, \ldots, h_{t+1 }\rb
    \end{equation*}
\end{assumption}
Most first-order and second-order methods, including accelerated versions, typically satisfy this assumption. However, we highlight that randomized methods are not covered by this lower bound. Randomized lower bound even for exact high-order methods is still an open problem. More details on randomized lower bounds for first-order methods are presented in \citep{woodworth2017lower, nemirovski1983problem}. Finally, we present the main \aaa{theoretical complexity} lower bound theorem for stochastic second-order methods.
\begin{theorem}\label{thm:lower_bound}
            Let some second-order method $\mathcal{M}$ with exactly solved subproblem satisfy Assumption \ref{as:lower_bound_sequence} and have access only to unbiased stochastic gradient and inexact Hessian satisfying Assumption \ref{as:2ord_stoch} or Assumption \ref{as:2ord_stoch_grad_inexact_hess} with $\sigma_2 = \delta_2=\max \limits_{t=1, \ldots, T} \delta_t^{x_{t-1}, x_{t}}$. Assume the method $\mathcal{M}$ ensures for any function $f$ with $L_1$-Lipschitz-continuous gradient and $L_2$-Lipschitz-continuous Hessian the following convergence rate 
            \begin{equation}
            \label{eq:lower_bound_convergence_stoh}
               \min\limits_{0\leq t \leq T} \E \left[  f(x_{t}) - f(x^{\ast}) \right]  
               \leq O(1) \max \lb \tfrac{\sigma_1 R}{\Xi_1(T)}; \tfrac{\sigma_2 R^{2}}{\Xi_2(T)};  
                \tfrac{L_2R^{3}}{\Xi_3(T)}\rb.
            \end{equation}
            Then for all $T\geq 1$ we haven
            \begin{equation}
            \label{eq:lower_bound_convergence_powers}
                \Xi_1 (T) \leq \sqrt{T}, \qquad \Xi_2 (T) \leq T^2, \qquad \Xi_3 (T) \leq T^{7/2}.
            \end{equation}
\end{theorem}
\begin{proof}
\vspace{-5pt}
We prove this Theorem from contradiction.
Let assume that there exist the method $\mathcal{M}$ that satisfies conditions of the Theorem \ref{thm:lower_bound} and it is faster in one of the bounds from \eqref{eq:lower_bound_convergence_powers}.

The first case, $\Xi_1 (T) > \sqrt{T}$ or  $\Xi_2 (T) > T^2$. Let us apply this method for the first-order lower bound function. It is well-known, that for the first-order methods, the lower bound is $\Omega\ls \frac{\sigma_1 R}{\sqrt{T}}+ \frac{L_1 R^{2}}{T^2}\rs$ \citep{nemirovski1983problem}. Also, the first-order lower bound function has $0$-Lipschitz-continuous Hessian. It means, that the method $\mathcal{M}$ can be applied for the first-order lower-bound function. We fix stochastic Hessian oracle as $H(x,\xi) = 2L_1 I$. It means that $\sigma_2= 2L$ for such inexact Hessian. With such matrix $H(x,\xi)= 2L_1 I$, the method $\mathcal{M}$ has only the first-order information and lies in the class of first-order methods. Hence, we apply the method $\mathcal{M}$ to the first-order lower bound function and get the rate $\min\limits_{0\leq t \leq T}\E \left[  f(x_{t}) - f(x^{\ast}) \right]  \leq O(1) \max \lb \frac{\sigma_1 R}{\Xi_1(T)}; \frac{\sigma_2 R^{2}}{\Xi_2(T)}\rb$, where $\Xi_1 (T) > \sqrt{T}$ or  $\Xi_2 (T) > T^2$. It means that we've got a faster method than a lower bound. It is a contradiction, hence the rates for the method $\mathcal{M}$ are bounded as $\Xi_1 (T) \leq \sqrt{T},  \Xi_2 (T) \leq T^2$.
The second case, $\Xi_3 (T) > T^{7/2}$. It is well-known, that the deterministic second-order lower bound is $\Omega\ls\frac{L_2 R^{3}}{T^{7/2}}\rs$. Let us apply the method $\mathcal{M}$ for the second-order lower bound function, where the oracle give us exact gradients and exact Hessians, then $\sigma_1=0$, $\sigma_2=0$ and the method $\mathcal{M}$ is in class of exact second-order methods but converges faster than the lower bound. It is a contradiction, hence the rate for the method $\mathcal{M}$ is bounded as $\Xi_3 (T) \leq T^{7/2}$.
\end{proof}
 \vspace{-13pt}
\section{Tensor generalization}
 \vspace{-8pt}
\label{sec:tensor}
In this section we propose a tensor generalization of Algorithm \ref{alg:inexact_acc_detailed}. We start with introducing the standard assumption on the objective $f$ for tensor methods.
\begin{assumption}\label{as:lip_p}
    Function $f$ is convex,  $p$ times differentiable on $\mathbb{R}^d$, and its $p$-th derivative is  Lipschitz continuous, i.e. for all $
    x, y \in \mathbb{R}^d$
    $$\|\nabla^p f(x) - \nabla^p f(y)\| \leq L_\aaa{p} \|x - y\|.$$ 
\end{assumption}

We denote the $i$-th directional derivative of function $f$ at $x$ along directions $s_1, \ldots, s_i \in \mathbb R^n$ as  
$\nabla^i f(x)[s_1, \ldots, s_i].$
If all directions are the same we write $\nabla^i f(x)[s]^i.$
For a $p$-th order tensor $U$, we denote by $\|U\|$ its tensor norm recursively induced \citep{cartis2017improved} by the Euclidean norm on the space of $p$-th order tensors: 
$$\|U\| = \max \limits_{\|s_1\| =  \ldots = \| s_p\| = 1} \{|U[s_1, \ldots, s_p]| \},$$
where  $\|\cdot\|$ is the standard Euclidean norm.
\\
We construct tensor methods based on the $p$-th order Taylor approximation of the function $f(x)$, which can be written as follows:
\begin{equation*}\label{eq:taylor}
    \Phi_{x,p}(y) \stackrel{\text { def }}{=} f(x)+\textstyle{\sum}  _{i=1}^{p} \tfrac{1}{i !} \nabla^{i} f(x)[y-x]^{i}, \quad y \in \mathbb{R}^d.
\end{equation*}
Using approximations $G_i(x)$ for the derivatives $\nabla^i f\left(x\right)$ we create an inexact $p$-th order Taylor series expansion of the objective 
\begin{equation*}\label{eq:approx_taylor}
    \phi_{x,p}(y)=f\left(x\right)+ \textstyle{\sum} _{i = 1}^p   \tfrac{1}{i!}G_{i}(x)[y-x]^i.
\end{equation*}
Next, we introduce a counterpart of Lemma \ref{lm:2ord_bound} for high-order methods.
\begin{lemma}[\aaa{{\cite[Lemma 2]{agafonov2023inexact}}}]
    \label{lm:p_ord_bound}
     Let Assumption \ref{as:lip_p} hold. Then, for any $x,y \in \mathbb{R}^d$, we have
    \begin{equation*}
    \begin{split}
         |f( y) - \phi_{{x},p}(y)|  
         \leq \textstyle{\sum \limits_{i = 1}^p} \frac{1}{i!}\|(G_i(x) - \nabla^i f({x}))[y - x]^{i-1}\|\|y-x\| + \frac{L_p}{(p+1)!}\|y - x\|^{p+1},
     \end{split}
\end{equation*}
    \vspace{-10pt}
\begin{equation*}
    \begin{split}
         \|\nabla f( y) - \nabla \phi_{x,p}(y)\| 
         \leq  \textstyle{\sum \limits_{i = 1}^{p}} \frac{1}{(i-1)!}\|(G_i(x) - \nabla^i f({x}))[y - x]^{i-1}\| + \frac{L_p}{p!}\|y - x\|^{p},
    \end{split}
\end{equation*}
where we use the standard convention $0!=1$.
\end{lemma}
\vspace{-0.2cm}
Following the assumptions for the second-order method we introduce analogical assumptions for high-order method. 
\begin{assumption}
    [Unbiased stochastic gradient with bounded variance and \aaa{stochastic high-order derivatives with bounded variance}] 
    \label{as:p_ord_stoch}
        For any $x \in \R^d$ stochastic gradient $G_1(x, \xi)$ and stochastic high-order derivatives $G_i(x, \xi), ~ i=2,\ldots, p$ satisfy
        \begin{gather}
                \mathbb{E}[G_1(x, \xi) \mid x] = \nabla f(x), \quad
                \mathbb{E}\left[\|G_1(x, \xi)-  \nabla f(x)\|^2 \mid x\right] \leq \sigma_1^2, \label{eq:p_ord_stoch_grad_def} \\
                \mathbb{E}\left[\|G_i(x, \xi) -  \nabla^i f(x)\|^2 \mid x\right] \leq \sigma_i^2, ~ i=2, \ldots, p \nonumber.        
        \end{gather}
\end{assumption}

\begin{assumption}
    [Unbiased stochastic gradient with bounded variance and  inexact high-order derivatives]
    \label{as:p_ord_stoch_grad_inexact_hess} 
       For any  $x \in \R^d$ stochastic gradient $G_1(x, \xi)$ satisfy~\eqref{eq:p_ord_stoch_grad_def}. For given 
        $x,~y \in \R^d$ inexact high-order derivatives $G_i(x), ~ i=2,\ldots, p$ satisfy
        \begin{gather*}  
                \|(G_i(x) - \nabla^i f({x}))[y - x]^{i-1}\| \le \delta_i^{x, y} \|y - x\|^{i-1}.
        \end{gather*}
\end{assumption}
\vspace{-0.3cm}
To extend Algorithm~\ref{alg:inexact_acc_detailed} to tensor methods, we introduce a $p$-th order model of the function:
\begin{equation*}
\omega^{M, \bar{\delta}}_{x, p}(y) \eqdef \phi_{x, p}(y) + \tfrac{\bar{\delta}}{2}\|x-y\|^2 + \textstyle{\sum \limits_{i=3}^p} \tfrac{\eta_i \delta_i}{i!}\|x-y\|^i + \tfrac{pM}{(p+1)!}\|x-y\|^{p+1},
\end{equation*}
where $\eta_i > 0, ~ 3\leq i\leq p$. 
Next, we modify Definition \ref{def:inexact_sub} for the high order derivatives case
\begin{definition}
    \label{def:inexact_sub_p}
    Denote by $S_p^{M, \bar \delta, \tau} (x)$ a point $S \eqdef S_p^{M, \bar \delta, \tau} (x)$ such that
    $
        \|\nabla \omega_{x,p}^{M, \bar{\delta}}(S)\| \leq \tau.
    $
\end{definition}
\vspace{-0.3cm}
Now, we are prepared to introduce the method and state the convergence theorem.

\begin{algorithm}
  \caption{Accelerated Stochastic Tensor Method}\label{alg:inexact_acc_detailed_p_ord}
  \begin{algorithmic}[1]
      \STATE \textbf{Input:} $y_0 = x_0$ is starting point; constants $M \geq \frac{2}{p}L_p$; $\eta_i \geq 4$, ~ $3 \leq i \leq p$; 
     starting inexactness $\bar{\delta}_0  \geq 0$;  nonnegative nondecreasing sequences $\{\bk^t_i\}_{t \geq 0}$ for $i = 2, \ldots, p+1$, and
      \begin{equation}\label{eq:p_ord_alphas}
          \alpha_t = \tfrac{p+1}{t + p + 1}, ~~~ A_t = \textstyle{\prod \limits_{j=1}^t(1 -\alpha_j)}, ~~~ A_0 = 1.
      \end{equation}
      \vspace{-10pt}
      \begin{equation*}
        \psi_{0}(x):= \tfrac{\bk_2^{0}+\lambda_0}{2}\|x - x_0\|^2+ \textstyle{\sum \limits_{i=3}^p} \tfrac{\bk_i^{0}}{i!}\|x - x_0\|^{i} .
    \end{equation*}
    \useshortskip
    \FOR{$t \geq 0$} 
        \STATE 
                \[v_t = (1 - \alpha_t)x_t + \alpha_t y_t, \quad x_{t+1} = S_p^{M, \bar{\delta}_t, \tau}(v_{t})\]
        \STATE Compute 
            \begin{equation*}
            \begin{aligned}
                y_{t+1}=\arg \min _{x \in \mathbb{R}^{n}}&\left\{\psi_{t+1}(x):= \psi_{t}(x)+ \tfrac{\lambda_{t+1} - \lambda_{t}}{2}\|x - x_0\|^2 \right.\\
                &\left.+ \textstyle{\sum \limits_{i = 2}^{p} }
                \tfrac{\bk^{t+1}_i - \bk^{t}_i}{i!}\|x - x_0\|^i
                +\tfrac{\alpha_{t}}{A_{t}} 
                l(x,x_{t+1}) \right\}.
                \end{aligned}
            \end{equation*}
    \ENDFOR
  \end{algorithmic}
\end{algorithm}

\begin{theorem}\label{thm:acc_convergence_p_ord}
    Let Assumption \ref{as:lip_p} hold and $M \geq \frac{2}{p}L_p$.
    \begin{itemize}[leftmargin=10pt,nolistsep]
        \item 
            Let Assumption \ref{as:p_ord_stoch} hold. After $T \geq 1$ with parameters
            \begin{equation}
            \begin{gathered}
             \label{eq:p_ord_params}
                \bk_2^{t} =  O\ls\tfrac{\bar{\delta}_t\al_t^2}{A_t}\rs, ~  \bk_{i}^{t+1} = O \ls \tfrac{\alpha_{t+1}^{i}\delta_i}{A_{t+1}}\rs, ~
                \bk_{p+1}^{t+1} = O\ls \tfrac{\alpha_{t+1}^{p+1}M}{A_{t+1}}\rs, \\
                \lambda_t = O\ls \tfrac{\sigma_1}{R}t^{p+1/2}\rs,~
                \delta_t = O \ls \sigma_2 + \tfrac{\sigma_1 + \tau}{R}t^{\frac{3}{2}}\rs
            \end{gathered}
            \end{equation}
            we get the following bound 
            \begin{equation*}
            \label{eq:acc_convergence_stoch_p}
            \begin{aligned}
               \E \left[ f(x_{T}) - f(x^{\ast}) \right]  
               &\leq 
               O \ls \tfrac{\tau R}{\sqrt{T}} + \tfrac{\sigma_1 R}{\sqrt{T}} 
               + \textstyle\sum  _{i=2}^p \tfrac{\sigma_i R^{i}}{T^{i}}  
                + \tfrac{MR^{p+1}}{T^{p+1}}\rs.
            \end{aligned}
            \end{equation*}
        \item 
            Let Assumption \ref{as:p_ord_stoch_grad_inexact_hess} hold. After $T \geq 1$ with parameters defined in~\eqref{eq:params} and \\
            $\sigma_i = \delta_i = \max \limits_{t=1, \ldots, T} \delta_{i, t}^{v_{t-1}, x_{t}}$
            we get the following bound 
            \begin{equation*}
            \label{eq:acc_convergence_stoch_grad_inexact_hess_p}
            \begin{aligned}
               \E \left[ f(x_{T}) - f(x^{\ast}) \right]  
               &\leq 
               O \ls \tfrac{\tau R}{\sqrt{T}} + \tfrac{\sigma_1 R}{\sqrt{T}} 
               + \textstyle{\sum}  _{i=2}^p \tfrac{\delta_i R^{i}}{T^{i}}  
                + \tfrac{MR^{p+1}}{T^{p+1}}\rs.
            \end{aligned}
            \end{equation*}
    \end{itemize}    
\end{theorem}

\vspace{-13pt}
\section{Strongly convex case}
\vspace{-8pt}
\label{sec:strong_cvx}
\begin{assumption}\label{as:lip_str_cvx}
    Function $f$ is $\mu$-strongly convex,  $p$ times differentiable on $\mathbb{R}^d$, and its $p$-th derivative is  Lipschitz continuous, i.e. for all $
    x, y \in \mathbb{R}^d$
    $$\|\nabla^p f(x) - \nabla^p f(y)\| \leq L_p \|x - y\|.$$ 
\end{assumption}
\vspace{-0.3cm}
To  exploit the strong convexity of the objective function and attain a linear convergence rate, we introduce a restarted version of Accelerated Stochastic Tensor Method (Algorithm \ref{alg:inexact_acc_detailed_p_ord}). In each iteration of  Restarted Accelerated Stochastic Tensor Method (Algorithm \ref{alg:restarts}), we execute Algorithm \ref{alg:inexact_acc_detailed_p_ord} for a predetermined number of iterations as specified in equation $\eqref{eq:restarts_iter}$. The output of this run is then used as the initial point for the subsequent iteration of Algorithm \ref{alg:inexact_acc_detailed}, which resets the parameters, and this process repeats iteratively.

\begin{algorithm}[t]
    \caption{Restarted Accelerated Stochastic Tensor Method}\label{alg:restarts}
	\textbf{Input}: $z_0  \in \mathbb{R}^d$, strong convexity parameter $\mu > 0$, $M \geq L_p$, and $R_0 > 0$ such that $\|z_0 - x^*\|\leq R_0$. 
	\textbf{For $s = 1, 2, \ldots$:} 
	\begin{enumerate}
	    \item Set $x_0 = z_{s - 1}$, $r_{s-1}= \frac{R_0}{2^{s-1}}$, and $R_{s-1} = \|z_{s-1} - x^*\|$.
	    \item Run Algorithm \ref{alg:inexact_acc_detailed_p_ord} 
	    for $t_s$ iterations, where
	    \begin{equation}\label{eq:restarts_iter}
             t_s = O (1)\textstyle{\max}
             \left\{
             1, 
             \ls\tfrac{\tau }{\mu r_{s-1}}\rs^2,
             \ls \tfrac{\sigma_1}{ \mu  r_{s-1}} \rs^2,
             \textstyle{\max \limits_{i = 2, \ldots, p}}\left(\tfrac{ \delta_i R_{s-1}^{i-2}}{\mu} \right)^\frac{1}{i},
             \left(\tfrac{ L_p R^{p-1}_{s-1}}{\mu}\right)^\frac{1}{p+1}\right\}.
	    \end{equation}
	    \item Set $z_s = x^{t_s}$.
	\end{enumerate}
\end{algorithm}

\begin{theorem}
\label{thm:ACRNM_conv_str_convex}
Let Assumption \ref{as:lip_str_cvx} hold and let parameters of Algorithm \ref{alg:inexact_acc_detailed} be chosen as in \eqref{eq:p_ord_params}. Let $\{z_s\}_{s \geq 0}$ be generated by Algorithm $\ref{alg:restarts}$ and $R > 0$ be such that $\|z_0 - x^*\|\leq R$. Then for any $s \geq 0$ we have 
    \begin{align}
         \E \|z_{s}-x^*\|^2 &\leq 4^{-s} R^2, 
       &
        \E f(z_{s}) - f(x^*) &
             \leq 2^{-2s-1} \mu R^2.
        \label{eq:restart_conv_func}
    \end{align}
Moreover, the total number of iterations to reach desired accuracy $\e:~f(z_s) - f(x^*)\leq \e$ in expectation is
\begin{equation*}\label{eq:ACRNM_complexity_}
    O\left( 
    \tfrac{(\tau + \sigma_1)^2}{\mu \e} 
    +
    \ls\sqrt{\tfrac{\sigma_2}{\mu}} + 1\rs\log\tfrac{f(z_0) - f(x^*)}{\e}
    + 
    \textstyle{\sum} _{i=3}^p \left(\tfrac{\sigma_i R^{i-2}}{\mu}\right)^{\frac{1}{i}}
    +  
    \left(\tfrac{L_p R^{p-1}}{\mu}\right)^\frac{1}{p+1} 
    \right).
\end{equation*}
\end{theorem}

Now, let us make a few observations regarding the results obtained in Theorem \ref{thm:ACRNM_conv_str_convex}. For simplicity let solution of the subproblem be exact and $p=2$, i.e. we do the restarts of the Accelerated Stochastic Cubic Newton, so the total number of iterations is  
\begin{equation} \label{eq:restarted_cubic_convergence}
    O\left( 
    \tfrac{\sigma_1^2}{\mu \e}
    +
    \ls\sqrt{\tfrac{\sigma_2}{\mu}} +1 \rs \log\tfrac{f(z_0) - f(x^*)}{\e} 
    +
    \left(\tfrac{L_2R}{\mu}\right)^\frac{1}{3} 
    \right).
\end{equation}
Next, let's consider solving the stochastic optimization problem $\min \limits_{x \in \R^d} F(x) = \E [f(x, \xi)]$ using the mini-batch Restarted Accelerated Stochastic Cubic Newton method (Algorithm \ref{alg:restarts}) with $p=2$. In this approach, the mini-batched stochastic gradient is computed as $\tfrac{1}{r_1}\textstyle \sum_{i=1}^{r_1} \nabla f(x, \xi_i)$ and the mini-batched stochastic Hessian is computed as $\tfrac{1}{r_2}\textstyle \sum_{i=1}^{r_2} \nabla^2 f(x, \xi_i)$, where $r_1$ and $r_2$ represent the batch sizes for gradients and Hessians, respectively.
\\
From the convergence estimates in  \eqref{eq:restart_conv_func} and \eqref{eq:restarted_cubic_convergence}, we can determine the required sample sizes for computing the batched gradients and batched Hessians. Specifically, we have $r_1 = \tilde{O}\ls \tfrac{\sigma_1^2}{\e \mu^{2/3}}\rs$ and $r_2 = O\ls \tfrac{\sigma_2}{\mu^{1/3}}\rs$. Consequently, the overall number of stochastic gradient computations is $O\ls \tfrac{\sigma_1^2}{\e \mu^{2/3}}\rs$, which is similar to the accelerated SGD method \citep{ghadimi2013optimal}. Interestingly, the number of stochastic Hessian computations scales linearly with the desired accuracy $\e$, i.e., $O\left(\tfrac{\sigma_2}{\mu^{1/3}}\log\frac{1}{\e}\right)$. 
\\
This result highlights the practical importance of second-order methods. Since the batch size of the Hessian is constant, there is no need to adjust it as the desired solution as accuracy increases. This is particularly useful in distributed optimization problems under the assumption of beta similarity~\citep{zhang2015disco}. In methods with such assumption~\citep{zhang2015disco, daneshmand2021newton, agafonov2023inexact}, the server stores a Hessian sample that provides a "good" approximation of the exact Hessian of the objective function. Algorithms utilize this approximation instead of exchanging curvature information with the workers. The constant batch size allows for accurately determining the necessary sample size to achieve fast convergence to any desired accuracy.
\vspace{-13pt}
\section{Experiments}\label{sec:experiments}
\vspace{-8pt}
In this section, we present numerical experiments conducted to demonstrate the efficiency of our proposed methods. We consider logistic regression problems of the form:
\begin{equation*}
    f(x) = \E \left[\log(1+ \exp(-b_\xi \cdot a_{\xi}^{\top}x)) \right],
\end{equation*}
where $(a_{\xi},b_{\xi})$ are the training samples described by features  $a_{\xi}\in \R^d$ and class labels $b_i \in \{-1,1\}$.

\textbf{Setup.} We present results on the \texttt{a9a} dataset ($d=123$) from LibSVM by \cite{chang2011libsvm}. We demonstrate the performance of Accelerated Stochastic Cubic Newton in three regimes: deterministic oracles (Figure \ref{fig:deterministic}), stochastic oracles with the same batch size for gradient and Hessians (Figures~\ref{subfig:stoch_train},~\ref{subfig:stoch_test}), and stochastic oracles with smaller batch size for Hessians (Figures~\ref{subfig:hess_train},~\ref{subfig:hess_test}). The final mode is especially intriguing because the convergence component of Algorithm~\ref{alg:inexact_acc_detailed} associated with gradient noise decreases as $1/\sqrt{t}$, while the component related to Hessian noise decreases as $1/t^2$. This enables the use of smaller Hessian batch sizes (see Corollary~\ref{cor:2ord_inexactenss_acc}).
\vspace{-0.5cm}
\begin{wrapfigure}[14]{r}{0.3\textwidth}
  \begin{center}
\includegraphics[width=0.3\textwidth]{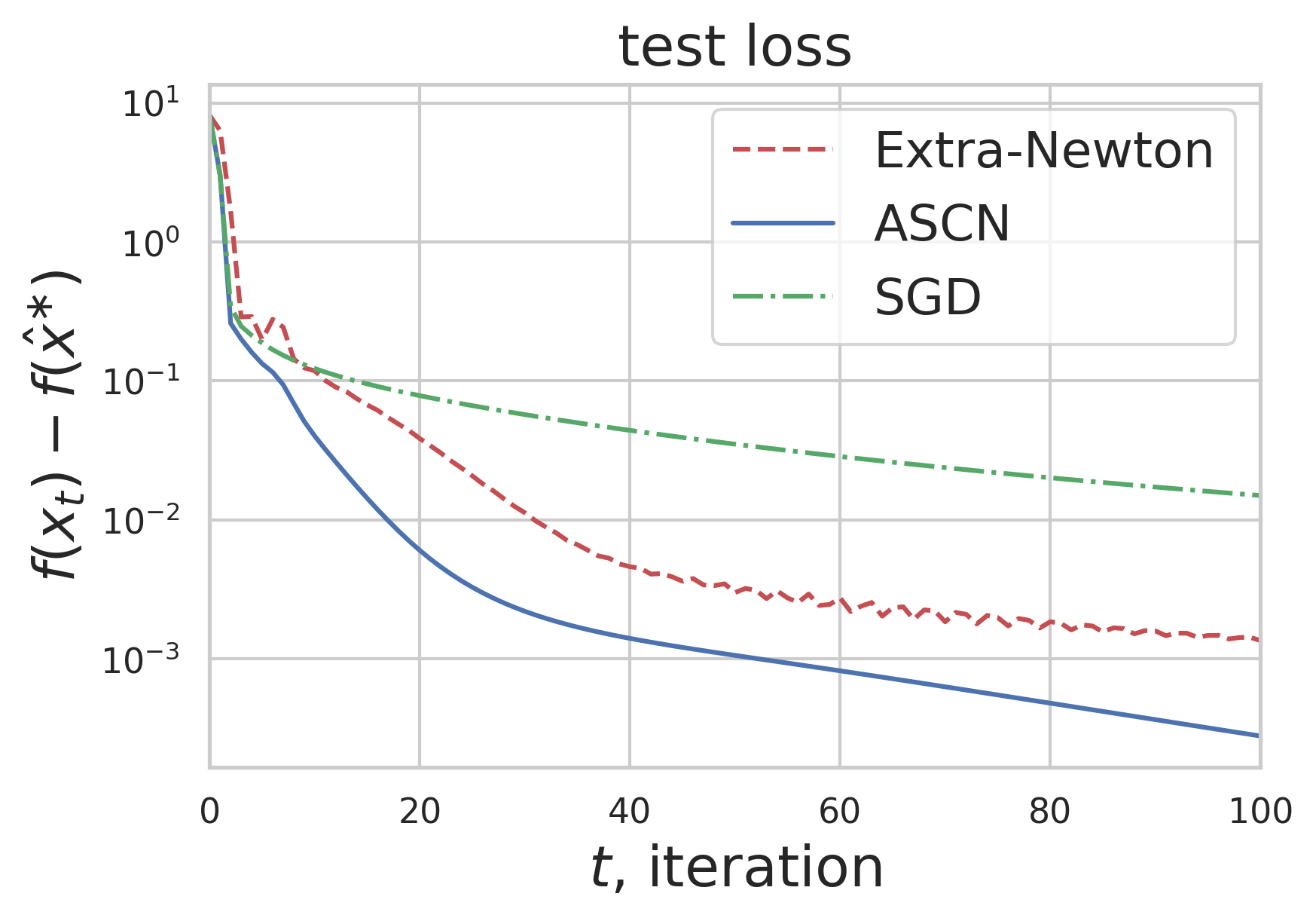}
  \end{center}
  \caption{Logistic regression on \texttt{a9a} with deterministic oracles}
  \label{fig:deterministic}
\end{wrapfigure}

For stochastic experiments, we randomly split the dataset into training ($30000$ data samples) and test ($2561$ data samples) sets. The methods randomly sample data from the training set and do not have access to the test data. In this case, the training loss represents finite sum minimization properties, and the test loss represents expectation minimization.
We compare the performance of the SGD, Extra-Newton (EN), and Accelerated Stochastic Cubic Newton (ASCN).  We present experiments for fine-tuned hyperparameters in Figures~\ref{fig:deterministic},~\ref{fig:stochastic}. For SGD, we've fine-tuned $1$ parameter $lr$. For EN, we've fine-tuned $2$ parameters: $\gamma$ and $\beta_0$. For ASCN, we've fine-tuned $2$ parameters: $M$ and $\tfrac{\sigma_1}{R}$ (only for stochastic case) as the entity, also $\tau = 0$ and $\sigma_2 = 0$ as they are dominated by $\tfrac{\sigma_1}{R}$. To demonstrate the globalization properties of the methods, we consider the starting point $x_0$ far from the solution, specifically $x_0 = 3\cdot e $, where $e$ is the all-one vector. 
All methods are implemented as PyTorch 2.0 optimizers. Additional details and experiments are provided in the Appendix~\ref{app:exp}.

\vspace{-0.3cm}
\begin{figure}[ht]
\centering
     \begin{subfigure}{0.24\textwidth}
         \centering       \includegraphics[width=\textwidth]{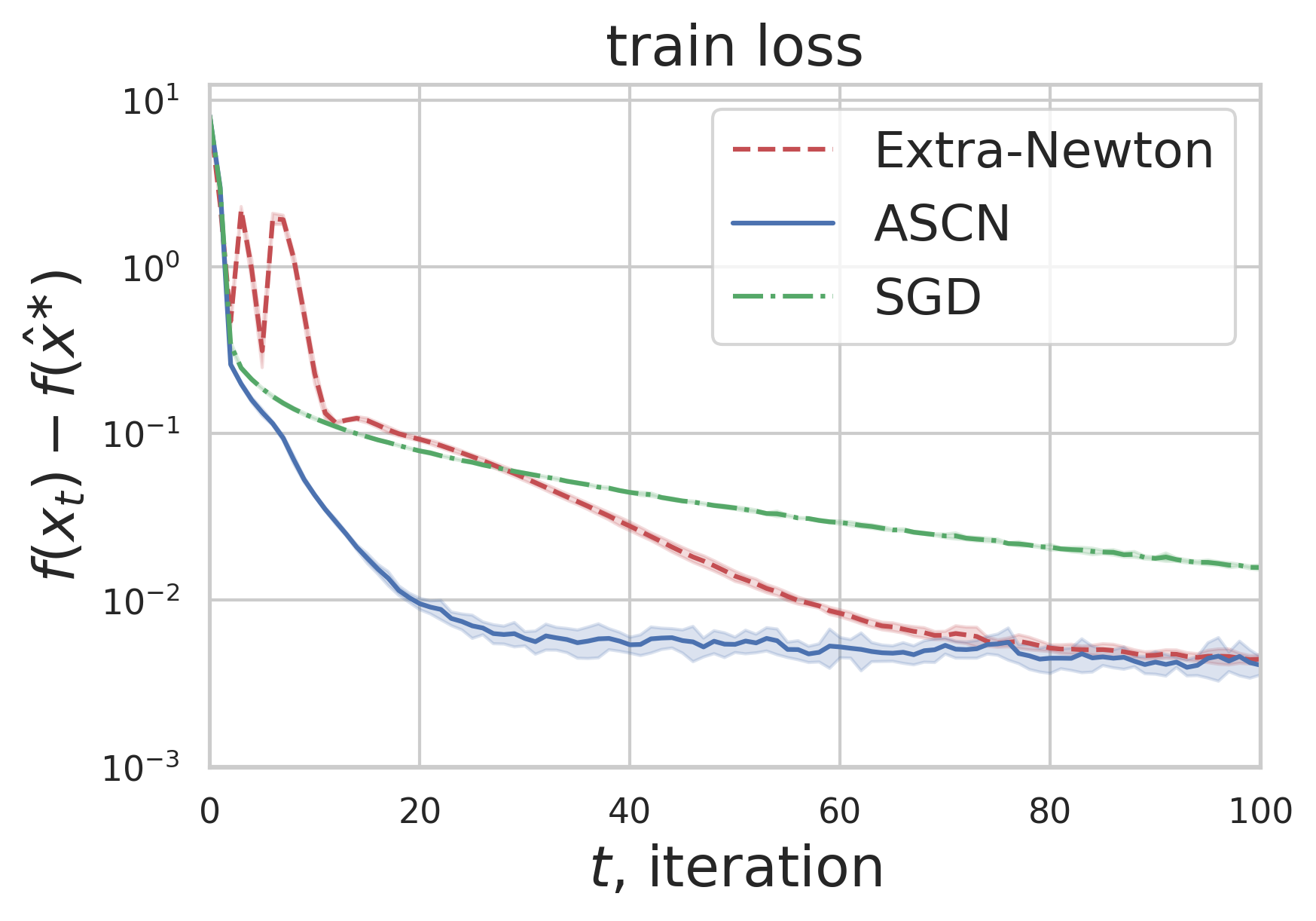}\vspace{-2mm}
         \caption{Train loss. Gradient   and Hessian batch sizes are $1500$}
         \label{subfig:stoch_train}
     \end{subfigure}
     \hfill
     \begin{subfigure}{0.24\textwidth}
         \centering
         \includegraphics[width=\textwidth]{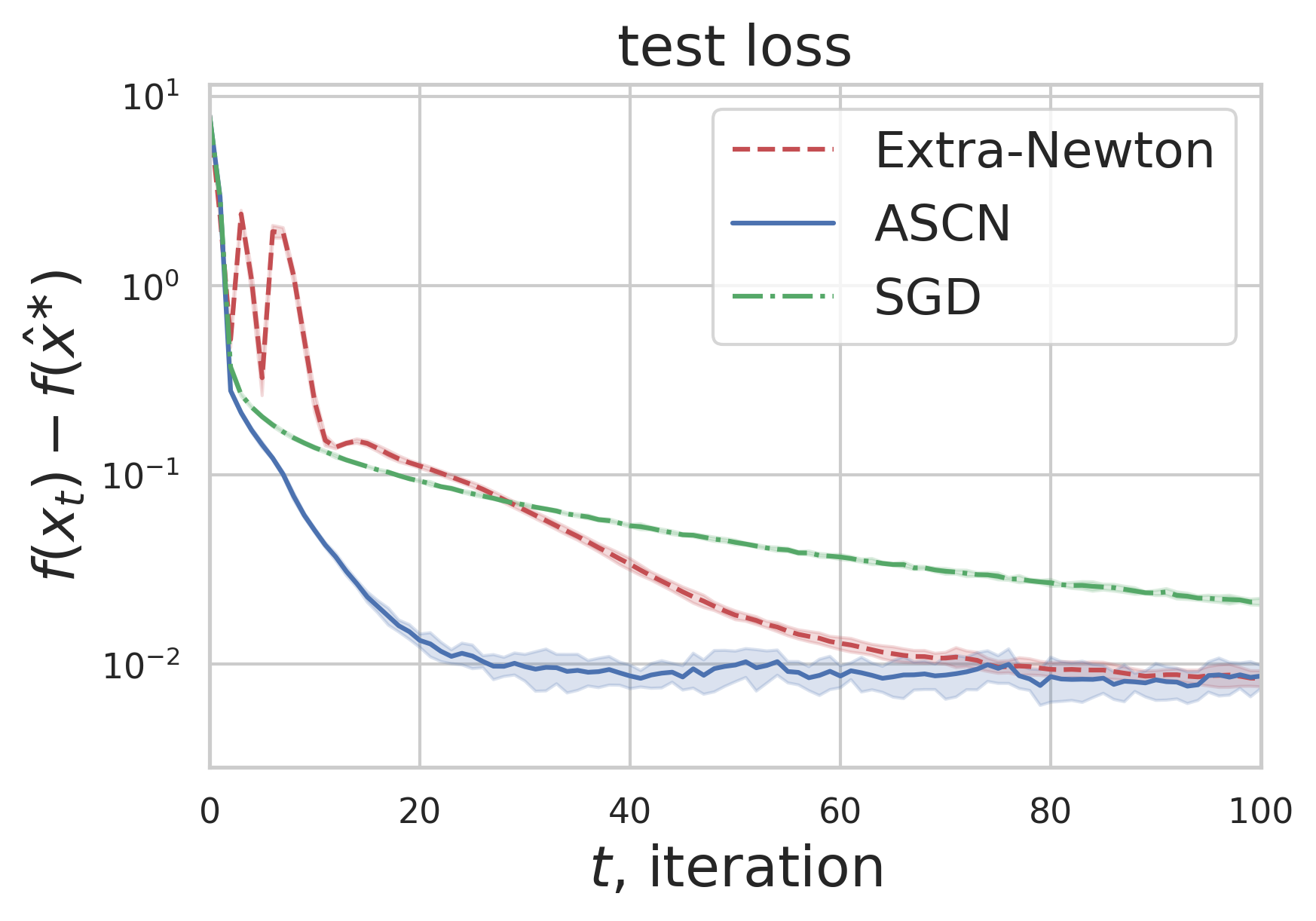}\vspace{-2mm}
         \caption{Test loss. Gradient   and Hessian batch sizes are $1500$}
         \label{subfig:stoch_test}
     \end{subfigure}
    \hfill
     \begin{subfigure}{0.24\textwidth}
         \centering
         \includegraphics[width=\textwidth]{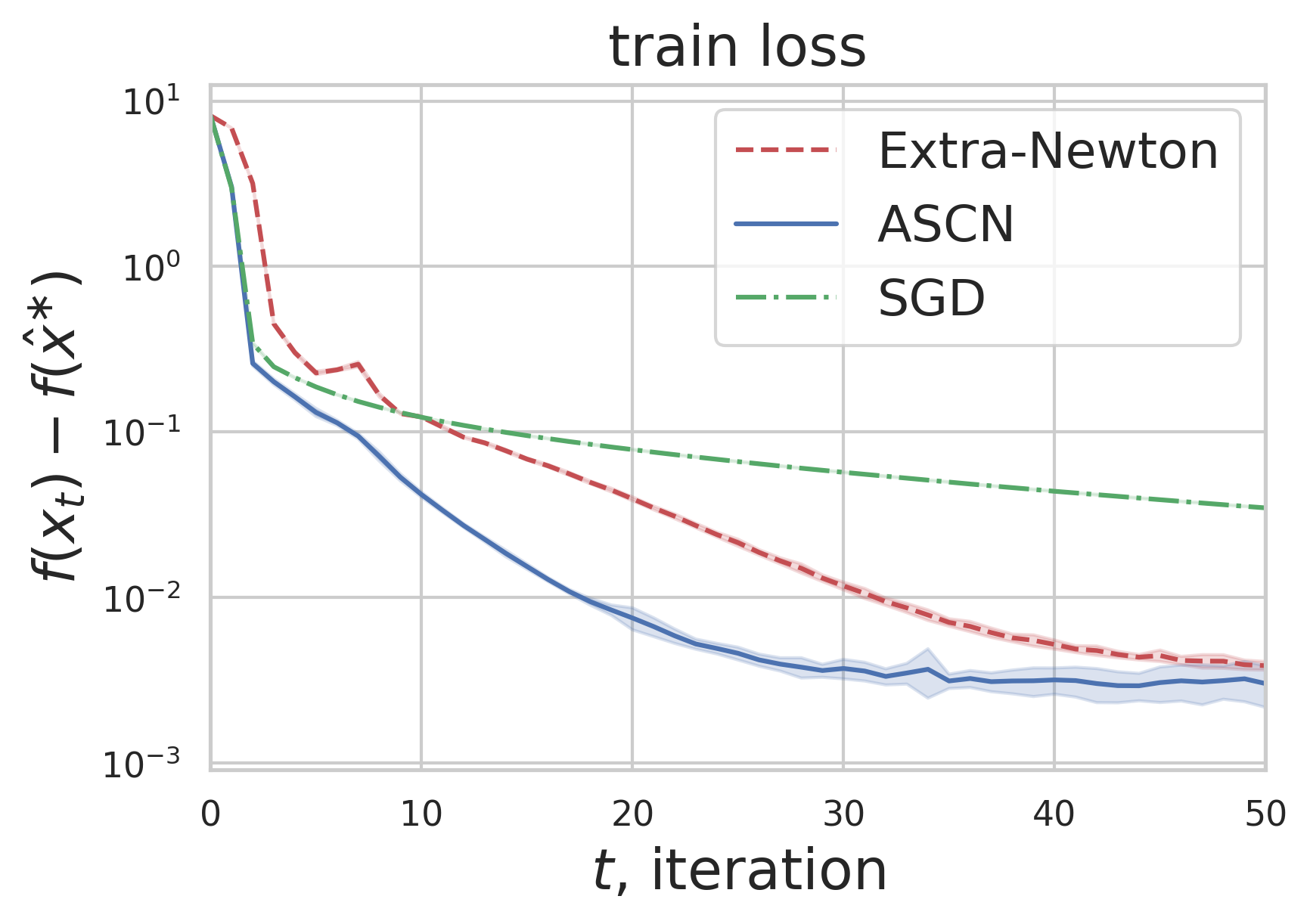}\vspace{-2mm}
         \caption{Train loss. Gradient   batch size is $10000$, Hessian batch size is $150$}
         \label{subfig:hess_train}
     \end{subfigure}
    \hfill
     \begin{subfigure}{0.24\textwidth}
         \centering
         \includegraphics[width=\textwidth]{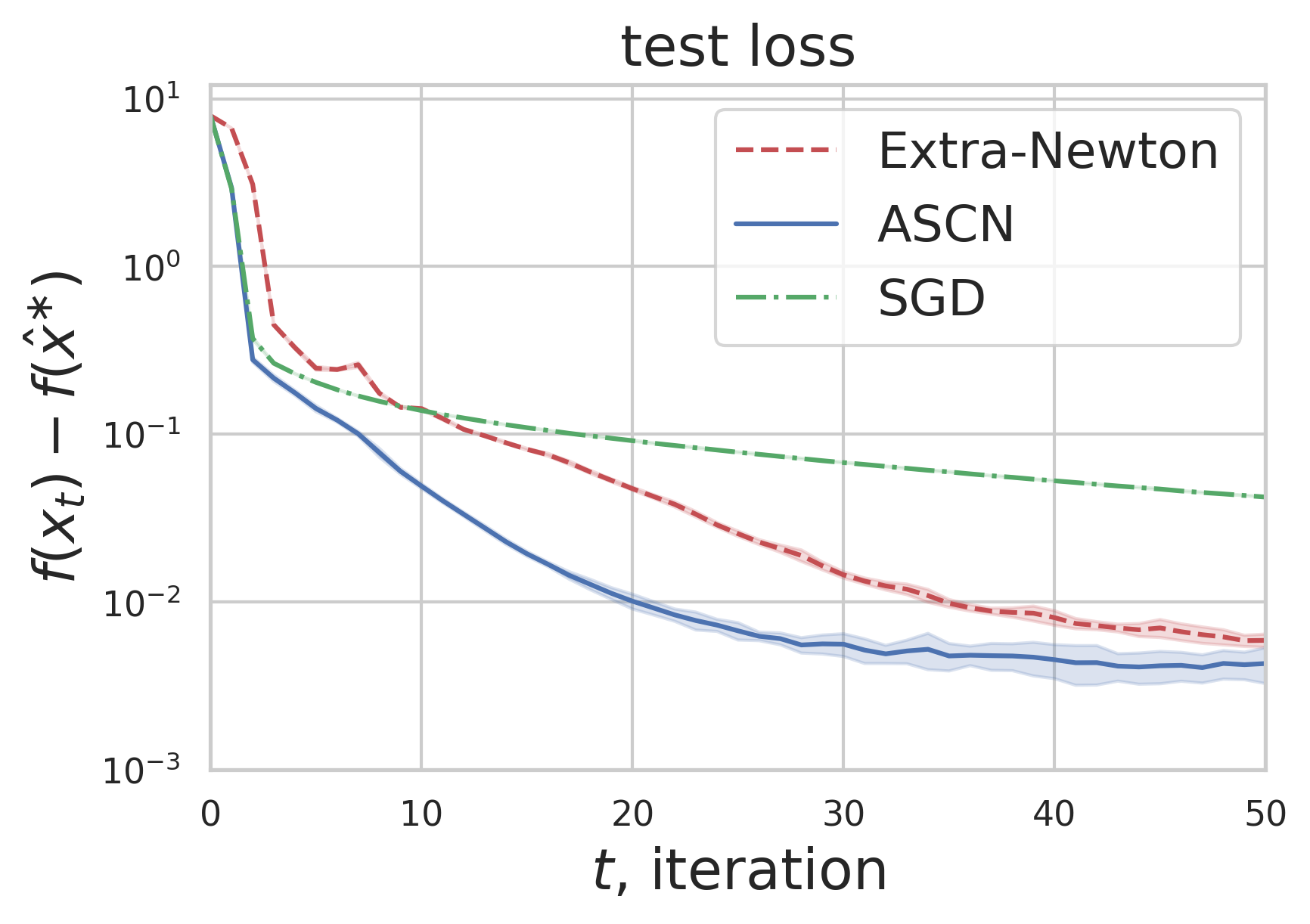}\vspace{-2mm}
         \caption{Test loss. Gradient   batch size is $10000$, Hessian batch size is $150$}
         \label{subfig:hess_test}
     \end{subfigure}
        \caption{Logistic regression on \texttt{a9a} with stochastic oracles}
        \label{fig:stochastic}
\end{figure}
\vspace{-0.3cm}
{\bf Results.} The ASCN method proposed in this study consistently outperforms Extra-Newton and SGD across all experimental scenarios. In deterministic settings, ASCN exhibits a slight superiority over Extra-Newton. In stochastic experiments, we observe a notable improvement as well. However,  it's worth noting that in stochastic regime as we approach convergence, all methods tend to converge to the same point. This convergence pattern is primarily influenced by the stochastic gradient noise $\tfrac{\sigma_1 R}{\sqrt{T}}$ term, which dominates in rates as we converge to solution. Furthermore, experiment with different batch sizes for gradients and Hessians support the theory, confirming that the Hessian inexactness term in ASCN $  \tfrac{\sigma_2 R^2}{T^2}  $ has faster rate than the corresponding term in Extra-Newton  $ \tfrac{\sigma_2 R^2}{T^{3/2}} $.
To conclude, the experiments show that second-order information could significantly accelerate the convergence. Moreover, the methods need significantly less stochastic Hessians than stochastic gradients. 
\vspace{-13pt}
\section{Conclusion}
\vspace{-8pt}
In summary, our contribution includes a novel stochastic accelerated second-order algorithm for convex and strongly convex optimization. We establish a lower bound for stochastic second-order optimization and prove our algorithm's achievement of optimal convergence in both gradient and Hessian inexactness. Additionally, we introduce a tensor generalization of second-order methods for stochastic high-order derivatives. Nevertheless, it's essential to acknowledge certain limitations. Like other globally convergent second-order methods, our algorithm involves a subproblem that necessitates an additional subroutine to find its solution. To mitigate this challenge, we offer theoretical insights into the required accuracy of the subproblem's solution.
Future research could involve enhancing the adaptiveness of the algorithm. Additionally, there is a potential for constructing optimal stochastic second-order and tensor methods by incorporating stochastic elements into existing exact methods. These efforts could further improve both practical and theoretical aspects of stochastic second-order and high-order optimization.

\newpage
\textbf{Acknowledgments.}
This work was supported by a grant for research centers in the field of artificial intelligence, provided by the Analytical Center for the Government of the Russian Federation in accordance with the subsidy agreement (agreement identifier 000000D730321P5Q0002) and the agreement with the Moscow Institute of Physics and Technology dated November 1, 2021 No. 70-2021-00138. This work was supported by Hasler Foundation Program: Hasler Responsible AI (project number 21043). This work was supported by the Swiss National Science Foundation (SNSF) under grant number 200021\_205011. Research was sponsored by the Army Research Office and was accomplished under Grant Number W911NF-24-1-0048.

\textbf{Ethics Statement.} The authors acknowledge that they have read and adhere to the ICLR Code of Ethics.

\textbf{Reproducibility Statement.} The experimental details are provided in Section~\ref{sec:experiments} and Appendix~\ref{app:exp}. The PyTorch code for the methods is available on \url{https://github.com/OPTAMI/OPTAMI}.

\bibliography{bibliography,kamzolov_full,Bibliography-PM,Bibliography-DQV}

\begin{thebibliography}{67}
\providecommand{\natexlab}[1]{#1}
\providecommand{\url}[1]{\texttt{#1}}
\expandafter\ifx\csname urlstyle\endcsname\relax
  \providecommand{\doi}[1]{doi: #1}\else
  \providecommand{\doi}{doi: \begingroup \urlstyle{rm}\Url}\fi

\bibitem[Agafonov et~al.(2021)Agafonov, Dvurechensky, Scutari, Gasnikov, Kamzolov, Lukashevich, and Daneshmand]{agafonov2021accelerated}
Artem Agafonov, Pavel Dvurechensky, Gesualdo Scutari, Alexander Gasnikov, Dmitry Kamzolov, Aleksandr Lukashevich, and Amir Daneshmand.
\newblock An accelerated second-order method for distributed stochastic optimization.
\newblock In \emph{2021 60th IEEE Conference on Decision and Control (CDC)}, pp.\  2407--2413, 2021.
\newblock ISBN 2576-2370.
\newblock \doi{10.1109/CDC45484.2021.9683400}.
\newblock URL \url{https://doi.org/10.1109/CDC45484.2021.9683400}.

\bibitem[Agafonov et~al.(2023)Agafonov, Kamzolov, Dvurechensky, Gasnikov, and Tak{\'a}{\v{c}}]{agafonov2023inexact}
Artem Agafonov, Dmitry Kamzolov, Pavel Dvurechensky, Alexander Gasnikov, and Martin Tak{\'a}{\v{c}}.
\newblock Inexact tensor methods and their application to stochastic convex optimization.
\newblock \emph{Optimization Methods and Software}, 0\penalty0 (0):\penalty0 1--42, 2023.
\newblock \doi{10.1080/10556788.2023.2261604}.
\newblock URL \url{https://doi.org/10.1080/10556788.2023.2261604}.

\bibitem[Agarwal \& Hazan(2018)Agarwal and Hazan]{pmlr-v75-agarwal18a}
Naman Agarwal and Elad Hazan.
\newblock Lower bounds for higher-order convex optimization.
\newblock In Sébastien Bubeck, Vianney Perchet, and Philippe Rigollet (eds.), \emph{Proceedings of the 31st Conference On Learning Theory}, volume~75 of \emph{Proceedings of Machine Learning Research}, pp.\  774--792. PMLR, 06--09 Jul 2018.
\newblock URL \url{https://proceedings.mlr.press/v75/agarwal18a.html}.

\bibitem[Antonakopoulos et~al.(2022)Antonakopoulos, Kavis, and Cevher]{antonakopoulos2022extra}
Kimon Antonakopoulos, Ali Kavis, and Volkan Cevher.
\newblock Extra-newton: A first approach to noise-adaptive accelerated second-order methods.
\newblock In S~Koyejo, S~Mohamed, A~Agarwal, D~Belgrave, K~Cho, and A~Oh (eds.), \emph{Advances in Neural Information Processing Systems}, volume~35, pp.\  29859--29872. Curran Associates, Inc., 2022.
\newblock URL \url{https://proceedings.neurips.cc/paper_files/paper/2022/file/c10804702be5a0cca89331315413f1a2-Paper-Conference.pdf}.

\bibitem[Arjevani et~al.(2019)Arjevani, Shamir, and Shiff]{arjevani2019oracle}
Yossi Arjevani, Ohad Shamir, and Ron Shiff.
\newblock Oracle complexity of second-order methods for smooth convex optimization.
\newblock \emph{Mathematical Programming}, 178:\penalty0 327--360, 2019.
\newblock ISSN 1436-4646.
\newblock \doi{10.1007/s10107-018-1293-1}.
\newblock URL \url{https://doi.org/10.1007/s10107-018-1293-1}.

\bibitem[Baes(2009)]{baes2009estimate}
Michel Baes.
\newblock Estimate sequence methods: extensions and approximations.
\newblock \emph{Institute for Operations Research, ETH, Z{\"u}rich, Switzerland}, 2\penalty0 (1), 2009.

\bibitem[Bellavia et~al.(2022)Bellavia, Gurioli, Morini, and Toint]{bellavia2022adaptive}
S~Bellavia, G~Gurioli, B~Morini, and Ph.L. Toint.
\newblock Adaptive regularization for nonconvex optimization using inexact function values and randomly perturbed derivatives.
\newblock \emph{Journal of Complexity}, 68:\penalty0 101591, 2022.
\newblock ISSN 0885-064X.
\newblock \doi{https://doi.org/10.1016/j.jco.2021.101591}.
\newblock URL \url{https://www.sciencedirect.com/science/article/pii/S0885064X21000467}.

\bibitem[Bellavia \& Gurioli(2022)Bellavia and Gurioli]{bellavia2022stochastic}
Stefania Bellavia and Gianmarco Gurioli.
\newblock Stochastic analysis of an adaptive cubic regularization method under inexact gradient evaluations and dynamic hessian accuracy.
\newblock \emph{Optimization}, 71:\penalty0 227--261, 2022.
\newblock \doi{10.1080/02331934.2021.1892104}.
\newblock URL \url{https://doi.org/10.1080/02331934.2021.1892104}.

\bibitem[Bubeck et~al.(2019)Bubeck, Jiang, Lee, Li, and Sidford]{bubeck2019near}
Sébastien Bubeck, Qijia Jiang, Yin~Tat Lee, Yuanzhi Li, and Aaron Sidford.
\newblock Near-optimal method for highly smooth convex optimization.
\newblock In Alina Beygelzimer and Daniel Hsu (eds.), \emph{Proceedings of the Thirty-Second Conference on Learning Theory}, volume~99, pp.\  492--507. PMLR, 5 2019.
\newblock URL \url{https://proceedings.mlr.press/v99/bubeck19a.html}.

\bibitem[Carmon et~al.(2022)Carmon, Hausler, Jambulapati, Jin, and Sidford]{carmon2022optimal}
Yair Carmon, Danielle Hausler, Arun Jambulapati, Yujia Jin, and Aaron Sidford.
\newblock Optimal and adaptive monteiro-svaiter acceleration.
\newblock In S~Koyejo, S~Mohamed, A~Agarwal, D~Belgrave, K~Cho, and A~Oh (eds.), \emph{Advances in Neural Information Processing Systems}, volume~35, pp.\  20338--20350. Curran Associates, Inc., 2022.
\newblock URL \url{https://proceedings.neurips.cc/paper_files/paper/2022/file/7ff97417474268e6b5a38bcbfae04944-Paper-Conference.pdf}.

\bibitem[Cartis \& Scheinberg(2018)Cartis and Scheinberg]{cartis2018global}
Coralia Cartis and Katya Scheinberg.
\newblock Global convergence rate analysis of unconstrained optimization methods based on probabilistic models.
\newblock \emph{Mathematical Programming}, 169:\penalty0 337--375, 2018.
\newblock ISSN 1436-4646.
\newblock \doi{10.1007/s10107-017-1137-4}.
\newblock URL \url{https://doi.org/10.1007/s10107-017-1137-4}.

\bibitem[Cartis et~al.(2011{\natexlab{a}})Cartis, Gould, and Toint]{cartis2011adaptive}
Coralia Cartis, Nicholas~IM Gould, and Philippe~L Toint.
\newblock Adaptive cubic regularisation methods for unconstrained optimization. part i: motivation, convergence and numerical results.
\newblock \emph{Mathematical Programming}, 127\penalty0 (2):\penalty0 245--295, 2011{\natexlab{a}}.

\bibitem[Cartis et~al.(2011{\natexlab{b}})Cartis, Gould, and Toint]{cartis2011adaptive2}
Coralia Cartis, Nicholas~IM Gould, and Philippe~L Toint.
\newblock Adaptive cubic regularisation methods for unconstrained optimization. part ii: worst-case function-and derivative-evaluation complexity.
\newblock \emph{Mathematical programming}, 130\penalty0 (2):\penalty0 295--319, 2011{\natexlab{b}}.

\bibitem[Cartis et~al.(2017)Cartis, Gould, and Toint]{cartis2017improved}
Coralia Cartis, Nicholas~IM Gould, and Philippe~L Toint.
\newblock Improved second-order evaluation complexity for unconstrained nonlinear optimization using high-order regularized models.
\newblock \emph{arXiv preprint arXiv:1708.04044}, 2017.

\bibitem[Chang \& Lin(2011)Chang and Lin]{chang2011libsvm}
Chih-Chung Chang and Chih-Jen Lin.
\newblock Libsvm: A library for support vector machines.
\newblock \emph{ACM transactions on intelligent systems and technology (TIST)}, 2\penalty0 (3):\penalty0 1--27, 2011.

\bibitem[Daneshmand et~al.(2021)Daneshmand, Scutari, Dvurechensky, and Gasnikov]{daneshmand2021newton}
Amir Daneshmand, Gesualdo Scutari, Pavel Dvurechensky, and Alexander Gasnikov.
\newblock Newton method over networks is fast up to the statistical precision.
\newblock In \emph{International Conference on Machine Learning}, pp.\  2398--2409. PMLR, 2021.

\bibitem[Devolder et~al.(2014)Devolder, Glineur, and Nesterov]{devolder2014first}
Olivier Devolder, Fran{\c{c}}ois Glineur, and Yurii Nesterov.
\newblock First-order methods of smooth convex optimization with inexact oracle.
\newblock \emph{Mathematical Programming}, 146:\penalty0 37--75, 2014.
\newblock ISSN 1436-4646.
\newblock \doi{10.1007/s10107-013-0677-5}.
\newblock URL \url{https://doi.org/10.1007/s10107-013-0677-5}.

\bibitem[Doikov \& Nesterov(2023)Doikov and Nesterov]{doikov2023gradient}
Nikita Doikov and Yurii Nesterov.
\newblock Gradient regularization of {N}ewton method with {B}regman distances.
\newblock \emph{Mathematical Programming}, 2023.
\newblock ISSN 1436-4646.
\newblock \doi{10.1007/s10107-023-01943-7}.
\newblock URL \url{https://doi.org/10.1007/s10107-023-01943-7}.

\bibitem[Doikov \& Nesterov(2020)Doikov and Nesterov]{doikov2020inexact}
Nikita Doikov and Yurii~E Nesterov.
\newblock Inexact tensor methods with dynamic accuracies.
\newblock In \emph{ICML}, pp.\  2577--2586, 2020.

\bibitem[Doikov et~al.(2023)Doikov, Chayti, and Jaggi]{doikov2023second}
Nikita Doikov, El~Mahdi Chayti, and Martin Jaggi.
\newblock Second-order optimization with lazy {H}essians.
\newblock In Andreas Krause, Emma Brunskill, Kyunghyun Cho, Barbara Engelhardt, Sivan Sabato, and Jonathan Scarlett (eds.), \emph{Proceedings of the 40th International Conference on Machine Learning}, volume 202, pp.\  8138--8161. PMLR, 9 2023.
\newblock URL \url{https://proceedings.mlr.press/v202/doikov23a.html}.

\bibitem[Doikov et~al.(2024)Doikov, Mishchenko, and Nesterov]{doikov2024super}
Nikita Doikov, Konstantin Mishchenko, and Yurii Nesterov.
\newblock Super-universal regularized newton method.
\newblock \emph{SIAM Journal on Optimization}, 34:\penalty0 27--56, 2024.
\newblock \doi{10.1137/22M1519444}.
\newblock URL \url{https://doi.org/10.1137/22M1519444}.

\bibitem[Dvurechensky et~al.(2022)Dvurechensky, Kamzolov, Lukashevich, Lee, Ordentlich, Uribe, and Gasnikov]{dvurechensky2022hyperfast}
Pavel Dvurechensky, Dmitry Kamzolov, Aleksandr Lukashevich, Soomin Lee, Erik Ordentlich, César~A Uribe, and Alexander Gasnikov.
\newblock Hyperfast second-order local solvers for efficient statistically preconditioned distributed optimization.
\newblock \emph{EURO Journal on Computational Optimization}, 10:\penalty0 100045, 2022.
\newblock ISSN 2192-4406.
\newblock \doi{https://doi.org/10.1016/j.ejco.2022.100045}.
\newblock URL \url{https://www.sciencedirect.com/science/article/pii/S2192440622000211}.

\bibitem[Gasnikov et~al.(2019{\natexlab{a}})Gasnikov, Dvurechensky, Gorbunov, Vorontsova, Selikhanovych, and Uribe]{gasnikov2019optimal}
Alexander Gasnikov, Pavel Dvurechensky, Eduard Gorbunov, Evgeniya Vorontsova, Daniil Selikhanovych, and César~A Uribe.
\newblock Optimal tensor methods in smooth convex and uniformly convex optimization.
\newblock In Alina Beygelzimer and Daniel Hsu (eds.), \emph{Proceedings of the Thirty-Second Conference on Learning Theory}, volume~99, pp.\  1374--1391. PMLR, 5 2019{\natexlab{a}}.
\newblock URL \url{https://proceedings.mlr.press/v99/gasnikov19a.html}.

\bibitem[Gasnikov et~al.(2019{\natexlab{b}})Gasnikov, Dvurechensky, Gorbunov, Vorontsova, Selikhanovych, Uribe, Jiang, Wang, Zhang, Bubeck, Jiang, Lee, Li, and Sidford]{gasnikov2019near}
Alexander Gasnikov, Pavel Dvurechensky, Eduard Gorbunov, Evgeniya Vorontsova, Daniil Selikhanovych, César~A Uribe, Bo~Jiang, Haoyue Wang, Shuzhong Zhang, Sébastien Bubeck, Qijia Jiang, Yin~Tat Lee, Yuanzhi Li, and Aaron Sidford.
\newblock Near optimal methods for minimizing convex functions with {L}ipschitz p-th derivatives.
\newblock In Alina Beygelzimer and Daniel Hsu (eds.), \emph{Proceedings of the Thirty-Second Conference on Learning Theory}, volume~99, pp.\  1392--1393. PMLR, 5 2019{\natexlab{b}}.
\newblock URL \url{https://proceedings.mlr.press/v99/gasnikov19b.html}.

\bibitem[Ghadimi \& Lan(2013)Ghadimi and Lan]{ghadimi2013optimal}
Saeed Ghadimi and Guanghui Lan.
\newblock Optimal stochastic approximation algorithms for strongly convex stochastic composite optimization, ii: Shrinking procedures and optimal algorithms.
\newblock \emph{SIAM Journal on Optimization}, 23:\penalty0 2061--2089, 2013.
\newblock \doi{10.1137/110848876}.
\newblock URL \url{https://doi.org/10.1137/110848876}.

\bibitem[Ghadimi et~al.(2017)Ghadimi, Liu, and Zhang]{ghadimi2017second}
Saeed Ghadimi, Han Liu, and Tong Zhang.
\newblock Second-order methods with cubic regularization under inexact information.
\newblock \emph{arXiv preprint arXiv:1710.05782}, 2017.

\bibitem[Ghanbari \& Scheinberg(2018)Ghanbari and Scheinberg]{ghanbari2018proximal}
Hiva Ghanbari and Katya Scheinberg.
\newblock Proximal {Q}uasi-{N}ewton methods for regularized convex optimization with linear and accelerated sublinear convergence rates.
\newblock \emph{Computational Optimization and Applications}, 69:\penalty0 597--627, 2018.
\newblock ISSN 1573-2894.
\newblock \doi{10.1007/s10589-017-9964-z}.
\newblock URL \url{https://doi.org/10.1007/s10589-017-9964-z}.

\bibitem[Grapiglia \& Nesterov(2020)Grapiglia and Nesterov]{grapiglia2020tensor}
Geovani~Nunes Grapiglia and Yu~Nesterov.
\newblock Tensor methods for minimizing convex functions with hölder continuous higher-order derivatives.
\newblock \emph{SIAM Journal on Optimization}, 30\penalty0 (4):\penalty0 2750--2779, 2020.

\bibitem[Grapiglia \& Nesterov(2021)Grapiglia and Nesterov]{grapiglia2021inexact}
Geovani~Nunes Grapiglia and Yurii Nesterov.
\newblock On inexact solution of auxiliary problems in tensor methods for convex optimization.
\newblock \emph{Optimization Methods and Software}, 36:\penalty0 145--170, 2021.
\newblock \doi{10.1080/10556788.2020.1731749}.
\newblock URL \url{https://doi.org/10.1080/10556788.2020.1731749}.

\bibitem[Griewank(1981)]{griewank1981modification}
Andreas Griewank.
\newblock The modification of {N}ewton’s method for unconstrained optimization by bounding cubic terms.
\newblock Technical report, Technical report NA/12, 1981.

\bibitem[Hanzely et~al.(2022)Hanzely, Kamzolov, Pasechnyuk, Gasnikov, Richt{\'a}rik, and Tak{\'a}{\v{c}}]{hanzely2022damped}
Slavom\'{\i}r Hanzely, Dmitry Kamzolov, Dmitry Pasechnyuk, Alexander Gasnikov, Peter Richt{\'a}rik, and Martin Tak{\'a}{\v{c}}.
\newblock A damped {N}ewton method achieves global $\mathcal{O}\left (\frac{1}{k^{2}}\right) $ and local quadratic convergence rate.
\newblock In S.~Koyejo, S.~Mohamed, A.~Agarwal, D.~Belgrave, K.~Cho, and A.~Oh (eds.), \emph{Advances in Neural Information Processing Systems}, volume~35, pp.\  25320--25334. Curran Associates, Inc., 2022.
\newblock URL \url{https://proceedings.neurips.cc/paper_files/paper/2022/file/a1f0c0cd6caaa4863af5f12608edf63e-Paper-Conference.pdf}.

\bibitem[Hoffmann \& Kornstaedt(1978)Hoffmann and Kornstaedt]{hoffmann1978higher}
K~H Hoffmann and H~J Kornstaedt.
\newblock Higher-order necessary conditions in abstract mathematical programming.
\newblock \emph{Journal of Optimization Theory and Applications}, 26:\penalty0 533--568, 1978.
\newblock ISSN 1573-2878.
\newblock \doi{10.1007/BF00933151}.
\newblock URL \url{https://doi.org/10.1007/BF00933151}.

\bibitem[Jiang \& Mokhtari(2023)Jiang and Mokhtari]{jiang2023accelerated}
Ruichen Jiang and Aryan Mokhtari.
\newblock Accelerated {Q}uasi-{N}ewton proximal extragradient: Faster rate for smooth convex optimization.
\newblock \emph{arXiv preprint arXiv:2306.02212}, 2023.

\bibitem[Jiang et~al.(2023)Jiang, Jin, and Mokhtari]{jiang2023online}
Ruichen Jiang, Qiujiang Jin, and Aryan Mokhtari.
\newblock Online learning guided curvature approximation: {A} quasi-newton method with global non-asymptotic superlinear convergence.
\newblock \emph{CoRR}, abs/2302.08580, 2023.
\newblock \doi{10.48550/arXiv.2302.08580}.
\newblock URL \url{https://doi.org/10.48550/arXiv.2302.08580}.

\bibitem[Kamzolov(2020)]{kamzolov2020near}
Dmitry Kamzolov.
\newblock Near-optimal hyperfast second-order method for convex optimization.
\newblock In Yury Kochetov, Igor Bykadorov, and Tatiana Gruzdeva (eds.), \emph{Mathematical Optimization Theory and Operations Research}, pp.\  167--178. Springer International Publishing, 2020.
\newblock ISBN 978-3-030-58657-7.

\bibitem[Kamzolov et~al.(2022)Kamzolov, Gasnikov, Dvurechensky, Agafonov, and Tak{\'a}{\v{c}}]{kamzolov2022exploiting}
Dmitry Kamzolov, Alexander Gasnikov, Pavel Dvurechensky, Artem Agafonov, and Martin Tak{\'a}{\v{c}}.
\newblock Exploiting higher-order derivatives in convex optimization methods.
\newblock \emph{arXiv preprint arXiv:2208.13190}, 2022.

\bibitem[Kamzolov et~al.(2023)Kamzolov, Ziu, Agafonov, and Tak{\'a}{\v{c}}]{kamzolov2023accelerated}
Dmitry Kamzolov, Klea Ziu, Artem Agafonov, and Martin Tak{\'a}{\v{c}}.
\newblock Accelerated adaptive cubic regularized {Q}uasi-{N}ewton methods.
\newblock \emph{arXiv preprint arXiv:2302.04987}, 2023.

\bibitem[Kantorovich(1949)]{kantorovich1949newton}
Leonid~Vitalyevich Kantorovich.
\newblock On {N}ewton's method.
\newblock \emph{Trudy Matematicheskogo Instituta imeni VA Steklova}, 28:\penalty0 104--144, 1949.
\newblock (In Russian).

\bibitem[Kohler \& Lucchi(2017)Kohler and Lucchi]{kohler2017sub}
Jonas~Moritz Kohler and Aurelien Lucchi.
\newblock Sub-sampled cubic regularization for non-convex optimization.
\newblock In Doina Precup and Yee~Whye Teh (eds.), \emph{Proceedings of the 34th International Conference on Machine Learning}, volume~70, pp.\  1895--1904. PMLR, 5 2017.
\newblock URL \url{https://proceedings.mlr.press/v70/kohler17a.html}.

\bibitem[Kovalev \& Gasnikov(2022)Kovalev and Gasnikov]{kovalev2022first}
Dmitry Kovalev and Alexander Gasnikov.
\newblock The first optimal acceleration of high-order methods in smooth convex optimization.
\newblock In S~Koyejo, S~Mohamed, A~Agarwal, D~Belgrave, K~Cho, and A~Oh (eds.), \emph{Advances in Neural Information Processing Systems}, volume~35, pp.\  35339--35351. Curran Associates, Inc., 2022.
\newblock URL \url{https://proceedings.neurips.cc/paper_files/paper/2022/file/e56f394bbd4f0ec81393d767caa5a31b-Paper-Conference.pdf}.

\bibitem[Lan(2012)]{lan2012optimal}
Guanghui Lan.
\newblock An optimal method for stochastic composite optimization.
\newblock \emph{Mathematical Programming}, 133:\penalty0 365--397, 2012.
\newblock ISSN 1436-4646.
\newblock \doi{10.1007/s10107-010-0434-y}.
\newblock URL \url{https://doi.org/10.1007/s10107-010-0434-y}.

\bibitem[Lucchi \& Kohler(2023)Lucchi and Kohler]{lucchi2022sub}
Aurelien Lucchi and Jonas Kohler.
\newblock A sub-sampled tensor method for nonconvex optimization.
\newblock \emph{IMA Journal of Numerical Analysis}, 43:\penalty0 2856--2891, 10 2023.
\newblock ISSN 0272-4979.
\newblock \doi{10.1093/imanum/drac057}.
\newblock URL \url{https://doi.org/10.1093/imanum/drac057}.

\bibitem[Mishchenko(2023)]{mishchenko2023regularized}
Konstantin Mishchenko.
\newblock Regularized {N}ewton method with global $\mathcal{O}\left (\frac{1}{k^{2}}\right)$ convergence.
\newblock \emph{SIAM Journal on Optimization}, 33:\penalty0 1440--1462, 2023.
\newblock \doi{10.1137/22M1488752}.
\newblock URL \url{https://doi.org/10.1137/22M1488752}.

\bibitem[Monteiro \& Svaiter(2013)Monteiro and Svaiter]{monteiro2013accelerated}
Renato D~C Monteiro and B~F Svaiter.
\newblock An accelerated hybrid proximal extragradient method for convex optimization and its implications to second-order methods.
\newblock \emph{SIAM Journal on Optimization}, 23:\penalty0 1092--1125, 2013.
\newblock \doi{10.1137/110833786}.
\newblock URL \url{https://doi.org/10.1137/110833786}.

\bibitem[Mor{\'e}(1977)]{more1977levenberg}
Jorge~J. Mor{\'e}.
\newblock The levenberg--marquardt algorithm: implementation and theory.
\newblock In \emph{Conference on Numerical Analysis}, University of Dundee, Scotland, 7 1977.
\newblock URL \url{https://www.osti.gov/biblio/7256021}.

\bibitem[Nemirovski et~al.(2009)Nemirovski, Juditsky, Lan, and Shapiro]{nemirovski2009robust}
A~Nemirovski, A~Juditsky, G~Lan, and A~Shapiro.
\newblock Robust stochastic approximation approach to stochastic programming.
\newblock \emph{SIAM Journal on Optimization}, 19:\penalty0 1574--1609, 2009.
\newblock \doi{10.1137/070704277}.
\newblock URL \url{https://doi.org/10.1137/070704277}.

\bibitem[Nemirovski \& Yudin(1983)Nemirovski and Yudin]{nemirovski1983problem}
Arkadi~Semenovich Nemirovski and David~Borisovich Yudin.
\newblock \emph{Problem Complexity and Method Efficiency in Optimization}.
\newblock A Wiley-Interscience publication. Wiley, 1983.

\bibitem[Nesterov(2008)]{nesterov2008accelerating}
Yurii Nesterov.
\newblock Accelerating the cubic regularization of {N}ewton’s method on convex problems.
\newblock \emph{Mathematical Programming}, 112:\penalty0 159--181, 2008.
\newblock ISSN 1436-4646.
\newblock \doi{10.1007/s10107-006-0089-x}.
\newblock URL \url{https://doi.org/10.1007/s10107-006-0089-x}.

\bibitem[Nesterov(2018)]{nesterov2018lectures}
Yurii Nesterov.
\newblock \emph{Lectures on Convex Optimization}.
\newblock Springer Cham, 2 edition, 2018.
\newblock ISBN 978-3-319-91577-7.
\newblock \doi{10.1007/978-3-319-91578-4}.

\bibitem[Nesterov(2021{\natexlab{a}})]{nesterov2021auxiliary}
Yurii Nesterov.
\newblock Inexact high-order proximal-point methods with auxiliary search procedure.
\newblock \emph{SIAM Journal on Optimization}, 31:\penalty0 2807--2828, 2021{\natexlab{a}}.
\newblock \doi{10.1137/20M134705X}.
\newblock URL \url{https://doi.org/10.1137/20M134705X}.

\bibitem[Nesterov(2021{\natexlab{b}})]{nesterov2021implementable}
Yurii Nesterov.
\newblock Implementable tensor methods in unconstrained convex optimization.
\newblock \emph{Mathematical Programming}, 186:\penalty0 157--183, 2021{\natexlab{b}}.
\newblock ISSN 1436-4646.
\newblock \doi{10.1007/s10107-019-01449-1}.
\newblock URL \url{https://doi.org/10.1007/s10107-019-01449-1}.

\bibitem[Nesterov(2021{\natexlab{c}})]{nesterov2021superfast}
Yurii Nesterov.
\newblock Superfast second-order methods for unconstrained convex optimization.
\newblock \emph{Journal of Optimization Theory and Applications}, 191:\penalty0 1--30, 2021{\natexlab{c}}.
\newblock ISSN 1573-2878.
\newblock \doi{10.1007/s10957-021-01930-y}.
\newblock URL \url{https://doi.org/10.1007/s10957-021-01930-y}.

\bibitem[Nesterov \& Polyak(2006)Nesterov and Polyak]{nesterov2006cubic}
Yurii Nesterov and Boris~T Polyak.
\newblock Cubic regularization of {N}ewton method and its global performance.
\newblock \emph{Mathematical Programming}, 108:\penalty0 177--205, 2006.
\newblock \doi{10.1007/s10107-006-0706-8}.
\newblock URL \url{https://doi.org/10.1007/s10107-006-0706-8}.

\bibitem[Newton(1687)]{newton1687philosophiae}
Isaac Newton.
\newblock \emph{Philosophiae naturalis principia mathematica}.
\newblock Edmond Halley, 1687.

\bibitem[Polyak \& Juditsky(1992)Polyak and Juditsky]{polyak1992acceleration}
Boris~T Polyak and Anatoli~B Juditsky.
\newblock Acceleration of stochastic approximation by averaging.
\newblock \emph{SIAM Journal on Control and Optimization}, 30:\penalty0 838--855, 1992.
\newblock \doi{10.1137/0330046}.
\newblock URL \url{https://doi.org/10.1137/0330046}.

\bibitem[Polyak(1990)]{polyak1990new}
Boris~Teodorovich Polyak.
\newblock A new method of stochastic approximation type.
\newblock \emph{Avtomatika i Telemekhanika}, 51:\penalty0 98--107, 1990.

\bibitem[Polyak(2017)]{polyak2017complexity}
Roman Polyak.
\newblock Complexity of the regularized {N}ewton method.
\newblock \emph{arXiv preprint arXiv:1706.08483}, 2017.

\bibitem[Polyak(2009)]{polyak2009regularized}
Roman~A Polyak.
\newblock Regularized {N}ewton method for unconstrained convex optimization.
\newblock \emph{Mathematical Programming}, 120:\penalty0 125--145, 2009.
\newblock ISSN 1436-4646.
\newblock \doi{10.1007/s10107-007-0143-3}.
\newblock URL \url{https://doi.org/10.1007/s10107-007-0143-3}.

\bibitem[Raphson(1697)]{raphson1697analysis}
Joseph Raphson.
\newblock \emph{Analysis Aequationum Universalis Seu Ad Aequationes Algebraicas Resolvendas Methodus Generalis \& Expedita, Ex Nova Infinitarum Serierum Methodo, Deducta Ac Demonstrata}.
\newblock Th. Braddyll, 1697.

\bibitem[Robbins \& Monro(1951)Robbins and Monro]{robbins1951stochastic}
Herbert Robbins and Sutton Monro.
\newblock A stochastic approximation method.
\newblock \emph{The Annals of Mathematical Statistics}, 22:\penalty0 400--407, 1951.
\newblock ISSN 0003-4851.
\newblock URL \url{http://www.jstor.org/stable/2236626}.

\bibitem[Scheinberg \& Tang(2016)Scheinberg and Tang]{scheinberg2016practical}
Katya Scheinberg and Xiaocheng Tang.
\newblock Practical inexact proximal {Q}uasi-{N}ewton method with global complexity analysis.
\newblock \emph{Mathematical Programming}, 160:\penalty0 495--529, 2016.
\newblock ISSN 1436-4646.
\newblock \doi{10.1007/s10107-016-0997-3}.
\newblock URL \url{https://doi.org/10.1007/s10107-016-0997-3}.

\bibitem[Scieur(2023)]{scieur2023adaptive}
Damien Scieur.
\newblock Adaptive {Q}uasi-{N}ewton and anderson acceleration framework with explicit global (accelerated) convergence rates.
\newblock \emph{arXiv preprint arXiv:2305.19179}, 2023.

\bibitem[Simpson(1740)]{simpson1740essays}
Thomas Simpson.
\newblock \emph{Essays on several curious and useful subjects, in speculative and mix'd mathematicks. Illustrated by a variety of examples.}
\newblock H. Woodfall, 1740.

\bibitem[Tripuraneni et~al.(2018)Tripuraneni, Stern, Jin, Regier, and Jordan]{tripuraneni2018stochastic}
Nilesh Tripuraneni, Mitchell Stern, Chi Jin, Jeffrey Regier, and Michael~I Jordan.
\newblock Stochastic cubic regularization for fast nonconvex optimization.
\newblock In S~Bengio, H~Wallach, H~Larochelle, K~Grauman, N~Cesa-Bianchi, and R~Garnett (eds.), \emph{Advances in Neural Information Processing Systems}, volume~31. Curran Associates, Inc., 2018.
\newblock URL \url{https://proceedings.neurips.cc/paper_files/paper/2018/file/db1915052d15f7815c8b88e879465a1e-Paper.pdf}.

\bibitem[Woodworth \& Srebro(2017)Woodworth and Srebro]{woodworth2017lower}
Blake Woodworth and Nathan Srebro.
\newblock Lower bound for randomized first order convex optimization.
\newblock \emph{arXiv preprint arXiv:1709.03594}, 2017.

\bibitem[Xu et~al.(2020)Xu, Roosta, and Mahoney]{xu2020newton}
Peng Xu, Fred Roosta, and Michael~W Mahoney.
\newblock Newton-type methods for non-convex optimization under inexact {H}essian information.
\newblock \emph{Mathematical Programming}, 184:\penalty0 35--70, 2020.
\newblock ISSN 1436-4646.
\newblock \doi{10.1007/s10107-019-01405-z}.
\newblock URL \url{https://doi.org/10.1007/s10107-019-01405-z}.

\bibitem[Zhang \& Lin(2015)Zhang and Lin]{zhang2015disco}
Yuchen Zhang and Xiao Lin.
\newblock Disco: Distributed optimization for self-concordant empirical loss.
\newblock In Francis Bach and David Blei (eds.), \emph{Proceedings of the 32nd International Conference on Machine Learning}, volume~37 of \emph{Proceedings of Machine Learning Research}, pp.\  362--370, Lille, France, 07--09 Jul 2015. PMLR.
\newblock URL \url{https://proceedings.mlr.press/v37/zhangb15.html}.

\end{thebibliography}

\newpage
\appendix
\onecolumn
\section{Related work}
\label{app:rel_works}
The idea of using high-order derivatives in optimization has been known for a long time \cite{hoffmann1978higher}. In 2009, M.~Baes extended the cubic regularization approach with second-order derivatives ($p=2$)  from \cite{nesterov2006cubic} to high-order derivatives ($p>2$) in \cite{baes2009estimate}. However, the subproblem of these methods was non-convex, making them impractical. In 2018, Yu.~Nesterov proposed the implementable (Accelerated) Tensor Method \cite{nesterov2021implementable}, wherein the convexity of the subproblem was reached by increasing a regularization parameter. Hence, the convex subproblem could be efficiently solved by appropriate subsolvers, making the algorithm practically applicable. In the same work, a lower complexity bound for tensor methods under higher-order smoothness assumption was proposed. Shortly after, near-optimal~\cite{gasnikov2019near,gasnikov2019optimal,bubeck2019near} and optimal~\cite{kovalev2022first,carmon2022optimal} high-order methods were introduced. Furthermore, under higher smoothness assumptions, \emph{second-order} methods~\cite{nesterov2021superfast, nesterov2021auxiliary, kamzolov2020near,doikov2024super} can surpass the corresponding lower complexity bound for functions with Lipschitz-continuous Hessians. For more comprehensive information on high-order methods, one can refer to the review~\cite{kamzolov2022exploiting}.\\[2pt]
In general, second and higher-order methods are known for their faster convergence compared to first-order methods. However, their computational cost per iteration is significantly higher due to the computation of high-order derivatives. To alleviate this computational burden, it is common to employ approximations of derivatives instead of exact values. While there is a wide range of second-order and tensor methods available for the non-convex case, assuming stochastic or inexact derivatives~\cite{cartis2011adaptive,cartis2011adaptive2, cartis2018global, kohler2017sub, xu2020newton, tripuraneni2018stochastic,lucchi2022sub,bellavia2022stochastic, bellavia2022adaptive,doikov2023second}, the same cannot be said for the convex case. In the context of convex problems, there have been studies on high-order methods such as second-order methods with inexact Hessian information~\cite{ghadimi2017second}, tensor methods with inexact and stochastic derivatives~\cite{agafonov2023inexact}, and Extra-Newton algorithm with stochastic gradients and Hessians~\cite{antonakopoulos2022extra}. As a possible application of methods with inexact Hessians, we highlight Quasi-Newton(QN) methods. Such methods approximate second-order derivatives using the history of gradient feedback. Quasi-Newton methods are known for their impressive practical performance and local superlinear convergence. However, for the long period of time, the main drawback of such methods was a slow theoretical global convergence, slower than gradient descent. First steps to improve the global convergence of such methods were done in ~\citep{scheinberg2016practical,ghanbari2018proximal} but the methods could be still slower than gradient descent. The first global Quasi-Newton methods that provably matches the gradient descent were reached by cubic regularization in two consecutive papers ~\cite{kamzolov2023accelerated,jiang2023online}. It also opened a possibility for accelerated QN \cite{kamzolov2023accelerated} that theoretically matches  fast gradient method and first-order lower-bounds. This direction were further explored for different methods in \cite{scieur2023adaptive,jiang2023accelerated}.
 Another possible application for high-order methods with inexact or stochastic derivatives is distributed optimization~\cite{zhang2015disco,daneshmand2021newton, agafonov2021accelerated, dvurechensky2022hyperfast}.\\[2pt]
One of the main challenges of tensor and regularized second-order methods is solving the auxiliary subproblem to compute the iterate update. In both second-order and higher-order cases, it usually requires running a subsolver algorithm. The impact of the accuracy, up to which we solve the auxiliary problem, on the convergence of the algorithm has been studied in several works~\cite{grapiglia2021inexact,grapiglia2020tensor,doikov2020inexact}. One actively developing direction relies on the constructions of CRN with explicit step~\cite{polyak2009regularized, polyak2017complexity, mishchenko2023regularized, doikov2023gradient,doikov2024super,hanzely2022damped}. 

\section{\aaa{On the intuition behind the algorithm}}
\label{app:intuition}

\textbf{Model.} For the second-order case the model $\omega_{x, M}^{\bar{\delta}}(y)$ comprises three key components: an inexact Taylor approximation $\phi_x(y)$; cubic regularization $\frac{M}{6}\|x-y\|^3$ and additional quadratic regularization $\frac{\bar{\delta}}{2}\|x-y\|^2$.
The combination of Taylor polynomial and cubic regularization is the standard model for exact second-order methods, as they create a model that is both convex and upper bounds the objective \citep{nesterov2008accelerating} (see \citep{nesterov2021implementable} for high-order optimization).
However, inserting inexactness to the Taylor approximation leads to the necessity of additional regularization \citep{agafonov2023inexact}. 

The first reason to add quadratic regularization is to ensure that the Hessian of the function is majorized by the Hessian of the model:
$$0 \preceq \nabla^2 f(y) \preceq \nabla^2 \phi_x(y) + \delta_2 I + L_2||y - x|| \preceq \nabla^2 \phi_x(y) + \bar{\delta}_2 I + M||y - x|| = \nabla^2 \omega_{x}^{\bar{\delta}}.$$
Moreover, this regularization is essential for handling stochastic gradients correctly. Note, that we add quadratic regularizer with the constant $\bar{\delta}_t = 2\delta_2 + \frac{\sigma_1}{R}(t+1)^{3/2}$. Here, $\delta_2$ accounts for a Hessian majorization, while $\frac{\sigma_1}{R}(t+1)^{3/2}$ is crucial for achieving optimal convergence in gradient inexactness. From our perspective, this regularization can be viewed as a damping for the size of stochastic Cubic Newton step, as stochastic gradients may lead to undesirable directions.  

For further clarification, please refer to Lemma~\ref{lm:scalar_lb_cases}. This lemma serves as a bound on the progress of the step. Take a look at the right-hand side of equation \eqref{eq:scalar_lb_cases}. Without proper quadratic regularization, we won't capture the correct term related to Hessian inexactness $\delta_2$. Consequently, the desired convergence term, $\frac{\delta_2 R^2}{T^2}$, cannot be achieved. Moving on to the left-hand side of \eqref{eq:scalar_lb_cases}, we encounter the term $\frac{2}{\bar{\delta}}||g(x) - \nabla f(x)||$. Here, choosing the appropriate $\bar{\delta}$ is crucial to compensate for stochastic gradient errors and achieve optimal convergence in the gradient inexactness term.

\textbf{Estimating sequences.} Estimating sequences are a standard optimization technique to achieve acceleration \citep{nesterov2018lectures}[Section 2.2.1]. As far as our knowledge extends, the application of estimating sequences to second-order methods was first introduced in \citep{nesterov2008accelerating}. The concept involves adapting acceleration techniques traditionally applied to first-order methods to the realm of second-order methods. In this work, we make slight modifications to the estimating sequences derived from \citep{nesterov2008accelerating} to preserve the customary relationships inherent in accelerated methods:
\begin{gather*}
    \frac{f(x_t)}{A_{t-1}} - err_{low} \leq \psi_t^* \leq \frac{f(x^*)}{A_{t-1}} + c_2 \|x^* - x_0\|^2 + c_3 \|x^* - x_0\|^3 + err_{up},
\end{gather*}
where $A_{t}$ and $\alpha_t$ are scaling factors, common for acceleration \citep{nesterov2018lectures, nesterov2008accelerating}.

For simplicity, let $err_{low} = 0, ~ err_{up} = 0, ~ c_2 = 0$. That is the case for exact derivatives and subproblem solutions. Then, one can get the convergence, with a specific choice of scaling factors:
$$f(x_t) - f(x^*) \leq A_{t-1} c_3 \|x^* - x_0\|^3.$$
In our case, errors and $c_2$ are non zero and stay for gradient, Hessian and subproblem inexactness. By applying estimating sequence technique we get rates \eqref{eq:acc_convergence_stoch},~\eqref{eq:acc_convergence_stoch_grad_inexact_hess}.

\textbf{The choice of parameters.} The cubic regularization parameter $M \geq 2 L_2$ represents the standard choice for second-order methods \citep{nesterov2006cubic}.
The quadratic regularization parameter  $\bar{\delta}_t = 2\delta_2 + \frac{\sigma_1 + \tau}{R}(t+1)^{3/2}$ consists of  $2 \delta_2$ for compensating Hessian errors, and $\frac{\sigma_1 + \tau}{R}(t+1)^{3/2}$ for compensating stochastic gradient and subproblem solution errors.
The regularization parameters $\bar{\kappa}_2^t, \bar{\kappa}_3^t, \lambda_t$ are utilized for the second step of the method. The $\bk$’s are chosen in~\eqref{eq:bar_kappa},~\eqref{eq:bar_kappa_3} to uphold the inequality~\eqref{eq:lemma_ass} for acceleration. $\lambda_t = \frac{\sigma}{R}(t+1)^{5/2} $ serves to compensate for stochastic gradient errors in the estimation functions  $\psi_t$. The specific choice for $\bar{\delta}$ and $\lambda$ is made in the proof of Theorem~\ref{thm:2_ord_app} to achieve optimal convergence rate.

\section{Additional Experiments}
\label{app:exp}
After tuning we got the following hyperparameters. For deterministic oracles: $lr=20$ for GD, $\gamma=5$ and $\beta_0 = 0.5$ for Extra-Newton, $M=0.01$  for ASCN.  For stochastic oracles: $lr=20$ for SGD, $\gamma=5$ and $\beta_0 = 0.05$ for Extra-Newton, $M=0.01$ and $\sigma=1e-7$ for ASCN. For stochastic oracles with different batch sizes: $lr=20$ for SGD, $\gamma=5$ and $\beta_0 = 1.0$ for Extra-Newton, $M=0.01$ and $\sigma=1e-7$ for ASCN.

On Figures~\ref{fig:app_stoch_450},~\ref{fig:app_stoch_900} we present additional experiments with different batch sizes for gradients and Hessians, and \aaa{on Figures~\ref{fig:app_stochastic_increased},~\ref{fig:app_stochastic_increased_150} we present Figure~\ref{fig:stochastic} with increased size. Moving forward to Figures~\ref{fig:app_stochastic_150_time} and~\ref{fig:app_stochastic_450_time}, a comparison in running time is illustrated for two distinct setups: the gradient batch size is set at $10000$, and Hessian batch sizes are configured to be $150$ and $450$. Specifically, one iteration of SGD consumes $0.16$ seconds, while the execution times for EN and ASCN are approximately $0.33$ seconds for both scenarios. }

\begin{figure}[h]
\centering
     \begin{subfigure}{0.48\textwidth}
         \centering         \includegraphics[width=\textwidth]{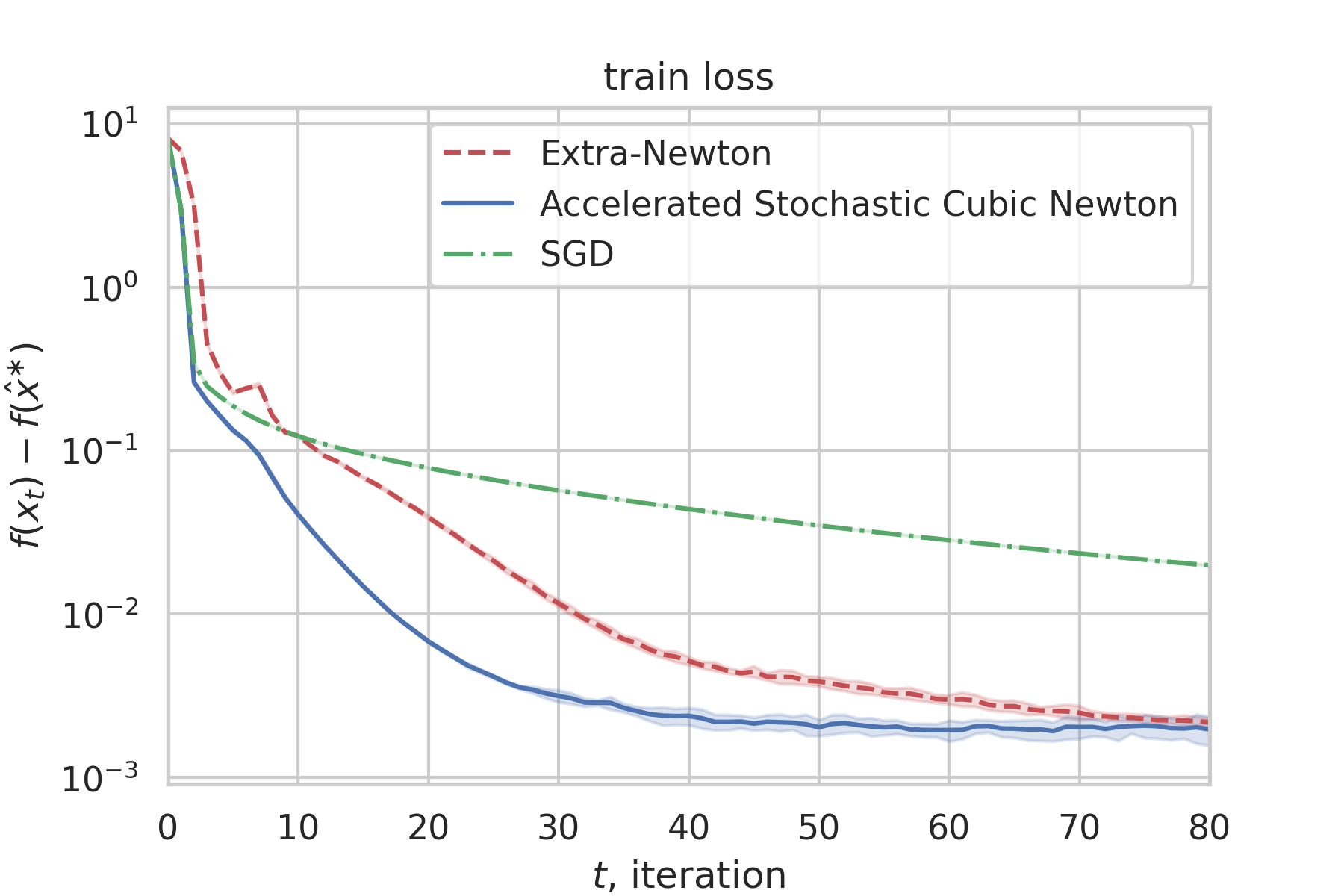}\vspace{-2mm}
         \caption{Train loss}
     \end{subfigure}
     \hfill
     \begin{subfigure}{0.48\textwidth}
         \centering
         \includegraphics[width=\textwidth]{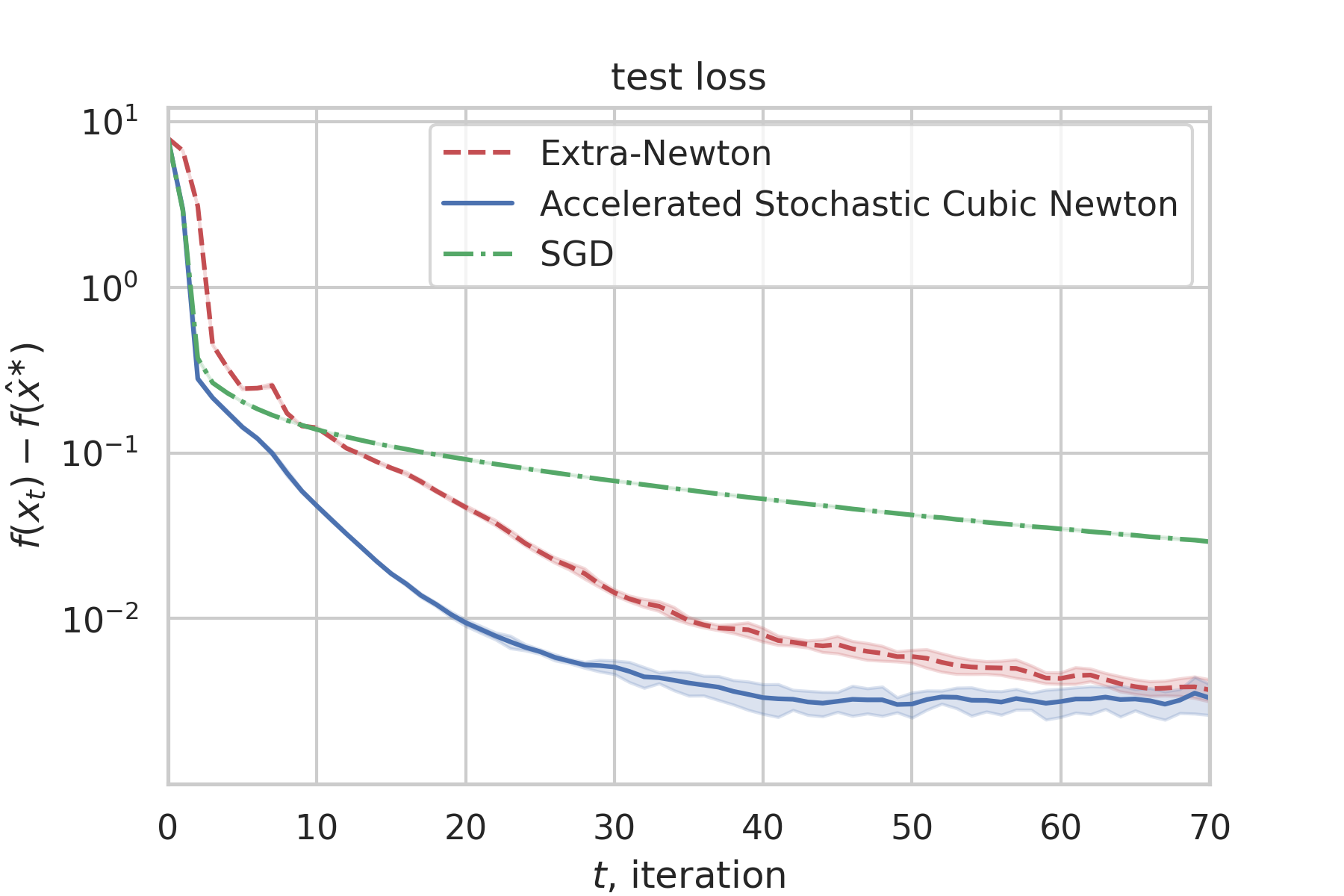}\vspace{-2mm}
         \caption{Test loss}
     \end{subfigure}
     \caption{Logistic regression on \texttt{a9a}. Gradient   batch size is $10000$, Hessian batch size is $450$}
     \label{fig:app_stoch_450}
\end{figure}

\begin{figure}[ht!]
\centering
     \begin{subfigure}{0.48\textwidth}
         \centering         \includegraphics[width=\textwidth]{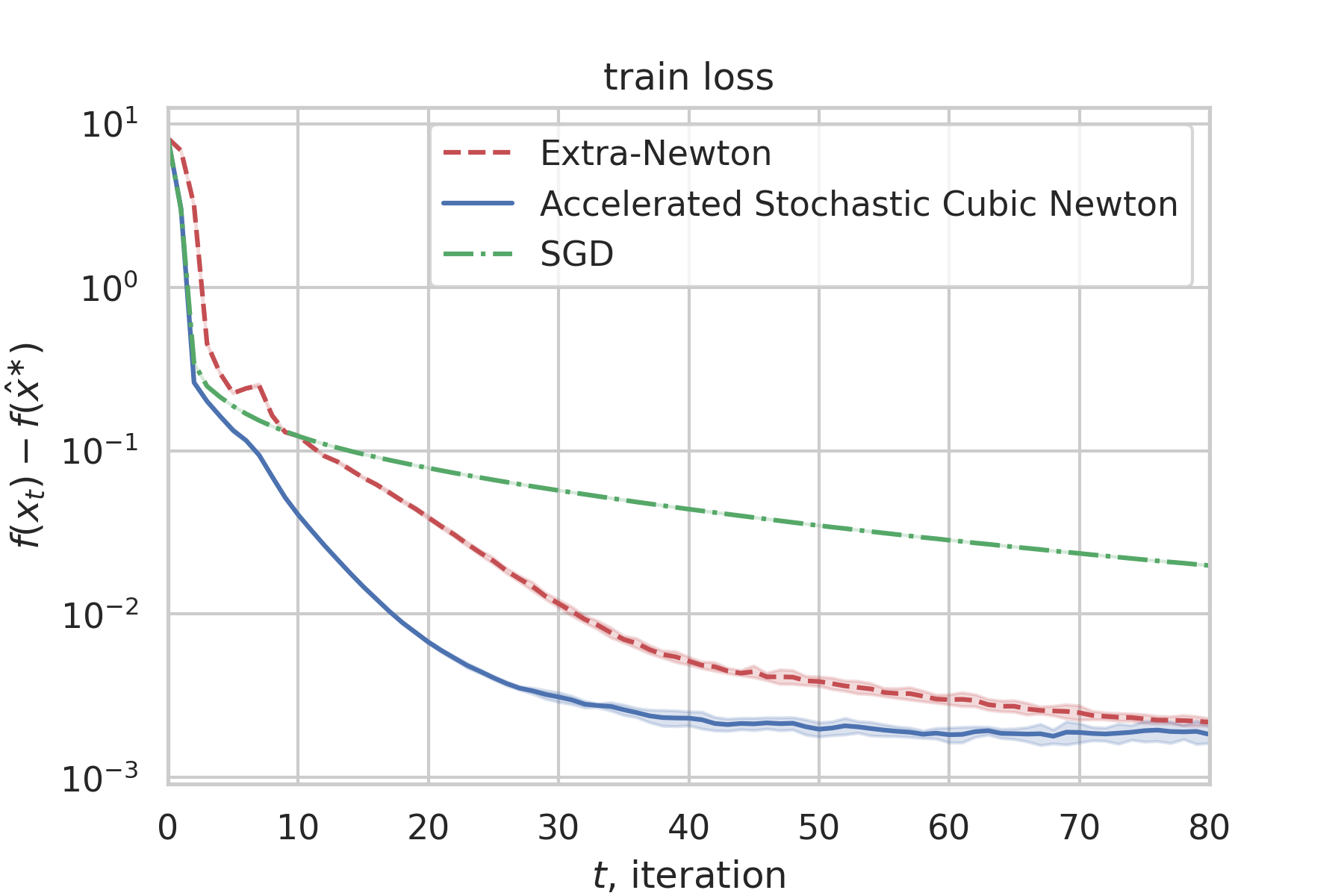}\vspace{-2mm}
         \caption{Train loss}
         \label{subfig:stoch_train_900}
     \end{subfigure}
     \hfill
     \begin{subfigure}{0.48\textwidth}
         \centering
         \includegraphics[width=\textwidth]{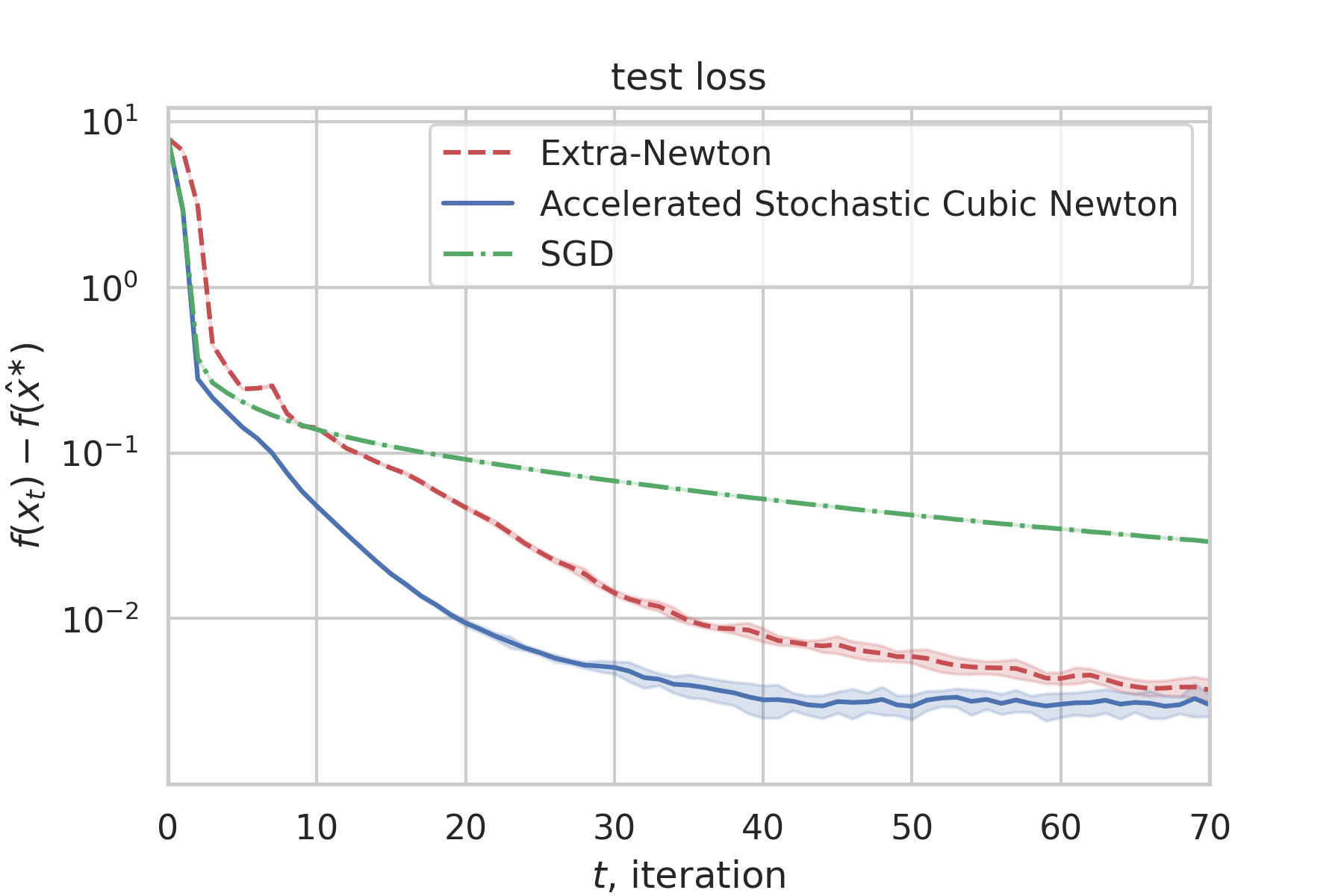}\vspace{-2mm}
         \caption{Test loss}
         \label{subfig:stoch_test_900}
     \end{subfigure}
     \caption{Logistic regression on \texttt{a9a}. Gradient   batch size is $10000$, Hessian batch size is $900$}
     \label{fig:app_stoch_900}
\end{figure}

\begin{figure}[h]
\centering
     \begin{subfigure}{0.48\textwidth}
         \centering         \includegraphics[width=\textwidth]{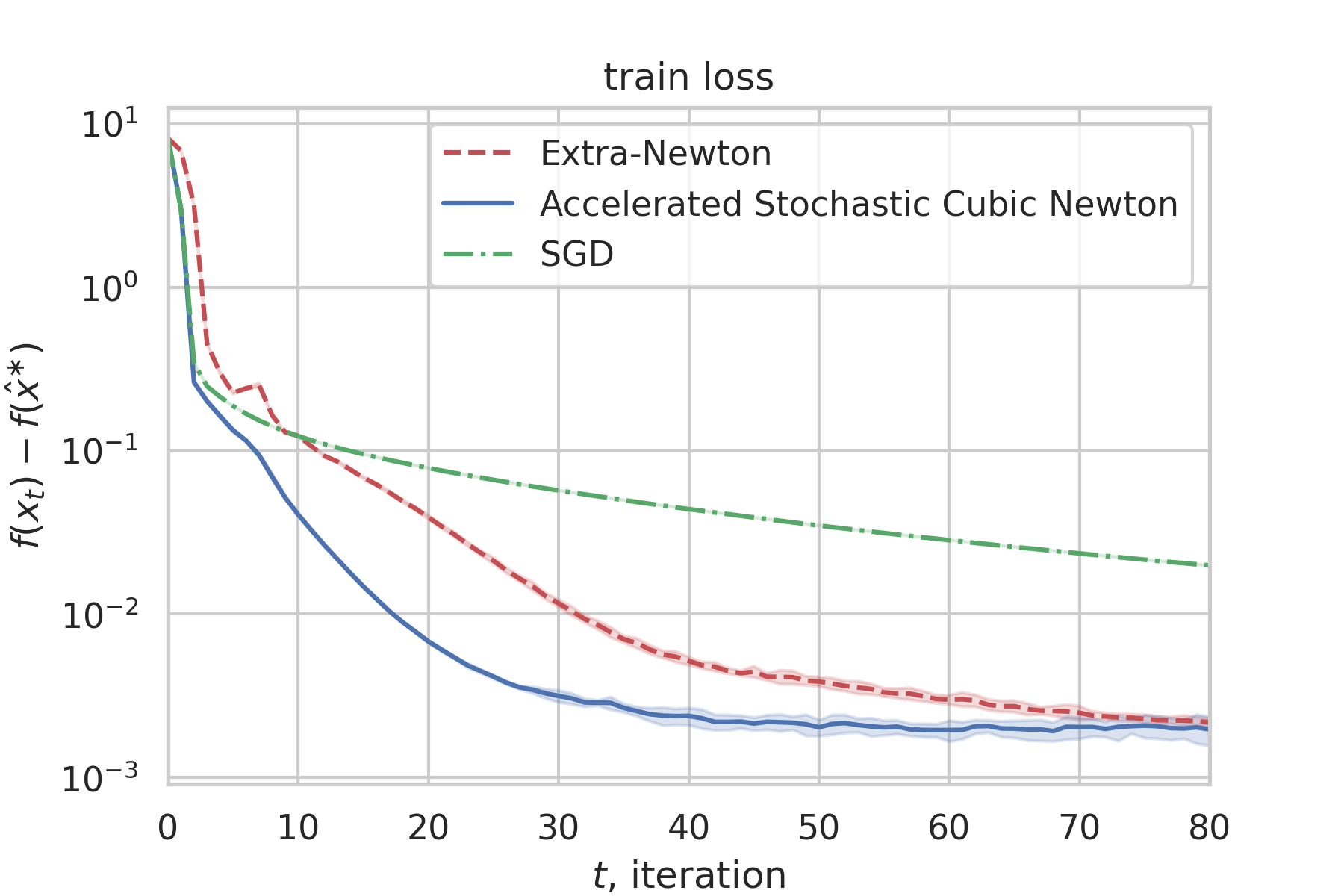}\vspace{-2mm}
         \caption{Train loss}
     \end{subfigure}
     \hfill
     \begin{subfigure}{0.48\textwidth}
         \centering
         \includegraphics[width=\textwidth]{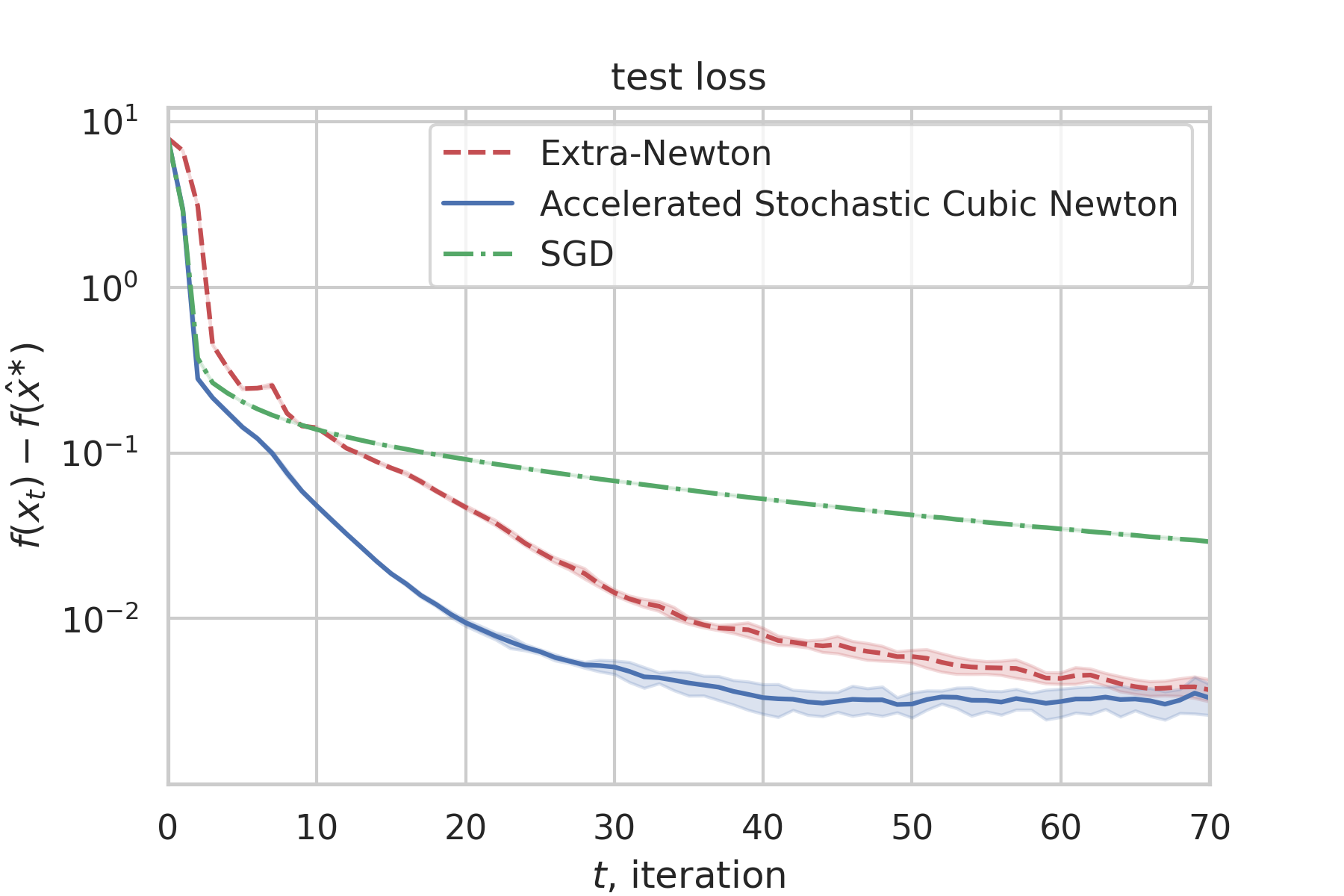}\vspace{-2mm}
         \caption{Test loss}
     \end{subfigure}
     \caption{Logistic regression on \texttt{a9a}. Gradient and Hessian batch sizes are  $1500$}
     \label{fig:app_stochastic_increased}
\end{figure}

\begin{figure}[ht!]
\centering
     \begin{subfigure}{0.48\textwidth}
         \centering         \includegraphics[width=\textwidth]{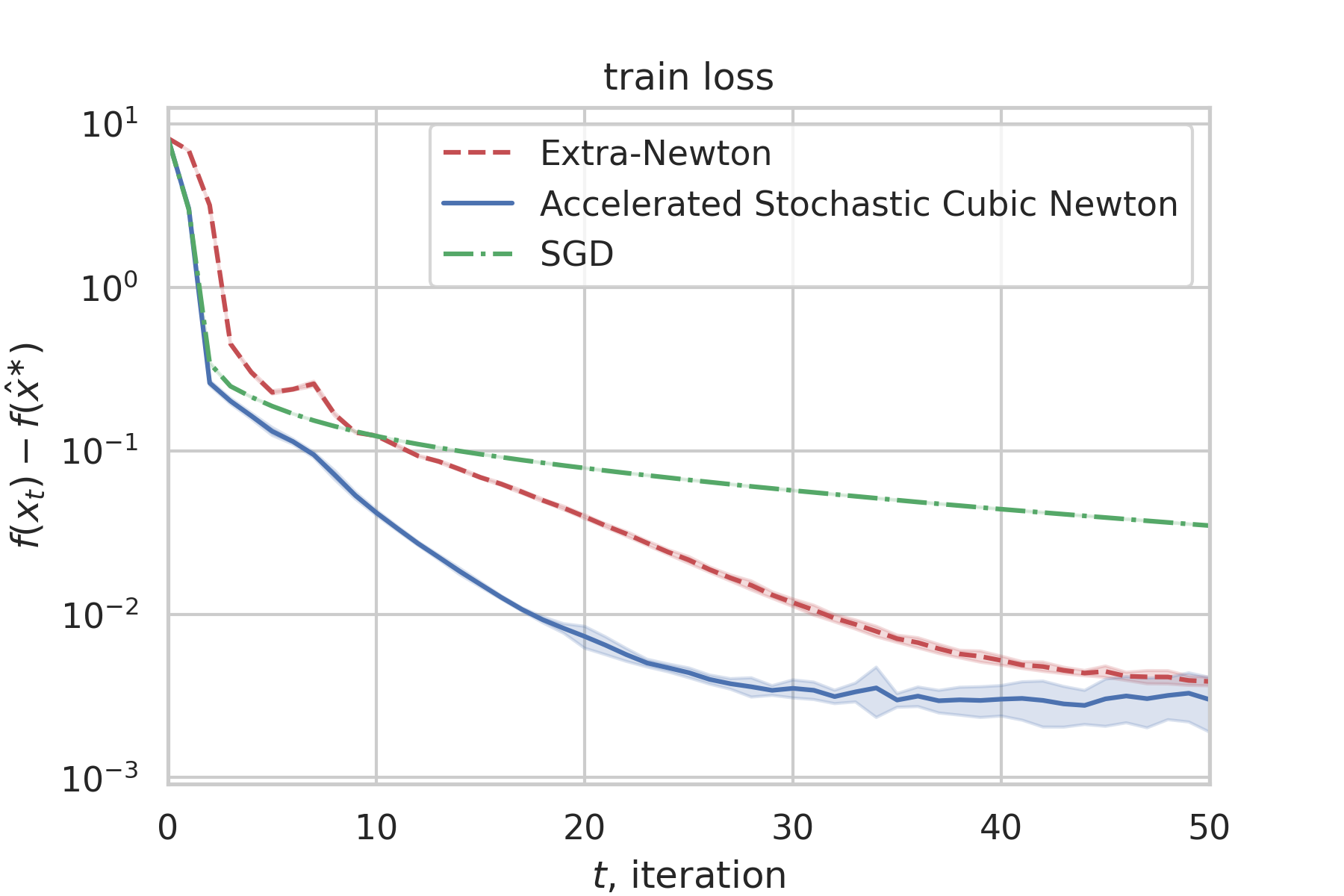}\vspace{-2mm}
         \caption{Train loss}
         \label{subfig:stoch_train_150}
     \end{subfigure}
     \hfill
     \begin{subfigure}{0.48\textwidth}
         \centering
         \includegraphics[width=\textwidth]{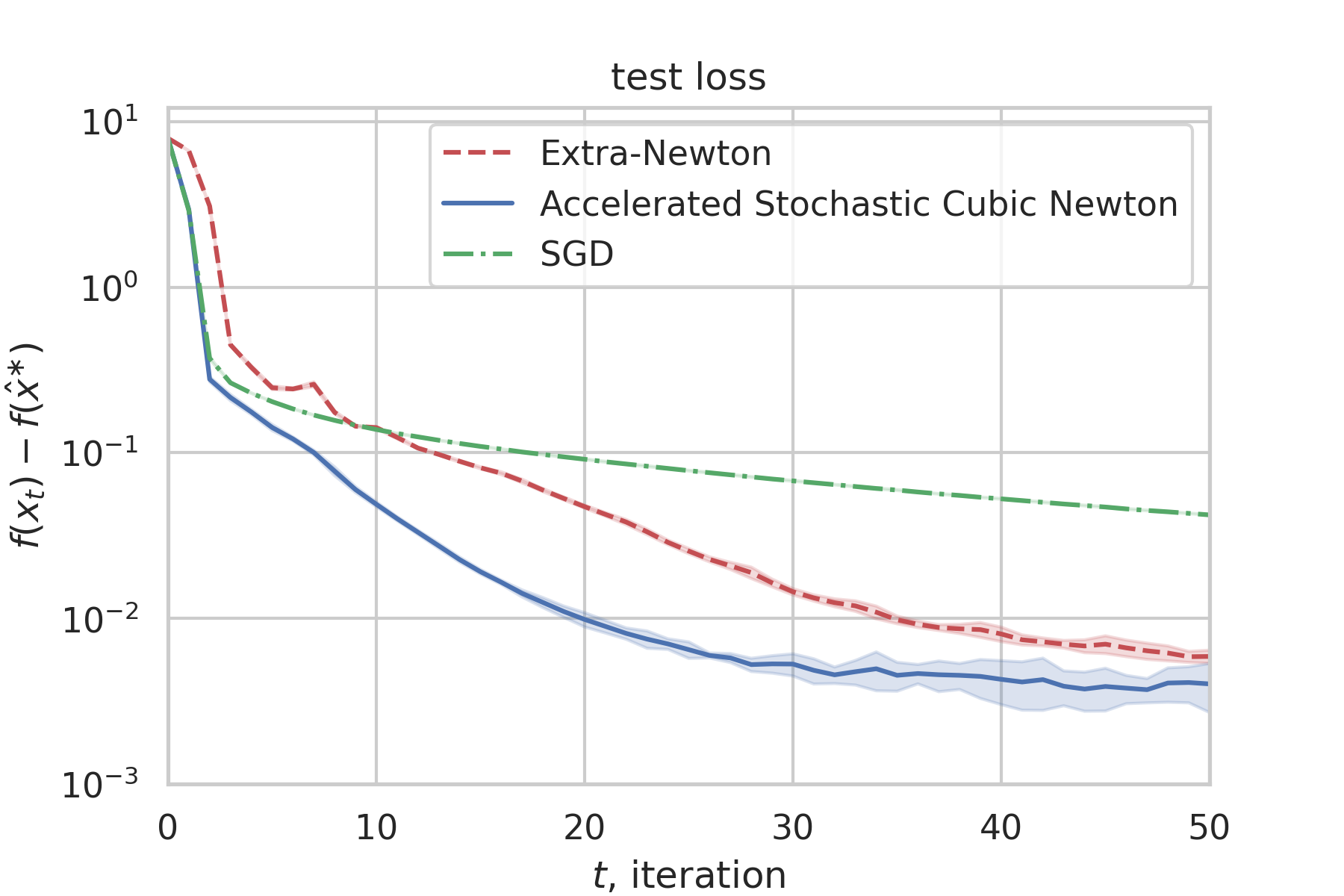}\vspace{-2mm}
         \caption{Test loss}
         \label{subfig:stoch_test_150}
     \end{subfigure}
     \caption{Logistic regression on \texttt{a9a}. Gradient   batch size is $10000$, Hessian batch size is $150$}
     \label{fig:app_stochastic_increased_150}
\end{figure}

\begin{figure}[h]
\centering
     \begin{subfigure}{0.48\textwidth}
         \centering         \includegraphics[width=\textwidth]{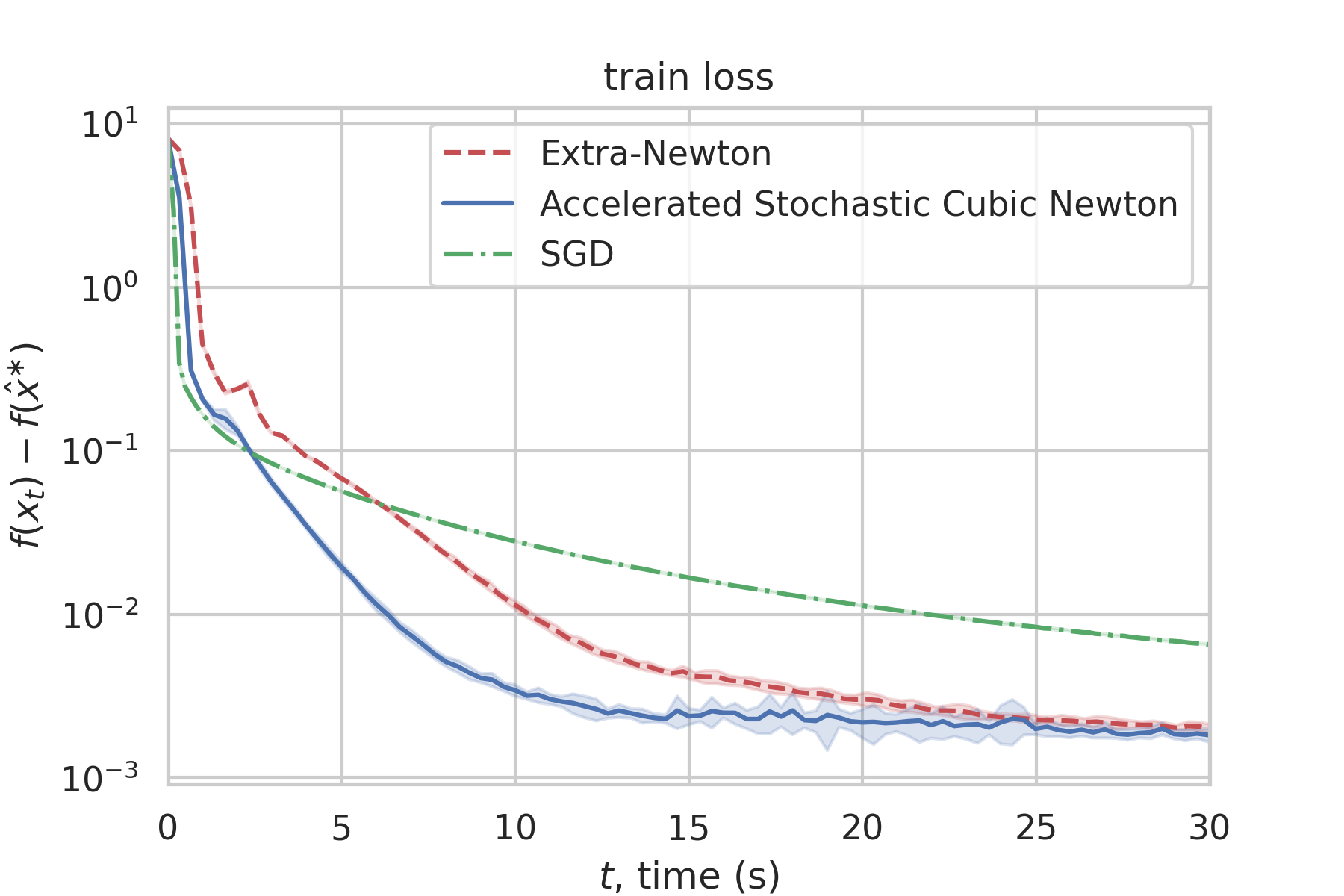}\vspace{-2mm}
         \caption{Train loss}
     \end{subfigure}
     \hfill
     \begin{subfigure}{0.48\textwidth}
         \centering
         \includegraphics[width=\textwidth]{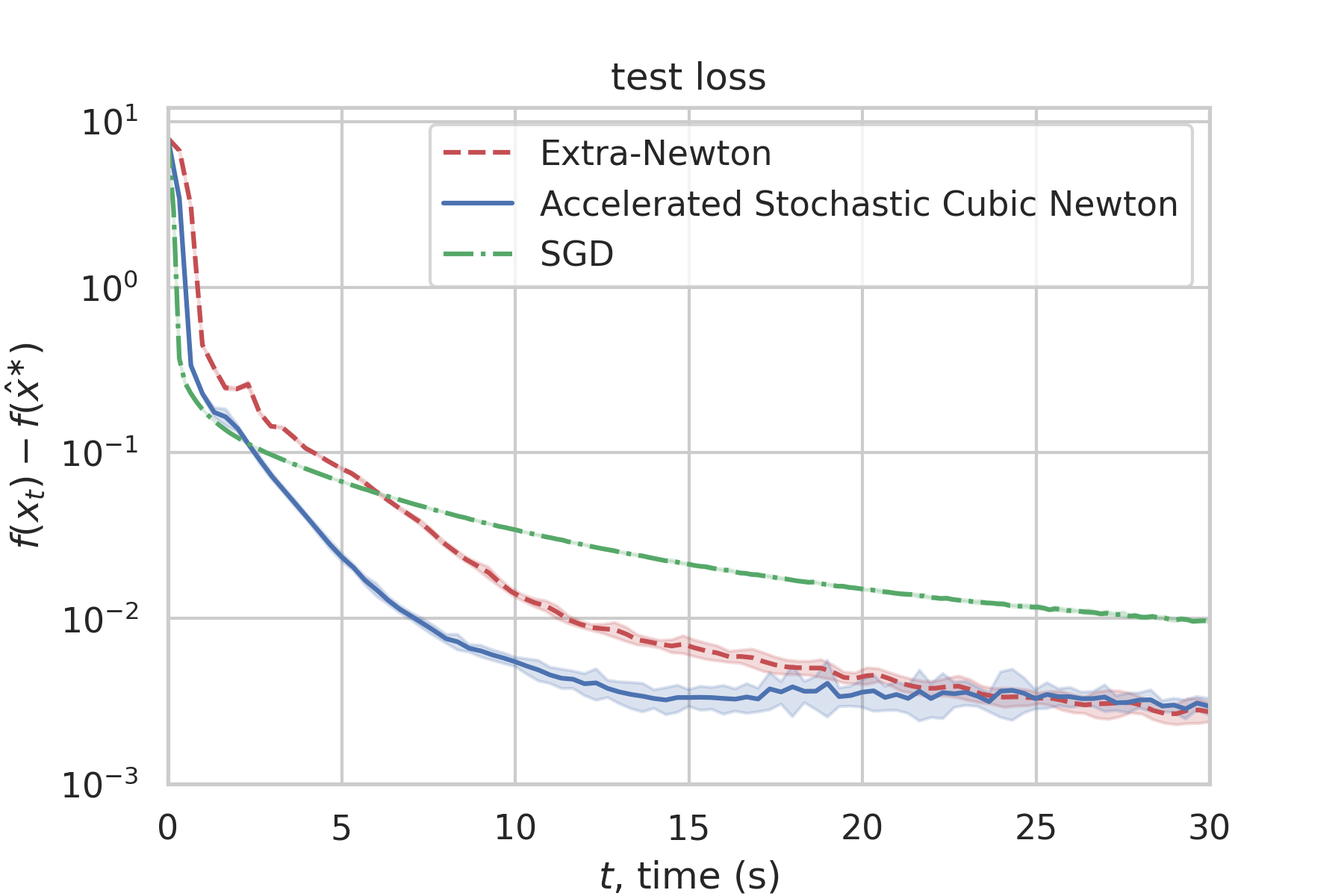}\vspace{-2mm}
         \caption{Test loss}
     \end{subfigure}
     \caption{Logistic regression on \texttt{a9a}. Gradient   batch size is $10000$, Hessian batch size is $150$}
     \label{fig:app_stochastic_150_time}
\end{figure}

\begin{figure}[ht!]
\centering
     \begin{subfigure}{0.48\textwidth}
         \centering         \includegraphics[width=\textwidth]{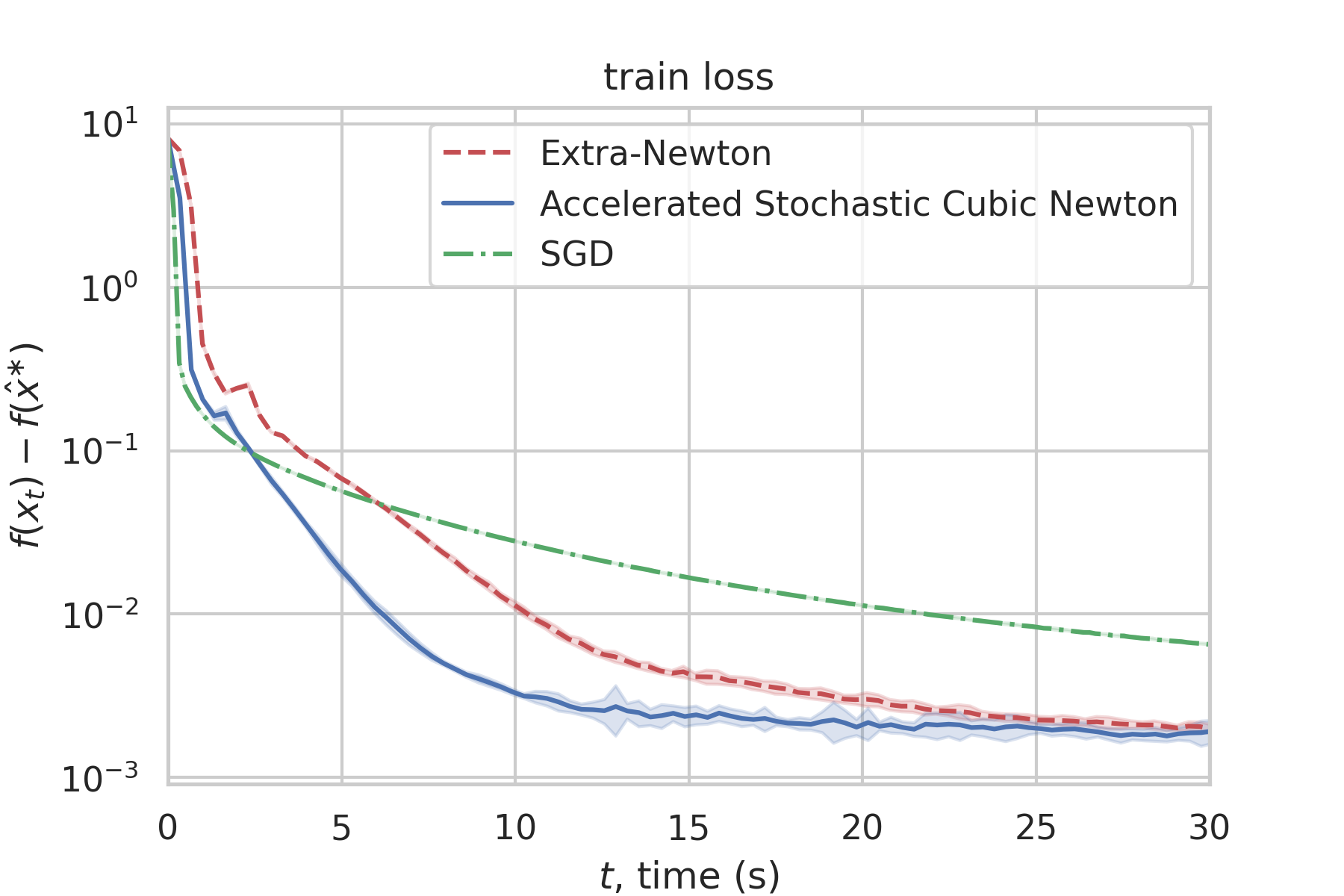}\vspace{-2mm}
         \caption{Train loss}
     \end{subfigure}
     \hfill
     \begin{subfigure}{0.48\textwidth}
         \centering
         \includegraphics[width=\textwidth]{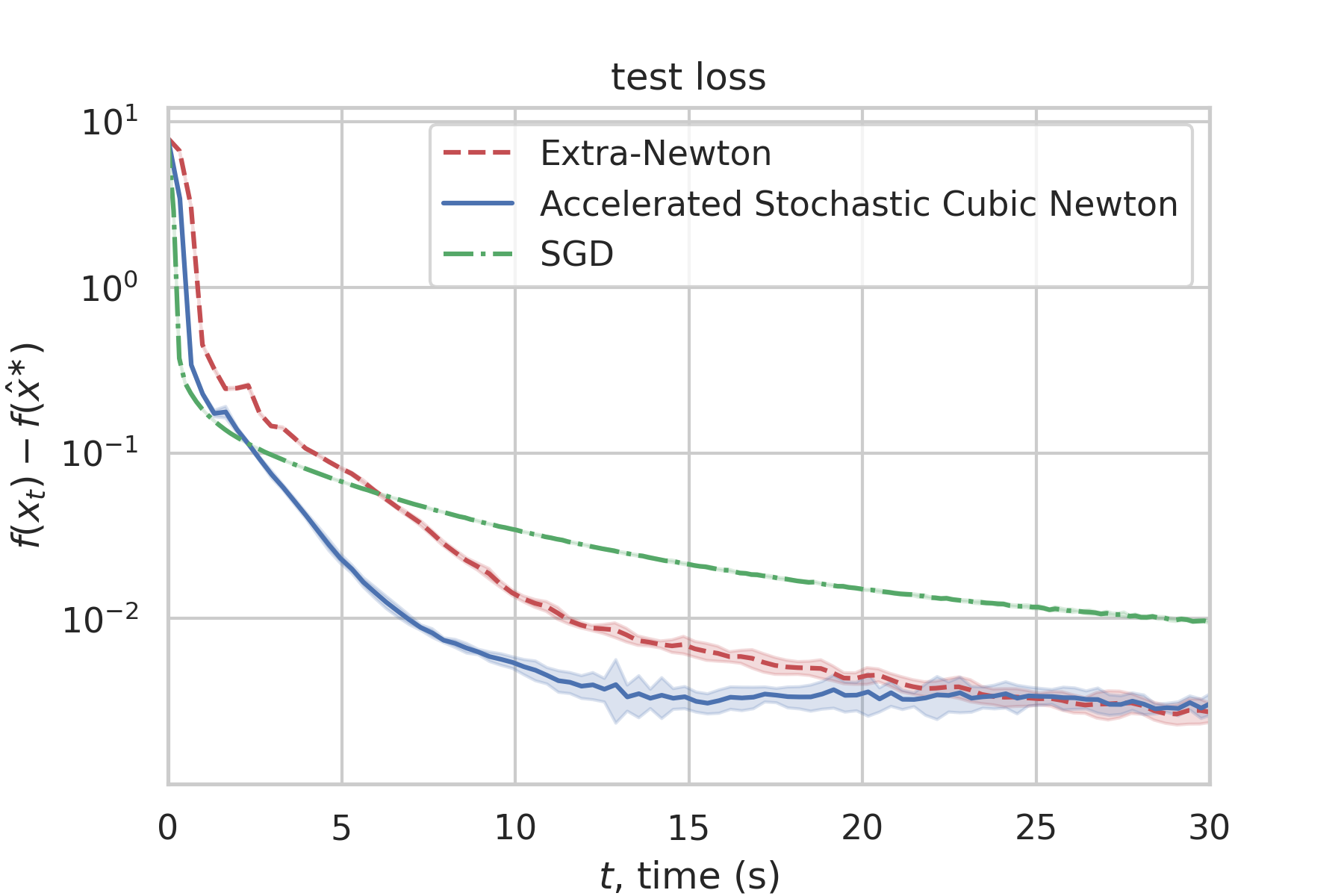}\vspace{-2mm}
         \caption{Test loss}
     \end{subfigure}
     \caption{Logistic regression on \texttt{a9a}. Gradient   batch size is $10000$, Hessian batch size is $450$}
     \label{fig:app_stochastic_450_time}
\end{figure}

\section{Proof of Lemmas \ref{lm:2ord_bound} and \ref{lm:p_ord_bound}}
\begin{customlemma}{\ref{lm:p_ord_bound}}
     Let Assumption \ref{as:lip_p} hold. Then, for any $x,y \in \mathbb{R}^d$, we have
    \begin{equation*}
    \begin{split}
         |f( y) - \phi_{{x},p}(y)|  
         \leq \textstyle{\sum \limits_{i = 1}^p} \frac{1}{i!}\|(G_i(x) - \nabla^i f({x}))[y - x]^{i-1}\|\|y-x\| + \frac{L_p}{(p+1)!}\|y - x\|^{p+1},
     \end{split}
\end{equation*}
    \vspace{-12pt}
\begin{equation*}
    \begin{split}
         \|\nabla f( y) - \nabla \phi_{x,p}(y)\| 
         \leq  \textstyle{\sum \limits_{i = 1}^{p}} \frac{1}{(i-1)!}\|(G_i(x) - \nabla^i f({x}))[y - x]^{i-1}\| + \frac{L_p}{p!}\|y - x\|^{p},
    \end{split}
\end{equation*}
\end{customlemma}

The proof of that lemma is the same as proofs of Lemmas~1,~2 from \cite{agafonov2023inexact}. Lemma~\ref{lm:2ord_bound} is a special case of the Lemma~\ref{lm:p_ord_bound} for $p=2$.

\section{Proof of Theorem \ref{thm:acc_convergence}}
\label{sec:proof_2_ord}

The full proof is organized as follows:
\begin{itemize}
    \item Lemma \ref{lem:upper_seq} provides an upper bound for the estimating sequence  $\psi_t(x)$;
    \item Lemmas \ref{lem:cubic_bound_acc}, \ref{lem:step} present the efficiency of Inexact Cubic Newton step $x_{t+1} = S_{M,\delta_t}(v_{t})$. 
    \item Lemma \ref{lem:step} provides a lower bound on $\psi_t(x)$ based on results of technical Lemmas
    \ref{lem:dual}- \ref{lm:argmin}; 
    \item Everything is combined together in Theorem \ref{thm:2_ord_app} in order to prove convergence and obtain convergence rate.
\end{itemize}

The following lemma shows that the sequence of functions $\bar{\psi}_t(x)$ can be upper bounded by the properly regularized objective function.

\begin{lemma}\label{lem:upper_seq}
    For convex function $f(x)$ and $\psi_t(x)$, we have
    \begin{equation}
        \label{eq:acc_upper_bound}
        \psi_t(x^{\ast}) \leq \frac{f(x^{\ast})}{A_{t - 1}}  + \frac{ \bk_2^{t} + \lambda_t}{2} \|x^{\ast} - x_0\|^2   + \frac{\bk^{t}_{3}}{6}\|x^{\ast} - x_0\|^3 + err_{t}^{up},
        \end{equation}
    where
    \begin{equation}
        \label{eq:err_up}
        err^{up}_t = \sum \limits_{j = 0}^{t - 1} \frac{\alpha_j}{A_j} \la g(x_{j+1}) - \nabla f(x_{j+1}) , x^{\ast} - x_{j+1}\ra
    \end{equation}
\end{lemma}

\begin{proof}
For $t=0$, let us define $A_{-1}$ such that $\tfrac{1}{A_{-1}}=0$ then $\tfrac{f(x^{\ast})}{A_{-1}}=0$ and
   $$
        \psi_0(x^{\ast}) \leq \frac{\bk^0_2+\lambda_0}{2}\|x^{\ast} - x_0\|^2 + \frac{\bk^0_3}{3}\|x^{\ast} - x_0\|^3$$
    From 
    \begin{equation}\label{eq:estimating_seq}
        \psi_{t+1}(x):= \psi_{t}(x)+ \tfrac{\lambda_{t+1} - \lambda_{t}}{2}\|x - x_0\|^2  + \textstyle{\sum \limits_{i = 2}^{3} }
                \tfrac{\bk^{t+1}_i - \bk^{t}_i}{i}\|x - x_0\|^i
                +\tfrac{\alpha_{t}}{A_{t}} 
                l(x,x_{t+1}),
    \end{equation}
    we have
    \begin{gather}
        \psi_t(x^{\ast}) 
        = \frac{\bk^{t }_2+\lambda_t }{2}\|x^{\ast} - x_0\|^2+\frac{\bk^{t }_3}{3}\|x^{\ast} - x_0\|^3 + \sum \limits_{j = 0}^{t - 1} \frac{\alpha_j}{A_j} l(x^{\ast},x_{j+1}).
        \label{eq:lem_upper_seq_pr2}
    \end{gather}
From \eqref{eq:alphas}, we have that, for all $j \geq 1$, $A_j = A_{j-1}(1-\alpha_j)$, which leads to $\frac{\alpha_j}{A_j}=\frac{1}{A_j}-\frac{1}{A_{j-1}}$. Hence, we have $\sum \limits_{j = 0}^{t - 1} \frac{\alpha_j}{A_j} = \frac{1}{A_{t-1}} - \frac{1}{A_{-1}}=\frac{1}{A_{t-1}}$ and, using the convexity of the objective $f$, we get
    \begin{gather}
        \sum \limits_{j = 0}^{t - 1} \frac{\alpha_j}{A_j} l(x^{\ast},x_{j+1}) \leq \sum \limits_{j = 0}^{t - 1} \frac{\alpha_j}{A_j} \bar{l}(x^{\ast},x_{j+1}) + 
        \sum \limits_{j = 0}^{t - 1} \frac{\alpha_j}{A_j} \la g(x_{j+1}) - \nabla f(x_{j+1}) , x^{\ast} - x_{j+1}\ra \\
        \leq f(x^{\ast})\sum \limits_{j = 0}^{t - 1} \frac{\alpha_j}{A_j} + err_t^{up} = \frac{f(x^{\ast})}{A_{t-1}} + err_t^{up}, \label{eq:lem_upper_seq_pr3}
    \end{gather}
    where $\bar{l}(x, y) = f(y) + \langle \nabla f(y), x - y \rangle$.
    Finally, combining all the inequalities from above, we obtain 
    \begin{gather}
        \psi_t(x^{\ast})  \stackrel{\eqref{eq:lem_upper_seq_pr2},\eqref{eq:lem_upper_seq_pr3}}{\leq}  \frac{f(x^{\ast})}{A_{t - 1}} + \frac{ \bk_2^{t}+\lambda_t}{2} \|x^{\ast} - x_0\|^2 + \frac{\bk^{t}_{3}}{6}\|x^{\ast} - x_0\|^3 + err_t^{up}.
    \end{gather}
\end{proof}

\begin{lemma}
\label{lem:cubic_bound_acc}
For the function $f(x)$ with $L_2$-Lipschitz-continuous Hessian and $H(x_t)$ is $ \delta_2^t$-inexact Hessian for $v_t,x_{t+1} \in \R^d$ we have
\begin{equation}
\label{eq:grad_model_bound}
    \|\nabla \phi_{v_t}(x_{t+1}) - \nabla f(x_{t+1})\| 
        \leq
         \delta_2^t \|x_{t+1} - v_t\| + \frac{L_2}{2}\|x_{t+1} - v_t\|^{2} +  \|g(v_t) - \nabla f(v_t) \|,
\end{equation}
where we denote $\delta_2^t = \delta_t^{v_t, x_{t+1}} $ to simplify the notation.
\end{lemma}
\begin{proof}
    \begin{align*}
         \|\nabla \phi_{v_t}(x_{t+1}) - \nabla f(x_{t+1})\| &= \|\nabla \phi_{v_t}(x_{t+1}) - \Phi_{v_t}(x_{t+1}) +\Phi_{v_t}(x_{t+1})  - \nabla f(x_{t+1})\| \\
         &\leq \|\nabla \phi_{v_t}(x_{t+1}) - \Phi_{v_t}(x_{t+1})\| +\|\Phi_{v_t}(x_{t+1})  - \nabla f(x_{t+1})\| \\
         &=\| (\nabla^2 f(v_t) - B_{v_t})(x_{t+1}- v_t)\| +\|\Phi_{v_t}(x_{t+1})  - \nabla f(x_{t+1})\|\\
         &+ \|g(v_t) - \nabla f(v_t) \| \\
            &{\leq} \delta_2^t \|x_{t+1} - v_t\| + \frac{L_2}{2}\|x_{t+1} - v_t\|^{2} + \|g(v_t) - \nabla f(v_t) \|
    \end{align*}
\end{proof}

The next Lemma characterizes the progress of the inexact cubic step $S_{M,\delta_t}(v_{t})$ in Algorithm \ref{alg:inexact_acc_detailed}.

\begin{lemma}
    \label{lm:scalar_lb_cases}
    Let $\{x_t, v_t\}_{t \geq 1}$ be generated by Algorithm \ref{alg:inexact_acc_detailed}. Then, for any $\bar{\delta}_t\geq 4{\delta}_2^t$, $M\geq 4L_2$  and $x_{t+1} = S_{M,\bar{\delta}_t}(v_{t})$, the following holds
    \begin{equation}\label{eq:scalar_lb_cases}
    \begin{gathered}
        \aa{\frac{\tau^2}{\bar{\delta}_t}} + \frac{\aa{2}}{\bar{\delta}_{t}}\|g(v_t) - \nabla f(v_t) \|^2 + \langle \nabla f(x_{t+1}), v_t - x_{t+1} \rangle 
        \\
        \geq 
         \min \left\{ \|\nabla f(x_{t+1})\|^2\left( \tfrac{1}{4\delta_t }\right), \right.
        \left.\|\nabla f(x_{t+1})\|^\frac{3}{2}
        \left( \tfrac{1}{3M}\right)^\frac{1}{2}\right\}.
        \end{gathered}
    \end{equation}
\end{lemma}

\begin{proof}
    For simplicity, we denote $r_{t+1} = \|x_{t+1} - v_t\|$ and
    \begin{equation}
        \label{big_sum_lemma6}
        \zeta_{t+1} =  \bar{\delta}_t + \frac{M}{2}\|x_{t+1} - v_t\|.
    \end{equation}
    
    By Definition \ref{def:inexact_sub} for $x_{t+1} = S^{M,\bar{\delta}_t, \tau}(v_{t})$
    
    \begin{equation}
   \label{eq:opt_cnd_acc}
    \begin{gathered}
    \aa{\tau \geq} \|\nabla \phi_{v_t}(x_{t+1})+ \bar{\delta}_t(x_{t+1} - v_t)  + \frac{M}{2}\|x_{t+1} - v_t\|(x_{t+1} - v_t)\| \\ \stackrel{\eqref{big_sum_lemma6}}{=} \|\nabla \phi_{v_t}(x_{t+1}) + \zeta_{t+1} (x_{t+1} - v_t)\|.
    \end{gathered}
    \end{equation}
    
    We start with getting an upper bound for $\|\nabla \phi_{v_t}(x_{t+1}) - \nabla f(x_{t+1})\|$. 
    \begin{equation}\label{eq:progress_step_error_upper_bound}
        \|\nabla \phi_{v_t}(x_{t+1}) - \nabla f(x_{t+1})\| 
        \stackrel{\eqref{eq:grad_model_bound}}{\leq} 
         \delta_2^t r_{t+1} + \frac{L_2}{2}r_{t+1}^{2} + \|g(v_t) - \nabla f(v_t) \|.
    \end{equation}

    From inexact solution of the subproblem we get
    \begin{equation}
    \begin{gathered}\label{eq:progress_step_inexact_upper_bound}
        \|\nabla f(x_{k+1}) + \zeta_{t+1}(x_{t+1} - v_t)\|^2 \\
        \leq  \|\nabla \phi_{v_t}(x_{t+1}) - \nabla f(x_{t+1}) - \nabla \phi_{v_t}(x_{t+1}) - \zeta_{t+1}(x_{t+1} - v_t)\|^2 \\
        \leq 2\|\nabla \phi_{v_t}(x_{t+1}) - \nabla f(x_{t+1})\|^2 + 2\|\nabla \phi_{v_t}(x_{t+1}) + \zeta_{t+1}(x_{t+1} - v_t)\|^2 \\
        \stackrel{\text{Def.~\ref{def:inexact_sub}}}{\leq} 2\|\nabla \phi_{v_t}(x_{t+1}) - \nabla f(x_{t+1})\|^2 + 2\tau^2
    \end{gathered}
    \end{equation}

   Next, from the previous inequality and \aa{\eqref{eq:progress_step_error_upper_bound}}, we get
    \begin{gather*}
        \aa{4}\|g(v_t) - \nabla f(v_t) \|^2 + \aa{4}(\delta_2^t  + \tfrac{L_2}{2}r_{t+1})^2r_{t+1}^2 \stackrel{\eqref{eq:progress_step_error_upper_bound}}{\geq} \aa{2}\|\nabla \phi_{v_t, p}(x_{t+1}) - \nabla f(x_{t+1})\|^2 
        \\
\aa{\stackrel{\eqref{eq:progress_step_inexact_upper_bound}}{\geq}} \left\|\nabla  f(x_{t+1}) +  \zeta_{t+1} (x_{t+1} - v_t) \right\|^2  - \aa{2\tau^2}\\
        = 2 \langle \nabla f(x_{t+1}), x_{t+1} - v_t \rangle \zeta_{t+1} + \|\nabla f(x_{t+1}) \|^2  + \zeta_{t+1}^2\|x_{t+1} - v_t\|^2 - \aa{2\tau^2}.
    \end{gather*}
    
    Hence,
    
    \begin{gather*}
        \aa{\frac{\tau^2}{\zeta_{t+1}}} + \frac{\aa{2}}{\zeta_{t+1}}\|g(v_t) - \nabla f(v_t) \|^2 + 
        \langle \nabla f(x_{t+1}), v_t - x_{t+1} \rangle 
        \geq
        \frac{1}{2\zeta_{t+1}}\|\nabla f(x_{t+1}) \|^2
        + \frac{1}{2\zeta_{t+1}}\ls \zeta_{t+1}^2 - 4(\delta_2^t  + \tfrac{L_2}{2}r_{t+1})^2  \rs r_{t+1}^2,
    \end{gather*}
    
    and finally by using defenetion of $\zeta_{t+1}$, we get
    
    \begin{gather*}
        \frac{\tau^2}{\bar{\delta}_t} + \frac{2}{\bar{\delta}_{t}}\|g(v_t) - \nabla f(v_t) \|^2 + 
        \langle \nabla f(x_{t+1}), v_t - x_{t+1} \rangle 
        \geq
        \frac{1}{2\zeta_{t+1}}\|\nabla f(x_{t+1}) \|^2
        + \frac{1}{2\zeta_{t+1}}\ls \zeta_{t+1}^2 - 4(\delta_2^t  + \tfrac{L_2}{2}r_{t+1})^2  \rs r_{t+1}^2
    \end{gather*}

Next, we consider $2$ cases depending on which term dominates in the $\zeta_{t+1}$.

\begin{itemize}[leftmargin=10pt,nolistsep]
    \item If $\bar{\delta}_t \geq \frac{M}{2}\|x_{t+1}-v_t\|$, then we get the following bound
    
\begin{gather*}
        \frac{\tau^2}{\bar{\delta}_t} + \frac{2}{\bar{\delta}_{t}}\|g(v_t) - \nabla f(v_t) \|^2 + 
        \langle \nabla f(x_{t+1}), v_t - x_{t+1} \rangle 
        \geq
        \frac{1}{2\zeta_{t+1}}\|\nabla f(x_{t+1}) \|^2
        + \frac{1}{2\zeta_{t+1}}\ls \zeta_{t+1}^2 - 4({\delta_2^t}  + \tfrac{L_2}{2}r_{t+1})^2  \rs r_{t+1}^2\\
        \geq \tfrac{1}{4\bar{\delta}_t }
        \|\nabla f(x_{t+1})\|^2
    \end{gather*}
    
    \item If $\bar{\delta}_t < \frac{M}{2}\|x_{t+1}-v_t\|$ , then similarly to previous case, we get
    
    \begin{gather*}
       \frac{\tau^2}{\bar{\delta}_t} + \frac{2}{\bar{\delta}_{t}}\|g(v_t) - \nabla f(v_t) \|^2 + 
        \langle \nabla f(x_{t+1}), v_t - x_{t+1} \rangle \\
        \geq
        \frac{\|\nabla f(x_{t+1}) \|^2}{2\zeta_{t+1}}
        + 
       \ls\ls \bar{\delta}_t + \tfrac{M}{2}r_{t+1}\rs^2- 4\ls{\delta_2^t} + \tfrac{L_2}{2}r_{t+1}\rs^2 \rs \frac{r_{t+1}^2}{2\zeta_{t+1}} \\
       = 
       \frac{\|\nabla f(x_{t+1}) \|^2}{2\zeta_{t+1}}
        + 
        \ls\ls \bar{\delta}_t - 2\delta_2^t + \tfrac{M-\aa{2}L_2}{2}r_{t+1}\rs \ls\bar{\delta_t} + 2\delta_2^t + \tfrac{\aa{2}L_2+M}{2}r_{t+1}\rs \rs \frac{r_{t+1}^2}{2\zeta_{t+1}}
    \end{gather*}
    
    \begin{gather*}
        \geq \frac{\|\nabla f(x_{t+1}) \|^2}{2Mr_{t+1}}
        + \frac{M^2-\aa{4}L_2^2}{4} \frac{r_{t+1}^3}{2M}
        \geq \frac{\|\nabla f(x_{t+1}) \|^2}{2Mr_{t+1}}
        + \frac{2M}{32} r_{t+1}^3
        \geq \ls\frac{1}{3M}\rs^{\frac{1}{2}}\|\nabla f(x_{t+1}) \|^{\frac{3}{2}},
    \end{gather*}
    
    where for the last inequality, we use $\tfrac{\alpha}{r} + \tfrac{\beta r^3}{3} \geq \frac{4}{3}\beta^{1/4}\al^{3/4}$.
\end{itemize} 
\end{proof}

We use the following lemma  \cite[Lemma~7]{agafonov2023inexact}.
\begin{lemma}\label{lm:argmin}
    Let $h(x)$ be a convex function, $x_0 \in \mathbb{R}^n$, $\theta_i \geq 0$~~for~ $i = 2,\ldots, p + 1$ and 
    $$\bar{x} = \arg \min \limits_{x \in \mathbb{R}^n} \{\bar{h}(x) = h(x) + \sum \limits_{i = 2}^{p + 1}  \theta_i d_i(x -x_0)\},$$ 

    where $d_i(x) = \tfrac{1}{i}\|x\|^i$ is a power-prox function.
    Then, for all $x\in \mathbb{R}^n$,
    
    $$\bar{h}(x) \geq \bar{h}(\bar{x}) + \sum  \limits_{i = 2}^{p + 1} \left(\frac{1}{2}\right)^{i - 2} \theta_i d_i(x - \bar{x}).$$
\end{lemma}

We will also use the next technical lemma \cite{nesterov2008accelerating,ghadimi2017second} on Fenchel conjugate for the $p$-th power of the norm.
\begin{lemma}\label{lem:dual}
    Let $g(z)=\frac{\theta}{p}\|z\|^{p}$ for $p \geq 2$ and $g^{*}$ be its conjugate function i.e., $g^{*}(v)=\sup _{z}\{\langle v, z\rangle-$ $g(z)\} .$ Then, we have
$$
g^{*}(v)=\frac{p-1}{p}\left(\frac{\|v\|^{p}}{\theta}\right)^{\frac{1}{p-1}}
$$
Moreover, for any $v, z \in \mathbb{R}^{n}$, we have $g(z)+g^{*}(v)-\langle z, v\rangle \geq 0 .$
\end{lemma}

Finally, the last step is the next Lemma which prove that $\frac{f(x_t)}{A_{t-1}} \leq  \min \limits_x \bar{\psi}_t(x) = \bar{\psi}_t^{\ast}$.

\begin{lemma}\label{lem:step}
    Let $\{x_t, y_t\}_{t \geq 1}$ be generated by Algorithm \ref{alg:inexact_acc_detailed}. 
    Then
    \begin{equation}\label{eq:lemma_ass}
    \psi_t^{\ast}  = \min \limits_x \psi_t(x)   \geq \frac{f(x_t)}{A_{t-1}} - err_{t}^{v}-err_{t}^{x}
    \aa{-err_{t}^{\tau}},
    \end{equation}
    where 
    \begin{equation}\label{eq:err_v}
    err_{t}^{v}  = \sum \limits_{j = 0}^{t - 1}  \frac{\aa{2}}{A_j \bar{\delta}_j}\|g(v_j) - \nabla f(v_j) \|^2 ,
    \end{equation}
    \aa{\begin{equation}\label{eq:err_tau_p}
    err_{t}^{\tau}  = \sum \limits_{j = 0}^{t - 1}  \frac{\tau^2}{A_j \bar{\delta}_j} ,
    \end{equation}}
    and
     \begin{equation}\label{eq:err_x}
    err_{t}^{x}  = \sum \limits_{j = 0}^{t-1}  \frac{\alpha_j^2}{2A_j^2\lambda_j} \|g(x_{j+1})-\nabla f(x_{j+1}) \|^2 + \sum \limits_{j = 0}^{t-1}  \frac{\alpha_j}{A_j} \la g(x_{j+1}) -\nabla f(x_{j+1}), y_{j} - x_{j+1}\ra .
    \end{equation}
\end{lemma}

\begin{proof}
    We prove Lemma by induction. Let us start with $t=0$, we define $A_{-1}$ such that $\tfrac{1}{A_{-1}}=0$. Then $\tfrac{f(x_0)}{A_{-1}}=0$ and $\psi_0^{\ast}=0$, hence, $\psi_0^{\ast}\geq \tfrac{f(x_0)}{A_{-1}}$. Let us assume that \eqref{eq:lemma_ass} is true for $t$
    and show that \eqref{eq:lemma_ass} is true for $t+1$.
    By definition,
    \begin{gather*}
        \psi_t(x) =  \frac{ \lambda_t+\bk^{t}_2}{2}\|x - x_0\|^2 + \frac{ \bk^{t}_3}{3}\|x - x_0\|^3 + \sum \limits_{j = 0}^{t - 1} \frac{\alpha_j}{A_j}l(x,x_{j+1}).
    \end{gather*}
    Next, we apply Lemma \ref{lm:argmin} with the following choice of parameters: $h(x) =  \sum \limits_{j = 0}^{t - 1}\frac{\alpha_j}{A_j}l(x,x_{j+1})$, $\theta_2 = \lambda_t + \bk^{t}_2$, and $\theta_3 = \bk^{t}_3$. \\
    By \eqref{eq:estimating_seq}, $y_t = \argmin \limits_{x\in \mathbb{R}^d} h(x)$, and we have
    \begin{gather*}
        \psi_t(y_{t+1}) \geq \psi_t^{\ast} + \frac{ \bk^{t}_2 + \lambda_t}{2}\|x - y_t\|^2 + \frac{ \bk^{t}_3}{6}\|x - y_t\|^3\\
        \stackrel{\eqref{eq:lemma_ass}}{\geq} \frac{f(x_t)}{A_{t - 1}} + \frac{ \bk^{t}_2+ \lambda_t}{2}\|x - y_t\|^2 + \frac{ \bk^{t}_3}{6}\|x - y_t\|^3 - err_{t}^{v}-err_{t}^{x} \aa{- err_{t}^{\tau}},
    \end{gather*}
    where the last inequality follows from the assumption of the lemma.
    
    By the definition of $\psi_{t+1}(x)$, the above inequality, and convexity of $f$, we obtain
    \begin{gather*}
        \psi_{t+1}(y_{t+1}) =  \psi_t(y_{t+1}) + \frac{ \lambda_{t+1}-\lambda_{t}}{2}\|y_{t+1} - x_0\|^2 
        + \frac{ \bk^{t+1}_2-\bk_2^{t}}{2}\|y_{t+1} - x_0\|^2\\
            +  \frac{ \bk^{t+1}_3-\bk_3^{t}}{3}\|y_{t+1} - x_0\|^3
         + \frac{\alpha_t}{A_t}l(y_{t+1},x_{t+1})\\
        \geq \frac{f(x_t)}{A_{t - 1}} + \frac{ \bk^{t}_2+\lambda_t}{2}\|y_{t+1} - y_t\|^2 
            +  \frac{ \bk^{t}_3}{6}\|y_{t+1} - y_t\|^3  + \frac{\alpha_t}{A_t}l(y_{t+1},x_{t+1}) - err_{t}^{v}-err_{t}^{x}\aa{- err_{t}^{\tau}}\\
             = \frac{f(x_t)}{A_{t - 1}} + \frac{ \bk^{t}_2+\lambda_t}{2}\|y_{t+1} - y_t\|^2 
            +  \frac{ \bk^{t}_3}{6}\|y_{t+1} - y_t\|^3  - err_{t}^{v}-err_{t}^{x}\aa{- err_{t}^{\tau}}\\ 
            \frac{\alpha_t}{A_t} (f(x_{t+1}) 
    + \la \nabla f(x_{t+1}), y_{t+1} - x_{t+1}\ra + \la g(x_{t+1}) -\nabla f(x_{t+1}), y_{t+1} - y_{t}\ra + \la g(x_{t+1}) -\nabla f(x_{t+1}), y_{t} - x_{t+1}\ra )\\
    \geq \frac{f(x_t)}{A_{t - 1}} + \frac{ \bk^{t}_2}{2}\|y_{t+1} - y_t\|^2 
            +  \frac{ \bk^{t}_3}{6}\|y_{t+1} - y_t\|^3  - err_{t}^{v}-err_{t}^{x}\aa{- err_{t}^{\tau}}\\ 
            +\frac{\alpha_t}{A_t} \ls f(x_{t+1}) 
    + \la \nabla f(x_{t+1}), y_{t+1} - x_{t+1}\ra \rs - \frac{\alpha_t^2}{2A_{t}^2\lambda_t}\| g(x_{t+1}) -\nabla f(x_{t+1})\|^2  + \frac{\alpha_t}{A_t}\la g(x_{t+1}) -\nabla f(x_{t+1}), y_{t} - x_{t+1}\ra\\
    = \frac{f(x_t)}{A_{t - 1}} + \frac{ \bk^{t}_2}{2}\|y_{t+1} - y_t\|^2 
            +  \frac{ \bk^{t}_3}{6}\|y_{t+1} - y_t\|^3  - err_{t}^{v}-err_{t+1}^{x}\aa{- err_{t}^{\tau}}
           + \frac{\alpha_t}{A_t} \ls f(x_{t+1}) 
    + \la \nabla f(x_{t+1}), y_{t+1} - x_{t+1}\ra \rs \\
        \geq  \frac{ \bk^{t}_2}{2}\|y_{t+1} - y_t\|^2 
            +  \frac{ \bk^{t}_3}{6}\|y_{t+1} - y_t\|^3 - err_{t}^{v}-err_{t+1}^{x}\aa{- err_{t}^{\tau}}\\
         + \frac{1}{A_{t-1}}(f(x_{t+1}) + \langle \nabla f(x_{t+1}), x_t - x_{t+1}\rangle) + \frac{\alpha_t}{A_t} \ls f(x_{t+1}) 
    + \la \nabla f(x_{t+1}), y_{t+1} - x_{t+1}\ra \rs.
    \end{gather*}
    Next, we consider the sum of two linear models from the last inequality:
    \begin{gather*}
    \frac{1}{A_{t-1}}(f(x_{t+1}) + \langle \nabla f(x_{t+1}), x_t - x_{t+1}\rangle)
    + \frac{\alpha_t}{A_t} (f(x_{t+1}) 
    + \langle \nabla f(x_{t+1}), y_{t+1} - x_{t+1}\rangle)\\
        \stackrel{\eqref{eq:alphas}}{=} \frac{1 - \alpha_t}{A_t}f(x_{t+1}) + \frac{1 - \alpha_t}{A_t}\langle \nabla f(x_{t+1}), x_t - x_{t+1}\rangle + \frac{\alpha_t}{A_t}f(x_{t+1}) +\frac{\alpha_t}{A_t}\langle \nabla f(x_{t+1}), y_{t+1}- x_{t+1}\rangle 
    \end{gather*}
    \begin{gather*}
        {=} \frac{f(x_{t+1})}{A_{t}} + \frac{1 - \alpha_t}{A_t}\langle \nabla f(x_{t+1}), \frac{v_t - \alpha_t y_t}{1 - \alpha_t}- x_{t+1}\rangle + \frac{\alpha_t}{A_t}\langle \nabla f(x_{t+1}), y_{t+1}- x_{t+1} \rangle \\
        = \frac{f(x_{t+1})}{A_{t}} + \frac{1}{A_t}\langle \nabla f(x_{t+1}), v_t- x_{t+1}\rangle + \frac{\alpha_t}{A_t}\langle \nabla f(x_{t+1}), y_{t+1}- y_t \rangle.
    \end{gather*}
    As a result, by \eqref{eq:estimating_seq}, we get
    \begin{equation}
    \label{eq:x_bar_def}
        \begin{gathered}
            \psi^{\ast}_{t+1} = \psi_{t+1}(y_{t+1})  \geq \frac{f(x_{t+1})}{A_{t}} + \frac{1}{A_t}\langle \nabla f(x_{t+1}), v_t- x_{t+1}\rangle 
            + \frac{ \bk^{t}_2}{2}\|y_{t+1} - y_t\|^2 \\
            +  \frac{ \bk^{t}_3}{6}\|y_{t+1} - y_t\|^3 +\frac{\alpha_t}{A_t}\langle \nabla f(x_{t+1}), y_{t+1}- y_t  \rangle - err_{t}^{v}-err_{t+1}^{x}\aa{- err_{t}^{\tau}}.
        \end{gathered}
    \end{equation}
 
    To complete the induction step, we show, that the sum of all terms in the RHS except $\frac{f(x_{t+1})}{A_t}$  is non-negative (except err). 
    
    Lemma \ref{lm:scalar_lb_cases} provides the lower bound for $\langle \nabla f(x_{t+1}), v_t - x_{t+1}\rangle$. Let us consider the case when the minimum in the RHS of \eqref{eq:scalar_lb_cases} is attained at the first term. 
    By Lemma \ref{lem:dual} with the following choice of the parameters
    $$z = y_t - y_{t+1}, ~~ v = \frac{\alpha_t}{A_t}\nabla f(x_{t+1}), ~~ \theta = \bk_i^{t},$$
    we have 
    \begin{equation}\label{eq:1case}
        \frac{\bk_2^{t}}{2}\|y_t -y_{t+1}\|^2 + \frac{\alpha_t}{A_t}\langle \nabla f(x_{t+1}), y_{t+1} - y_t \rangle \geq - \frac{1}{2}\left( \frac{\|\frac{\alpha_t}{A_t}\nabla f(x_{t+1})\|^2}{\bk_2^{t}} \right).
    \end{equation}
    Hence,
    \begin{gather*}
         \frac{1}{A_t}\langle \nabla f(x_{t+1}), v_t - x_{t+1}\rangle + \frac{ \bk_2^{t}}{2}\|y_{t+1} - y_t\|^2 + \frac{\alpha_t}{A_t}\langle \nabla f(x_{t+1}), y_{t+1} - {y_t}\rangle \\
        \stackrel{\eqref{eq:1case}}{\geq}  \frac{1}{A_t}\langle \nabla f(x_{t+1}), v_t - x_{t+1}\rangle - \frac{\|\frac{\alpha_t}{A_t}\nabla f(x_{t+1})\|^i}{2\bk_2^{t}}  \\
        \stackrel{\eqref{eq:scalar_lb_cases}}{\geq}
        \frac{1}{A_t}\|\nabla f(x_{t+1})\|^2 \left( \frac{1}{4\bar{\delta}_t}\right) - \frac{\|\frac{\alpha_t}{A_t}\nabla f(x_{t+1})\|^2}{2\bk_2^{t+1}} - \frac{2}{A_t \bar{\delta}_{t}}\|g(v_t) - \nabla f(v_t) \|^2 - \frac{\tau^2}{A_t \bar{\delta}_t} \\
         \geq - \frac{2}{A_t \bar{\delta}_{t}}\|g(v_t) - \nabla f(v_t) \|^2 - \frac{\tau^2}{A_t \bar{\delta}_t}, 
    \end{gather*}
    where the last inequality holds by our choice of the parameters
    \begin{equation}\label{eq:bar_kappa}
        \bk_2^{{t+1}} \geq  \frac{2\bar{\delta}_t\al_t^2}{A_t} .
    \end{equation}

    Next, we consider the case when the minimum in the RHS of \eqref{eq:scalar_lb_cases} is achieved on the second term. Again, by Lemma \ref{lem:dual} with the same choice of $z, v$ and with $\theta = \frac{\bk_{3}^{t}}{2}$, we have
    \begin{equation}
        \label{eq:2case}
        \frac{\bk_{3}^{t}}{6}\|y_t -y_{t+1}\|^3 + \frac{\alpha_t}{A_t}\langle \nabla f(x_{t+1}), y_{t+1} - {y_t} \rangle \geq - \frac{2}{3}\left( \frac{2\|\frac{\alpha_t}{A_t}\nabla f(x_{t+1})\|^{3}}{\bk_{3}^{t}} \right)^\frac{1}{2}.
    \end{equation}

   Hence, we get
    \begin{gather*}
        \frac{1}{A_t}\langle \nabla f(x_{t+1}), v_t - x_{t+1}\rangle + \frac{\bk_{3}^{t}}{6}\|y_{t+1} - y_t\|^3 + \frac{\alpha_t}{A_t}\langle \nabla f(x_{t+1}), y_{t+1} - {y_t}\rangle\\
        \stackrel{\eqref{eq:2case}}{\geq}  \frac{1}{A_t}\langle \nabla f(x_{t+1}), v_t - x_{t+1}\rangle  - \frac{2}{3}\left( \frac{2\|\frac{\alpha_t}{A_t}\nabla f(x_{t+1})\|^{3}}{\bk_{3}^{t}} \right)^\frac{1}{2}\\
        \stackrel{\eqref{eq:scalar_lb_cases}}{\geq}
         \frac{1}{A_t}\|\nabla f(x_{t+1})\|^\frac{3}{2}
        \left( \frac{1}{3M}\right)^\frac{1}{2} - \frac{2}{3}\left( \frac{2\|\frac{\alpha_t}{A_t}\nabla f(x_{t+1})\|^{3}}{\bk_{3}^{t}} \right)^\frac{1}{2} - \frac{2}{A_t \bar{\delta}_{t}}\|g(v_t) - \nabla f(v_t) \|^2 - \frac{\tau^2}{A_t \bar{\delta}_t}\\
         \geq - \frac{2}{A_t \bar{\delta}_{t}}\|g(v_t) - \nabla f(v_t) \|^2 - \frac{\tau^2}{A_t \bar{\delta}_t},
    \end{gather*}
    where the last inequality holds by our choice of $\bk_{3}^{t}$:
    \begin{equation}\label{eq:bar_kappa_3}
        \bk_{3}^t \geq \frac{8M}{3} \frac{\alpha_t^{3}}{A_t}.
    \end{equation}
    
    As a result, we unite both cases and get 
        \begin{gather*}
            \psi^{\ast}_{t+1} \geq \frac{f(x_{t+1})}{A_{t}} + \frac{1}{A_t}\langle \nabla f(x_{t+1}), v_t- x_{t+1}\rangle 
            + \frac{ \bk^{t}_2}{2}\|y_{t+1} - y_t\|^2 \\
            +  \frac{ \bk^{t}_3}{6}\|y_{t+1} - y_t\|^3 +\frac{\alpha_t}{A_t}\langle \nabla f(x_{t+1}), y_{t+1}- y_t  \rangle - err_{t}^{v}-err_{t+1}^{x}\aa{- err_{t}^{\tau}}\\
            \geq \frac{f(x_{t+1})}{A_{t}} - \frac{\aa{2}}{A_t \bar{\delta}_{t}}\|g(v_t) - \nabla f(v_t) \|^2 - \aa{\frac{\tau^2}{A_t \bar{\delta}_{t}}}
            - err_{t}^{v}-err_{t+1}^{x}\\
            = \frac{f(x_{t+1})}{A_{t}} 
            - err_{t+1}^{v}-err_{t+1}^{x}\aa{- err_{t+1}^{\tau}}
        \end{gather*}

    To sum up, by our choice of the parameters $\bk_{i}^{t}$, $i=2,3$, we prove the induction step.
    
\end{proof}

Finally, we are in a position to prove the convergence rate theorem The proof uses the following technical assumption.
Let $R$ be such that
\begin{equation}
\label{R_teorem3}
\|x_0 - x^{\ast}\| \leq R.
\end{equation}

\begin{theorem}
    \label{thm:2_ord_app}
     Let Assumption \ref{as:lip} hold and $M \geq 4L_2$. Let Assumption \ref{as:2ord_stoch_grad_inexact_hess} hold. After $T \geq 1$ with parameters defined in~\eqref{eq:params} and 
            $\sigma_2 = \delta_2 = \max \limits_{t=1, \ldots, T} \delta_t^{v_{t-1}, x_{t}}$
            we get the following bound for the objective residual
            \begin{equation*}
            \begin{aligned}
               \E \left[ f(x_{T}) - f(x^{\ast}) \right]  
               &\leq 
              \frac{10 \tau R}{\sqrt{T+2}} + \frac{19 \sigma_1 R}{\sqrt{T+1}} + \frac{18 \delta_2 R^{2}}{(T+3)^{2}}  
                + \frac{20 L_2 R^{3}}{(T+1)^{3}}.
            \end{aligned}
            \end{equation*}
\end{theorem}

\begin{proof}  
        First of all, let us bound $A_T$.
        \begin{equation}
        \label{alpha_t}
            \alpha_t = \frac{3}{t+3}, ~ t \geq 1.
        \end{equation}
        Then, we have
        \begin{equation}\label{eq:A_t_bound}
            A_{T} = 
            \prod_{t=1}^{T}\left(1-\alpha_{t}\right)=\prod_{t=1}^{T} \frac{t}{t+3}=\frac{T !3 !}{(T+3) !}=\frac{6}{(T+1)(T+2)(T+3)}.
         \end{equation}
         And from \cite{agafonov2023inexact} we get
         \begin{equation}
            \label{eq:alpha_sum_bound}
             \sum_{t=0}^{T} \frac{A_{T} \alpha_{t}^{i}}{A_{t}} 
            \leq\frac{3^{i}}{(T+3)^{i-1}}
         \end{equation}
        
        From Lemmas \ref{lem:step} and \ref{lem:upper_seq}, we obtain that, for all $t \geq 1$,
    \begin{gather*}
       \frac{f(x_{t+1})}{A_t} 
            - err_{t+1}^{v}-err_{t+1}^{x} \aa{- err_{t+1}^{\tau}} \stackrel{\eqref{eq:lemma_ass}}{\leq} \psi_{t+1}^{\ast} 
            \leq \psi_{t+1}(x^{\ast}) \\
            \stackrel{\eqref{eq:acc_upper_bound}}{\leq}
       \frac{f(x^{\ast})}{A_{t}} + \frac{ \bk^{t}_2 + \lambda_t}{2}\|x^{\ast} - x_0\|^2  + \frac{\bk^{t}_{3}}{6}\|x^{\ast} - x_0\|^3 + err_{t+1}^{up}.
    \end{gather*}
    Next, we apply expectation
    \begin{equation} 
    \label{eq:final_sum}
        \E \left[ f(x_{T+1}) - f(x^{\ast}) \right]
        \leq 
        A_T \E \left[ \frac{\bk_2^T + \lambda_T}{2}R^2 + \frac{\bk_3^T}{6}R^3 + err_{T+1}^{up} + err_{T+1}^v + err_{T+1}^x + err_{T+1}^\tau \right].
    \end{equation}

    Let us choose 
    \begin{gather}
        \bar{\delta}_t = \delta_2 + \frac{\tau+\sigma}{R}(t+3)^{3/2}, \label{eq:delta}
        \\ 
        \lambda_t = \frac{\sigma}{R}(t+3)^{5/2}. 
        \label{eq:lambda}
    \end{gather}
    Then, we bound terms in \eqref{eq:final_sum} step by step. We start from deterministic terms.
    \begin{gather*}
        A_T \E \left[ \frac{\bk_2^T + \lambda_T}{2}R^2 + \frac{\bk_3^T}{6}R^3\right] 
        \stackrel{\eqref{eq:params}}{=} 
        \frac{18\bar{\delta}_TR^2}{(T+3)^2}
        + 
        \frac{6\lambda_T R^2}{(T+3)^3}
        + 
        \frac{72MR^3}{(T+3)^3}
        \\
        \stackrel{\eqref{eq:delta},\eqref{eq:lambda}}{=}
        \frac{18\tau R}{(T+3)^{1/2}}
        +  \frac{18\sigma R}{(T+3)^{1/2}} + \frac{18\delta_2 R^2}{(T+3)^2} + \frac{6\sigma R}{(T+3)^{1/2}} +  \frac{72MR^3}{(T+3)^3}.
    \end{gather*}   
    Now, we bound expectation of all error terms. Firstly, we consider $err_{T+1}^{up}$
    \begin{gather*}
        A_T \E \left[err_{T+1}^{up}\right] \stackrel{\eqref{eq:err_up}}{=}  A_T \E \left[ \sum \limits_{j = 0}^{T} \frac{\alpha_j}{A_j} \la g(x_{j+1}) - \nabla f(x_{j+1}) , x^{\ast} - x_{j+1}\ra \right] = 0.
    \end{gather*}
    Next, we bound $A_T \E \left[err_{T+1}^{v}\right]$
    \begin{gather*}
         A_T \E \left[err_{T+1}^{v}\right] 
         \stackrel{\eqref{eq:err_v}}{=} 
         A_T \E \left[ 
            \sum \limits_{j = 0}^{T}  \frac{\aa{2}}{A_j \bar{\delta}_j}\|g(v_j) - \nabla f(v_j) \|^2 
        \right]
        \leq 
        2\sigma^2 \sum \limits_{j=0}^{T} \frac{A_T}{A_j \bar{\delta}_j}\\
        \stackrel{\eqref{alpha_t}, \eqref{eq:delta}}{=} \frac{2\sigma R}{3^{3/2}} \sum \limits_{j=0}^{T} \frac{A_T \alpha_j^{3/2}}{A_j} \stackrel{\eqref{eq:alpha_sum_bound}}{\leq} \frac{2\sigma R}{(T+3)^{1/2}}
    \end{gather*}
    
    Now we calculate $A_T \E \left[err_{T+1}^{x}\right]$
    
    \begin{gather*}
        A_T \E \left[err_{T+1}^{x}\right] 
        \stackrel{\eqref{eq:err_x}}{=} 
        A_T \E \left[ 
            \sum \limits_{j = 0}^{T}  \frac{\alpha_j^2}{2A_j^2\lambda_j} \|g(x_{j+1})-\nabla f(x_{j+1}) \|^2 
            + 
            \sum \limits_{j = 0}^{T}  \frac{\alpha_j}{A_j} \la g(x_{j+1}) -\nabla f(x_{j+1}), y_{j} - x_{j+1}\ra
        \right]
        \end{gather*}
    
    \begin{gather*} 
        \stackrel{\eqref{eq:lambda}}{=} 
        \frac{\sigma R}{2} \sum \limits_{j = 0}^{T} \frac{A_T \alpha_j^2}{A_j^2 (j+3)^{5/2}} 
        \stackrel{\eqref{alpha_t}}{=} 
        \frac{\sigma R}{3^{5/2} 2} \sum \limits_{j = 0}^{T} \frac{A_T \alpha_j^{9/2}}{A_j^2} \leq \frac{\sigma R}{3^{5/2} 2A_T} \sum \limits_{j = 0}^{T} \frac{A_T \alpha_j^{9/2}}{A_j}
        \\
        \stackrel{\eqref{eq:alpha_sum_bound}}{\leq} \frac{3\sigma R}{ 2A_T (T+3)^{7/2}} 
        \stackrel{\eqref{eq:A_t_bound}}{\leq} \frac{\sigma R}{ 4 (T+3)^{1/2}}.
    \end{gather*}

    Finally, we consider $err_{T+1}^\tau$
    \begin{gather*}
        A_T \E \left[err_{T+1}^{\tau}\right] 
        \stackrel{\eqref{eq:err_tau}}{=} 
        \sum \limits_{j = 0}^{T}  \frac{A_T\tau^2}{A_j \bar{\delta}_j} 
        \stackrel{\eqref{eq:delta}}{=} 
        \frac{\tau R}{3^{3/2}} \sum \limits \frac{A_T \alpha_t^{3/2}}{A_j} 
        \stackrel{\eqref{eq:alpha_sum_bound}}{\leq}
        \frac{\tau R}{(T+3)^{1/2}}.
    \end{gather*}

    Combining all bounds from above we achieve convergence rate
    \begin{gather*}
         \E \left[ f(x_{T+1}) - f(x^{\ast}) \right]
        \leq \frac{19 \tau R}{(T+3)^{1/2}} + \frac{27 \sigma R}{(T+3)^{1/2}} + \frac{18\delta_2 R^2}{(T+3)^2} + \frac{72 MR^3}{(T+3)^3}.
    \end{gather*}

\end{proof}

The case of stochastic Hessian (Theorem~\ref{thm:acc_convergence} under Assumption~\ref{as:2ord_stoch_grad_inexact_hess}) can be obtained in the same way by taking expectation in Lemma~\ref{lm:scalar_lb_cases}.

\section{Proof of Theorem \ref{thm:acc_convergence_p_ord}}
Algorithm~\ref{alg:inexact_acc_detailed_p_ord} is the tensor generalization of Algorithm~\ref{alg:inexact_acc_detailed}. So, we will follow the same steps as in Appendix~\ref{sec:proof_2_ord}.

First of all, we provide tensor counterpart of Lemmas~\ref{lem:upper_seq},~\ref{lem:cubic_bound_acc},~\ref{lm:scalar_lb_cases}. The proofs directly follow the proofs of Lemmas for $2$-nd order case.

\begin{lemma}\label{lem:upper_seq_p}
    For convex function $f(x)$ and $\psi_t(x)$, we have
    \begin{equation}
        \label{eq:acc_upper_bound_p}
        \psi_t(x^{\ast}) \leq \frac{f(x^{\ast})}{A_{t - 1}}  + \frac{ \bk_2^{t} + \lambda_t}{2} \|x^{\ast} - x_0\|^2   + \sum \limits_{i=3}^{p+1}\frac{\bk^{t}_{i}}{i!}\|x^{\ast} - x_0\|^i + err_{t}^{up},
        \end{equation}
    where
    \begin{equation}
        err^{up}_t = \sum \limits_{j = 0}^{t - 1} \frac{\alpha_j}{A_j} \la g(x_{j+1}) - \nabla f(x_{j+1}) , x^{\ast} - x_{j+1}\ra
    \end{equation}
\end{lemma}

\begin{lemma}
\label{lem:cubic_bound_acc_p}
For the function $f(x)$ with $L_p$-Lipschitz-continuous Hessian and $G_i(x_t)$ is $ \delta_i^t$-inexact $i$-th order derivative for $v_t,x_{t+1} \in \R^d$ we have
\begin{equation}
\label{eq:grad_model_bound_p}
    \|\nabla \phi_{v_t}(x_{t+1}) - \nabla f(x_{t+1})\| 
        \leq
         \sum \limits_{i=2}^p\frac{\delta_i^t}{(i-1)!} \|x_{t+1} - v_t\|^i + \frac{L_p}{p!}\|x_{t+1} - v_t\|^{p+1} +  \|g(v_t) - \nabla f(v_t) \|,
\end{equation}
where we denote $\delta_2^t = \delta_t^{v_t, x_{t+1}} $ to simplify the notation.
\end{lemma}

\begin{lemma}
    \label{lm:scalar_lb_cases_p}
    Let $\{x_t, v_t\}_{t \geq 1}$ be generated by Algorithm \ref{alg:inexact_acc_detailed}. Then, for any $\bar{\delta}_t\geq 4{\delta}_2^t$, $\eta_i \geq 4$, $M \geq \tfrac{2}{p}L_p$  and $x_{t+1} = S_{M,\bar{\delta}_t}(v_{t})$, the following holds
    \begin{equation}
    \begin{gathered}\label{eq:scalar_lb_cases_p}
        \aa{\frac{\tau^2}{\delta_t}} + \frac{\aa{2}}{\delta_{t}}\|g(v_t) - \nabla f(v_t) \|^2 + \langle \nabla f(x_{t+1}), v_t - x_{t+1} \rangle 
        \\
        \geq 
         \min \left\{
            \frac{2}{p}\|\nabla f(x_{t+1})\|^\frac{p+1}{p} \ls \frac{(p-1)!}{M}\rs^\frac{1}{p};
            ~ 
                 \tfrac{1}{4\bar{\delta}_t }
        \|\nabla f(x_{t+1})\|^2
            ; 
        \right.
        \\
        \left.
             \min \limits_{i=3,\ldots, p} \ls
                 \frac{2}{p}\|\nabla f(x_{t+1})\|^\frac{i}{i-1}\ls \frac{(i-1)!}{\eta_i \delta_i^t}\rs^\frac{1}{i-1}
            \rs
        \right\}.
    \end{gathered}
    \end{equation}
\end{lemma}

\begin{proof}
    For simplicity, we denote $r_{t+1} = \|x_{t+1} - v_t\|$ and
    \begin{equation}
        \label{eq:zeta_p}
        \zeta_{t+1} =  \bar{\delta}_t +\sum \limits_{i=3}^p \frac{\eta_i \delta_i}{(i-1)!}\|x_{t+1} - v_t\|^{i-2} +  \frac{M}{(p-1)!}\|x_{t+1} - v_t\|^{p-1}.
    \end{equation}
    
    By Definition \ref{def:inexact_sub_p} for $x_{t+1} = S_p^{M,\bar{\delta}_t, \tau}(v_{t})$

    \begin{equation}
    \begin{gathered}
        \tau^2 \geq \left\|\nabla \phi_{v_t, p}(x_{t+1}) + \bar{\delta}_t(x_{t+1} - v_t) + \sum \limits_{i=3}^p \frac{\eta_i \delta_i^t}{(i-1)!}\|x_{t+1} - v_t\|^{i-2}(x_{t+1} - v_t)\right. 
        \\
        \left. + \frac{M}{(p-1)!}\|x_{t+1} - v_t\|^{p-1}(x_{t+1} - v_t)\right\|^2 
        \\
        \stackrel{\eqref{eq:zeta_p}}{=}  \|\nabla \phi_{v_t, p} + \zeta_{t+1}(x_{t+1} - v_t)\|^2 .
    \end{gathered}
    \end{equation}

    From inexact solution of subproblem we get
    \begin{equation}
    \begin{gathered}
        \|\nabla f(x_{t+1}) + \zeta_{t+1}(x_{t+1} - v_t)\|^2 
        \\
        \leq \|\nabla \phi_{v_t, p}(x_{t+1}) - \nabla f(x_{t+1}) - \nabla \phi_{v_t, p}(x_{t+1}) - \zeta_{t+1}(x_{t+1}-v_t)\|^2
        \\
        \leq 2 \|\nabla \phi_{v_t, p}(x_{t+1}) - \nabla f(x_{t+1})\|^2 + 2\|\nabla \phi_{v_t, p}(x_{t+1}) + \zeta_{t+1}(x_{t+1}-v_t) \|^2
        \\
        \leq 2 \|\nabla \phi_{v_t, p}(x_{t+1}) - \nabla f(x_{t+1})\|^2 + 2\tau^2.
    \end{gathered}
    \end{equation}

    Next, from previous inequality and Lemma \ref{lem:cubic_bound_acc_p} 
    \begin{equation}
        \begin{gathered}
            4\|g(v_t) - \nabla f(v_t)\|^2 + 4 \left(\sum \limits_{i=2}^p \frac{\delta_i^t}{(i-1)!}r_{t+1}^i + \frac{L_p}{p!}r_{t+1}^{p+1} \right)^2 + 2\tau^2
            \\
            \geq 2\|\nabla \phi_{v_t, p}(x_{t+1} - v_t)\|^2 + 2\tau^2 \\
            \geq \|\nabla f(x_{t+1}) + \zeta_{t+1}(x_{t+1} - v_t)\|^2
            \\
            \geq \|\nabla f(x_{t+1})\|^2 + 2\zeta_{t+1} \langle \nabla f(x_{t+1}), x_{t+1} - v_t \rangle + \zeta_{t+1}^2\|x_{t+1} - v_t\|^2.
        \end{gathered}
    \end{equation}

    Hence, 
    \begin{gather*}
        \frac{\tau^2}{\zeta_{t+1}}  + \frac{2}{\zeta_{t+1}}\|g(v_t) - \nabla f(v_t)\|^2  + \langle \nabla f(x_{t+1}), v_t - x_{t+1} \rangle \geq
        \\
        \frac{1}{2\zeta_{t+1}}\|\nabla f(x_{t+1})\|^2 
        +\frac{1}{2\zeta_{t+1}} \ls \zeta_{t+1}^2 - 4\left(\sum \limits_{i=2}^p \frac{\delta_i^t}{(i-1)!}r_{t+1}^{i-2} + \frac{L_p}{p!}r_{t+1}^{p-1} \right)^2\rs r_{t+1}^2,
    \end{gather*}
    and finally by using definition of $\zeta_{t+1}$, we get
    \begin{gather*}
        \frac{\tau^2}{\bar{\delta}_t}  + \frac{2}{\bar{\delta}_t}\|g(v_t) - \nabla f(v_t)\|^2  + \langle \nabla f(x_{t+1}), v_t - x_{t+1} \rangle \geq
        \\
        \frac{1}{2\zeta_{t+1}}\|\nabla f(x_{t+1})\|^2 
        +\frac{1}{2\zeta_{t+1}} \ls \zeta_{t+1}^2 - 4\left(\sum \limits_{i=2}^p \frac{\delta_i^t}{(i-1)!}r_{t+1}^{i-2} + \frac{L_p}{p!}r_{t+1}^{p-1} \right)^2\rs r_{t+1}^2.
    \end{gather*}
    Next, we consider $p$ cases depending on which term dominates in $\zeta_{t+1}$.
    \begin{itemize}[leftmargin=10pt,nolistsep]
    \item If 
    $\bar{\delta}_t \geq \sum \limits_{i=3}^p \frac{\eta_i \delta_i^t}{(i-1)!}r_{t+1}^{i-2} + \frac{M}{(p-1)!}r_{t+1}^{p-1}$, then we get the following bound
    \begin{gather*}
        \frac{\tau^2}{\bar{\delta}_t} + \frac{2}{\bar{\delta}_{t}}\|g(v_t) - \nabla f(v_t) \|^2 + 
        \langle \nabla f(x_{t+1}), v_t - x_{t+1} \rangle 
        \\
        \geq
        \frac{1}{2\zeta_{t+1}}\|\nabla f(x_{t+1}) \|^2
        + \frac{1}{2\zeta_{t+1}}\ls \zeta_{t+1}^2 - 4\left(\sum \limits_{i=2}^p \frac{\delta_i^t}{(i-1)!}r_{t+1}^{i-2} + \frac{L_p}{p!}r_{t+1}^{p-1} \right)^2 \rs r_{t+1}^2
        \\
        \geq \tfrac{1}{2p\bar{\delta}_t }
        \|\nabla f(x_{t+1})\|^2.
    \end{gather*}

    \item If $\frac{\eta_i \delta_i^t}{(i-1)!}r_{t+1}^{i-1} \geq \bar{\delta}_t + \sum \limits_{j=3, j\neq i}^p \frac{\eta_j \delta_j}{(j-1)!}r_{t+1}^{j-1} + \frac{M}{(p-1)!}r_{t+1}^{p-1}$ for $i, ~ 3 \leq i \leq p$, we get

    \begin{gather*}
        \frac{\tau^2}{\bar{\delta}_t} + 
        \frac{2}{\bar{\delta}_{t}}\|g(v_t) - \nabla f(v_t) \|^2 
        + 
        \langle \nabla f(x_{t+1}), v_t - x_{t+1} \rangle 
        \\
        \geq
        \frac{1}{2\zeta_{t+1}}\|\nabla f(x_{t+1}) \|^2
        + 
        \frac{1}{2\zeta_{t+1}}\ls \zeta_{t+1}^2 - 4\left(\sum \limits_{i=2}^p \frac{\delta_i^t}{(i-1)!}r_{t+1}^i + \frac{L_p}{p!}r_{t+1}^{p+1} \right)^2 \rs r_{t+1}^2
        \end{gather*}
    
        \begin{gather*}
        \geq
        \frac{(i-1)! \|\nabla f(x_{t+1})\|^2}{p\eta_i \delta_i^t r_{t+1}^{i-1}} 
        +
        \ls \bar{\delta}_t - 2\delta_2^t + \sum \limits_{i=3}^p \frac{\eta_i\delta_i^t - 2\delta_i^t}{(i-1)!}r_{t+1}^{i-2} + \frac{pM - 2L_p}{p!}r_{t+1}^{p-1}\rs \\
        \times 
        \ls \bar{\delta}_t + 2\delta_2^t + \sum \limits_{i=3}^p \frac{\eta_i\delta_i^t + 2\delta_i^t}{(i-1)!}r_{t+1}^{i-2} 
        + 
        \frac{pM + 2L_p}{p!}r_{t+1}^{p-1}\rs \frac{r_{t+1}^2}{\zeta_{t+1}}
        \\
        \geq 
        \frac{(i-1)! \|\nabla f(x_{t+1})\|^2}{p\eta_i \delta_i^t r_{t+1}^{i-2}} 
        +
        \frac{\eta_i \delta_i^t - 2\delta_i^t}{(i-1)!}  \frac{\eta_i \delta_i^t + 2\delta_i^t}{(i-1)!}  \frac{r_{t+1}^{i}(i-1)!}{p \eta_i \delta_i^t} 
        \\
        \geq 
        \frac{(i-1)! \|\nabla f(x_{t+1})\|^2}{p\eta_i \delta_i^t r_{t+1}^{i-2}} + \frac{\eta_i \delta_i^t}{(i-1)! p}r_{t+1}^i.
        \\
        \geq \frac{2}{i} \ls \frac{(i-2)(i-1)!\|\nabla f(x_{t+1})\|^2}{p\eta_i\delta_i^t}\rs^\frac{i}{2(i-1)}\ls \frac{i \eta_i \delta_i^t}{(i-1)! p}\rs^\frac{i-2}{2(i-1)} 
        \\
        \geq \frac{2}{p}\|\nabla f(x_{t+1})\|^\frac{i}{i-1}\ls \frac{(i-1)!}{\eta_i \delta_i^t}\rs^\frac{1}{i-1},
    \end{gather*}
    where we used $\frac{\alpha}{(i-2)r^{i-2}} + \frac{\beta r^i}{i} \geq \frac{2}{i} \alpha^\frac{i}{2(i-1)}\beta^\frac{i-2}{2(i-1)}$.
    
    \item If $\frac{M}{(p-1)!}r_{t+1}^{p-1} \geq \bar{\delta}_t + \sum \limits_{i=3}^p \frac{\eta_i \delta_i^t}{(i-1)!}r_{t+1}^{i-1}$, then similarly to previous case, we get
    \begin{gather*}
        \frac{\tau^2}{\bar{\delta}_t} + 
        \frac{2}{\bar{\delta}_{t}}\|g(v_t) - \nabla f(v_t) \|^2 
        + 
        \langle \nabla f(x_{t+1}), v_t - x_{t+1} \rangle 
        \\
        \geq 
        \frac{1}{2\zeta_{t+1}}\|\nabla f(x_{t+1}) \|^2
        + 
        \frac{1}{2\zeta_{t+1}}\ls \zeta_{t+1}^2 - 4\left(\sum \limits_{i=2}^p \frac{\delta_i^t}{(i-1)!}r_{t+1}^i + \frac{L_p}{p!}r_{t+1}^{p+1} \right)^2 \rs r_{t+1}^2
        \\
        \frac{(p-1)!}{pMr_{t+1}^{p-1}}\|\nabla f(x_{t+1})\|^2 + \frac{pM}{2p (p!)}r_{t+1}^{p+1}
        \\
        \geq \frac{2}{p}\|\nabla f(x_{t+1})\|^\frac{p+1}{p} \ls \frac{(p-1)!}{M}\rs^\frac{1}{p}.
    \end{gather*}
    
    \end{itemize}
        
\end{proof}

Next we will need technical Lemmas~\ref{lm:argmin},~\ref{lem:dual}.

\begin{lemma}\label{lem:step_p}
    Let $\{x_t, y_t\}_{t \geq 1}$ be generated by Algorithm \ref{alg:inexact_acc_detailed_p_ord}. 
    Then
    \begin{equation}\label{eq:lemma_ass_p}
    \psi_t^{\ast}  = \min \limits_x \psi_t(x)   \geq \frac{f(x_t)}{A_{t-1}} - err_{t}^{v}-err_{t}^{x}
    \aa{-err_{t}^{\tau}},
    \end{equation}
    where 
    \begin{equation}\label{eq:err_v_p}
    err_{t}^{v}  = \sum \limits_{j = 0}^{t - 1}  \frac{\aa{2}}{A_j \bar{\delta}_j}\|g(v_j) - \nabla f(v_j) \|^2 ,
    \end{equation}
    \aa{\begin{equation}\label{eq:err_tau}
    err_{t}^{\tau}  = \sum \limits_{j = 0}^{t - 1}  \frac{\tau^2}{A_j \bar{\delta}_j} ,
    \end{equation}}
    and
     \begin{equation}\label{eq:err_x_p}
    err_{t}^{x}  = \sum \limits_{j = 0}^{t-1}  \frac{\alpha_j^2}{2A_j^2\lambda_j} \|g(x_{j+1})-\nabla f(x_{j+1}) \|^2 + \sum \limits_{j = 0}^{t-1}  \frac{\alpha_j}{A_j} \la g(x_{j+1}) -\nabla f(x_{j+1}), y_{j} - x_{j+1}\ra .
    \end{equation}
\end{lemma}
\begin{proof}
    We prove Lemma by induction. Let us start with $t=0$, we define $A_{-1}$ such that $\tfrac{1}{A_{-1}}=0$. Then $\tfrac{f(x_0)}{A_{-1}}=0$ and $\psi_0^{\ast}=0$, hence, $\psi_0^{\ast}\geq \tfrac{f(x_0)}{A_{-1}}$. Let us assume that \eqref{eq:lemma_ass_p} is true for $t$
    and show that \eqref{eq:lemma_ass_p} is true for $t+1$.
    
    Following the steps of Lemma~\ref{lem:step} for $p$-th order case, we get
    \begin{equation}
    \label{eq:x_bar_def_p}
        \begin{gathered}
            \psi^{\ast}_{t+1} = \psi_{t+1}(y_{t+1})  \geq \frac{f(x_{t+1})}{A_{t}} + \frac{1}{A_t}\langle \nabla f(x_{t+1}), v_t- x_{t+1}\rangle 
            \\
            + \sum \limits_{i = 2}^{p+1} \left(\frac{1}{2} \right)^{i - 2} \frac{ \bk_i^{t}}{(i - 1)!}d_i(y_{t+1} - y_t) +\frac{\alpha_t}{A_t}\langle \nabla f(x_{t+1}), y_{t+1}- y_t  \rangle - err_{t}^{v}-err_{t+1}^{x}\aa{- err_{t}^{\tau}}.
        \end{gathered}
    \end{equation}

    To complete the induction step we will show, that the sum of all terms in the RHS except $\frac{f(x_{t+1})}{A_t}$ and error terms is non-negative. 
    
    Lemma \ref{lm:scalar_lb_cases_p} provides the lower bound for $\langle \nabla f(x_{t+1}), v_t - x_{t+1}\rangle$. Let us consider the case when the minimum in the RHS of \eqref{eq:scalar_lb_cases_p} is attained at the  term with particular $i = 3, \ldots, p$. 
    By Lemma \ref{lem:dual} with the following choice of the parameters
    $$z = y_t - y_{t+1}, ~~ v_t = \frac{\alpha_t}{A_t}\nabla f(x_{t+1}), ~~ \theta = \left(\frac{1}{2}\right)^{i-2}\frac{\bk_i^{t}}{(i-1)!},$$
    we have 
    \begin{equation}\label{eq:1case_p}
        \frac{1}{i}\left( \frac{1}{2} \right)^{i - 2}\frac{\bk_i^{t}}{(i-1)!}\|y_t -y_{t+1}\|^i + \frac{\alpha_t}{A_t}\langle \nabla f(x_{t+1}), y_{t+1} - y_t \rangle \geq - \frac{i - 1}{i}\left( \frac{\|\frac{\alpha_t}{A_t}\nabla f(x_{t+1})\|^i}{\left(\frac{1}{2}\right)^{i-2} \frac{\bk_i^{{t}}}{(i-1)!}} \right)^\frac{1}{i - 1}.
    \end{equation}
    Hence,
    \begin{gather*}
        \frac{f(x_{t+1})}{A_t} + \frac{1}{A_t}\langle \nabla f(x_{t+1}), v_t - x_{t+1}\rangle + \left(\frac{1}{{2}} \right)^{i - 2} \frac{ \bk_i^{t}}{(i - 1)!}d_i(y_{t+1} - y_t) + \frac{\alpha_t}{A_t}\langle \nabla f(x_{t+1}), y_{t+1} - {y_t}\rangle \\
        \stackrel{\eqref{eq:1case}}{\geq} \frac{f(x_{t+1})}{A_t} + \frac{1}{A_t}\langle \nabla f(x_{t+1}), v_t - x_{t+1}\rangle - \frac{i - 1}{i}\left( \frac{\|\frac{\alpha_t}{A_t}\nabla f(x_{t+1})\|^i}{\left(\frac{1}{{2}}\right)^{i-{2}} \frac{\bk_i^{{t}}}{{(i-1)}!}} \right)^\frac{1}{i - 1} 
        \\
        \stackrel{\eqref{eq:scalar_lb_cases_p}}{\geq}
        \frac{f(x_{t+1})}{A_t} + \frac{2}{p}\|\nabla f(x_{t+1})\|^\frac{i}{i-1}\ls \frac{(i-1)!}{\eta_i \delta_i^t}\rs^\frac{1}{i-1} - \frac{i - 1}{i}\left( \frac{\|\frac{\alpha_t}{A_t}\nabla f(x_{t+1})\|^i}{\left(\frac{1}{{2}}\right)^{i-{2}} \frac{\bk_i^{{t}}}{{(i-1)}!}} \right)^\frac{1}{i - 1}
        \\
         \geq \frac{f(x_{t+1})}{A_t}, 
    \end{gather*}
    where the last inequality holds by our choice of the parameters
    \begin{equation}\label{eq:bar_kappa_p}
        \bk_i^{{t}} \geq \frac{p^{i-1}}{2}\frac{\alpha_t^i}{A_t}\eta_i \delta_i^t.
    \end{equation}

    Let us consider the case when the minimum in the RHS of \eqref{eq:scalar_lb_cases_p} is achieved on the second term. Following similar steps, we get
    \begin{equation}\label{eq:bar_kappa_p2}
        \bk_2^{t} \geq \frac{2p\alpha_t^2}{A_t}\bar{\delta}_t.
    \end{equation}
    
    Next, we consider the case when the minimum in the RHS of \eqref{eq:scalar_lb_cases_p} is achieved on the first term. Again, by Lemma \ref{lem:dual} with the same choice of $z, v$ and with $\theta = {\left(\frac{1}{2}\right)^{p-1}\frac{\bk_{p+1}^{t-1}}{(p-1)!}}$, we have
    \begin{equation}
        \label{eq:2case_p}
        {\frac{1}{p}}\left( \frac{1}{{2}} \right)^{p - 1}\frac{\bk_{p+1}^{t-1}}{{ {p}!}}\|y_t -y_{t+1}\|^ {p+1} + \frac{\alpha_t}{A_t}\langle \nabla f(x_{t+1}), y_{t+1} - {y_t} \rangle \geq - \frac{p}{p+1}\left( \frac{\|\frac{\alpha_t}{A_t}\nabla f(x_{t+1})\|^{p+1}}{{\left(\frac{1}{2}\right)^{p-1} \frac{\bk_{p+1}^{t}}{(p-1)!}}} \right)^\frac{1}{p}.
    \end{equation}

   Hence, we get
    \begin{gather*}
        \frac{f(x_{t+1})}{A_t} + \frac{1}{A_t}\langle \nabla f(x_{t+1}), v_t - x_{t+1}\rangle + \frac{1}{2^{p-1}}\frac{\bk_{p+1}^{t-1}}{ {p}!}d_{p+1}(y_{t+1} - y_t) + \frac{\alpha_t}{A_t}\langle \nabla f(x_{t+1}), y_{t+1} - {y_t}\rangle
        \\
        \stackrel{\eqref{eq:2case_p}}{\geq} \frac{f(x_{t+1})}{A_t} + \frac{1}{A_t}\langle \nabla f(x_{t+1}), v_t - x_{t+1}\rangle  - \frac{p}{p+1}\left( \frac{\|\frac{\alpha_t}{A_t}\nabla f(x_{t+1})\|^{p+1}}{{\left(\frac{1}{2}\right)^{p-1} \frac{\bk_{p+1}^{t}}{ {p}!}}} \right)^\frac{1}{p}
        \end{gather*}
    
        \begin{gather*}
        \stackrel{\eqref{eq:scalar_lb_cases_p}}{\geq}
        \frac{f(x_{t+1})}{A_t} + \frac{1}{A_t}\frac{2}{p}\|\nabla f(x_{t+1})\|^\frac{p+1}{p} \ls \frac{(p-1)!}{M}\rs^\frac{1}{p} -  \frac{p}{p+1}\left( \frac{\|\frac{\alpha_t}{A_t}\nabla f(x_{t+1})\|^{p+1}}{{\left(\frac{1}{2}\right)^{p-1} \frac{\bk_{p+1}^{t}}{ {p}!}}} \right)^\frac{1}{p}\\
         \geq \frac{f(x_{t+1})}{A_t},
    \end{gather*}
    where the last inequality holds by our choice of $\bk_{p+1}^{t}$:
    \begin{equation}\label{eq:bar_kappa_p+1}
        \bk_{p+1}^{t} \geq \frac{(p+1)^{p+1}}{2}\frac{\alpha_t^{p+1}}{A_t}M.
    \end{equation}
    
    To sum up, by our choice of the parameters $\bk_{i}^{t}$, $i=2,...,p$, we obtain
    $$
        \psi_{t+1}^{\ast} \geq \frac{f(x_{t+1})}{A_t}- err_{t+1}^{v}-err_{t+1}^{x}- err_{t+1}^{\tau}.
    $$
\end{proof}

\begin{theorem}
    \label{thm:p_ord_app}
     Let Assumption \ref{as:lip_p} hold and $M \geq \frac{2}{p}L_p$. Let Assumption \ref{as:p_ord_stoch_grad_inexact_hess} hold. After $T \geq 1$ with parameters defined in~\eqref{eq:p_ord_params} and 
            $\sigma_2 = \delta_2 = \max \limits_{t=1, \ldots, T} \delta_t^{v_{t-1}, x_{t}}$
            we get the following bound for the objective residual
            \begin{equation*}
            \begin{aligned}
               \E \left[ f(x_{T}) - f(x^{\ast}) \right]  
               &\leq 
              \frac{2(p+1)^3\tau R}{(T+p+1)^{1/2}} + \frac{3(p+1)^3\sigma R}{(T+p+1)^{1/2}} + \sum \limits_{i=3}^p \frac{2(p+1)^{2i-1}\delta_i R^i}{i! (T+p+1)^i} + \frac{(p+1)^{2(p+1)}}{(p+1)!} \frac{MR^{p+1}}{(T+p+1)^{p+1}}
            \end{aligned}
            \end{equation*}
\end{theorem}

\begin{proof}
     First of all, let us bound $A_T$.
        \begin{equation}
        \label{alpha_t_p}
            \alpha_t = \frac{p+1}{t+p+1}, ~ t \geq 1.
        \end{equation}
        Then, we have
        \begin{equation}\label{eq:A_t_bound_p}
            \frac{1}{(T+p+1)^{p+1}} \leq A_{T} \leq \frac{(p+1)!}{(T+1)^{p+1}}.
         \end{equation}
         And from \cite{agafonov2023inexact} we get
         \begin{equation}
            \label{eq:alpha_sum_bound_p}
             \sum_{t=0}^{T} \frac{A_{T} \alpha_{t}^{i}}{A_{t}} 
            \leq\frac{(p+1)^{i}}{(T+p+1)^{i-1}}
         \end{equation}
        
        From Lemmas \ref{lem:step_p} and \ref{lem:upper_seq_p}, we obtain that, for all $t \geq 1$,
    \begin{gather*}
       \frac{f(x_{t+1})}{A_t} 
            - err_{t+1}^{v}-err_{t+1}^{x} \aa{- err_{t+1}^{\tau}} \stackrel{\eqref{eq:lemma_ass_p}}{\leq} \psi_{t+1}^{\ast} 
            \leq \psi_{t+1}(x^{\ast}) \\
            \stackrel{\eqref{eq:acc_upper_bound_p}}{\leq}
       \frac{f(x^{\ast})}{A_{t}} + \frac{ \bk^{t}_2 + \lambda_t}{2}\|x^{\ast} - x_0\|^2  + \frac{\bk^{t}_{3}}{6}\|x^{\ast} - x_0\|^3 + err_{t+1}^{up}.
    \end{gather*}
    Next, we apply expectation
    \begin{equation} 
    \label{eq:final_sum_p}
        \E \left[ f(x_{T+1}) - f(x^{\ast}) \right]
        \leq 
        A_T \E \left[ \frac{\bk_2^T + \lambda_T}{2}R^2 + \sum \limits_{i=3}^{p+1} \frac{\bk_i^T}{i!}R^i + err_{T+1}^{up} + err_{T+1}^v + err_{T+1}^x + err_{T+1}^\tau \right].
    \end{equation}

    Let us choose 
    \begin{gather}
        \bar{\delta}_t = \delta_2 + \frac{\tau+\sigma}{R}(t+p+1)^{3/2}, \label{eq:delta_p}
        \\ 
        \lambda_t = \frac{\sigma}{R}(t+p + 1)^{p+1/2}. 
        \label{eq:lambda_p}
    \end{gather}
    Then, we bound terms in \eqref{eq:final_sum_p} step by step.

    We start from deterministic terms.
    
    \begin{gather*}
        A_T \E \left[ 
            \frac{\bk_2^T + \lambda_T}{2}R^2 
            + 
            \sum \limits_{i=3}^p \frac{\bk_i^T}{i!}R^i 
            + 
            \frac{\bk_{p+1 }^T}{(p+1)!}R^{p+1}
        \right] 
        \\ 
        \stackrel{\eqref{eq:p_ord_params}}{=} 
        p\alpha_T^2 \bar{\delta}_T R^2 + \sum \limits_{i=3}^p \frac{4p^{i-1}}{i! 2}\alpha_T^i \delta_i R^i+ \frac{(p+1)^{p+1}}{(p+1)!}\alpha_T^{p+1}MR^{p+1}
        \\
        \stackrel{\eqref{eq:delta_p},\eqref{eq:lambda_p}}{=} \frac{(p+1)^3(\tau+ \sigma)R}{(T+p+1)^{1/2}} + \frac{(p+1)^3 \delta_2 R^2}{(T+p+1)^2} + \sum \limits_{i=3}^p \frac{2(p+1)^{2i-1}\delta_i R^i}{i! (T+p+1)^i} + \frac{(p+1)^{2(p+1)}}{(p+1)!} \frac{MR^{p+1}}{(T+p+1)^{p+1}}.
    \end{gather*}   
    
    Now, we bound expectation of all error terms. Firstly, we consider $err_{T+1}^{up}$
    
    \begin{gather*}
        A_T \E \left[err_{T+1}^{up}\right] = A_T \E \left[ \sum \limits_{j = 0}^{T} \frac{\alpha_j}{A_j} \la g(x_{j+1}) - \nabla f(x_{j+1}) , x^{\ast} - x_{j+1}\ra \right] = 0.
    \end{gather*}
    
    Next, we bound $A_T \E \left[err_{T+1}^{v}\right]$
    
    \begin{gather*}
         A_T \E \left[err_{T+1}^{v}\right] 
         = 
         A_T \E \left[ 
            \sum \limits_{j = 0}^{T}  \frac{\aa{2}}{A_j \bar{\delta}_j}\|g(v_j) - \nabla f(v_j) \|^2 
        \right]
        \leq 
        2\sigma^2 \sum \limits_{j=0}^{T} \frac{A_T}{A_j \bar{\delta}_j}\\
        \stackrel{\eqref{alpha_t_p}, \eqref{eq:delta_p}}{=} \frac{2\sigma R}{(p+1)^{3/2}} \sum \limits_{j=0}^{T} \frac{A_T \alpha_j^{3/2}}{A_j} \stackrel{\eqref{eq:alpha_sum_bound_p}}{\leq} \frac{2\sigma R}{(T+p+1)^{1/2}}
    \end{gather*}
    
    Now we calculate $A_T \E \left[err_{T+1}^{x}\right]$
    \begin{gather*}
        A_T \E \left[err_{T+1}^{x}\right] 
        \stackrel{\eqref{eq:err_x}}{=} 
        A_T \E \left[ 
            \sum \limits_{j = 0}^{T}  \frac{\alpha_j^2}{2A_j^2\lambda_j} \|g(x_{j+1})-\nabla f(x_{j+1}) \|^2 
            + 
            \sum \limits_{j = 0}^{T}  \frac{\alpha_j}{A_j} \la g(x_{j+1}) -\nabla f(x_{j+1}), y_{j} - x_{j+1}\ra
        \right]
        \\
        \stackrel{\eqref{eq:lambda_p}}{=} 
        \frac{\sigma R}{2} \sum \limits_{j = 0}^{T} \frac{A_T \alpha_j^2}{A_j^2 (j+3)^{p+1/2}} 
        \stackrel{\eqref{alpha_t}}{=} 
        \frac{\sigma R}{(p+1)^{5/2} 2} \sum \limits_{j = 0}^{T} \frac{A_T \alpha_j^{p+5/2}}{A_j^2} \leq \frac{\sigma R}{(p+1)^{5/2} 2A_T} \sum \limits_{j = 0}^{T} \frac{A_T \alpha_j^{p+5/2}}{A_j}
        \\
        \stackrel{\eqref{eq:alpha_sum_bound}}{\leq} \frac{(p+1)^2\sigma R}{ 2A_T (T+p+1)^{p+3/2}} 
        \stackrel{\eqref{eq:A_t_bound}}{\leq} \frac{(p+1)^2\sigma R}{ 2 (T+p+1)^{1/2}}.
    \end{gather*}

    Finally, we consider $err_{T+1}^\tau$
    \begin{gather*}
        A_T \E \left[err_{T+1}^{\tau}\right] 
        =
        \sum \limits_{j = 0}^{T}  \frac{A_T\tau^2}{A_j \bar{\delta}_j} 
        \stackrel{\eqref{eq:delta_p}}{=} 
        \frac{\tau R}{(p+1)^{3/2}} \sum \limits \frac{A_T \alpha_t^{3/2}}{A_j} 
        \stackrel{\eqref{eq:alpha_sum_bound}}{\leq}
        \frac{\tau R}{(T+p+1)^{1/2}}.
    \end{gather*}

    Combining all bounds from above we achieve convergence rate
    \begin{gather*}
         \E \left[ f(x_{T+1}) - f(x^{\ast}) \right]
        \leq \frac{2(p+1)^3\tau R}{(T+p+1)^{1/2}} + \frac{3(p+1)^3\sigma R}{(T+p+1)^{1/2}} + \sum \limits_{i=3}^p \frac{2(p+1)^{2i-1}\delta_i R^i}{i! (T+p+1)^i} + \frac{(p+1)^{2(p+1)}}{(p+1)!} \frac{MR^{p+1}}{(T+p+1)^{p+1}}.
    \end{gather*}
\end{proof}
\vspace{-0.5cm}
Again, the case of stochastic Hessian (Theorem~\ref{thm:acc_convergence_p_ord} under Assumption~\ref{as:p_ord_stoch}) can be obtained in the same way by taking expectation in Lemma~\ref{lm:scalar_lb_cases_p}.

\section{Restarts for strongly-convex function}

\begin{customthm}{\ref{thm:ACRNM_conv_str_convex}}
Let Assumption \ref{as:lip_str_cvx} hold and let parameters of Algorithm \ref{alg:inexact_acc_detailed} be chosen as in \eqref{eq:p_ord_params}. Let $\{z_s\}_{s \geq 0}$ be generated by Algorithm $\ref{alg:restarts}$ and $R > 0$ be such that $\|z_0 - x^*\|\leq R$. Then for any $s \geq 0$ we have 
    \begin{align}
         \E \|z_{s}-x^*\|^2 &\leq 4^{-s} R^2, 
        \label{eq:restart_conv_arg_APP}
        \\
        \E f(z_{s}) - f(x^*) &
             \leq 2^{-2s-1} \mu R^2.
        \label{eq:restart_conv_func_app}
    \end{align}
Moreover, the total number of iterations to reach desired accuracy $\e:~f(z_s) - f(x^*)\leq \e$ in expectation is
\begin{equation}\label{eq:ACRNM_complexity_app}
    O\left( 
    \tfrac{(\tau + \sigma_1)^2}{\mu \e} 
    +
    \ls\sqrt{\tfrac{\sigma_2}{\mu}} + 1\rs\log\tfrac{f(z_0) - f(x^*)}{\e}
    + 
    \sum \limits_{i=3}^p \left(\tfrac{\sigma_i R^{i-2}}{\mu}\right)^{\frac{1}{i}}
    +  
    \left(\tfrac{L_p R^{p-1}}{\mu}\right)^\frac{1}{p+1} 
    \right).
\end{equation}
\end{customthm}

\begin{proof}
We prove by induction that $\E \|z_s-x^*\|^2\leq 4^{-s}\|z_0-x^*\|^2 = 4^{-s}R^2_0$. For $s=0$ this obviously holds. By strong convexity and convergence of Algorithm~\ref{alg:inexact_acc_detailed_p_ord}
$$
               \E \left[ f(x_{T}) - f(x^{\ast}) \right]  
               \leq 
                \frac{C_\tau \tau R}{\sqrt{T}} + \frac{C_1 \sigma_1 R}{\sqrt{T}} + \sum \limits_{i=2}^p \frac{C_i\sigma_i R^{i}}{T^{i}}  
                + \frac{C_{p+1}L_pR^{p+1}}{T^{p+1}}
$$
we get

\begin{gather*}
    \E_{[z_{s+1}\vert z_s, z_{s-1}, \ldots, z_0]}{ \|z_{s+1}-x^*\|^2}
    \leq 
    \frac{2}{\mu} \E_{[z_{s+1}\vert z_{s}, \ldots, z_0]}{(f(x_{t_{s}+1}) - f(x^*))} 
    \\
    \leq \frac{2}{\mu} \left(\frac{C_\tau \tau R}{\sqrt{T}} + \frac{C_1 \sigma_1 R}{\sqrt{T}} + \sum \limits_{i=2}^p \frac{C_i\sigma_i R^{i}}{T^{i}}  
                + \frac{C_{p+1}L_pR^{p+1}}{T^{p+1}}.\right)
    \\
    \leq \frac{2}{\mu} \left(
    \frac{\mu C_{\tau}\sigma\|{z_s} - x^*\|}{8(p+2) C_{\tau} \sigma \|z_{s} - x^*\|}
    +
     \frac{\mu C_{1}\sigma\|{z_s} - x^*\|}{8(p+2) C_{1} \sigma \|z_{s} - x^*\|}\right.\\
    + 
    \sum \limits_{i=2}^p \frac{\mu C_{i}\delta_i\|{z_s} - x^*\|^{i}}{8 (p+2) C_{i} \delta_i \|z_{s} - x^*\|^{i}} 
    \left.
    + 
    \frac{\mu C_{p+1}L_p\|{z_s} - x^*\|^{p+1}}{8 (p+2) C_{p+1} L_p \|z_{s} - x^*\|^{p+1}}
    \right)
    \\
    \leq \frac{p}{4(p+2)}\|z_s - x^*\|^2
    + 
    \frac{2}{4(p+2)}  r_s \|{z_s}-x^*\|.
\end{gather*}

Then by taking full expectation we obtain 
\begin{gather*}
    \E {\|z_{s+1}-x^*\|^2} \leq \frac{p}{4(p+1)} \E \|z_{s}-x^*\|^2 + \frac{2}{4(p+2)} r_s \E \|z_{s}-x^*\|   \leq \frac{1}{4^{s+1}}R_0^2
\end{gather*}

Thus, by induction, we obtain that \eqref{eq:restart_conv_arg_APP}, \eqref{eq:restart_conv_func_app} hold.

Next we provide the corresponding complexity bounds. From the above induction bounds, we obtain that after $S$ restarts the total number of iterations of Algorithm \ref{alg:inexact_acc_detailed_p_ord}
\begin{gather*}
    \E T = \E \sum_{s=1}^{S} t_{s}
    \leq  \sum_{s=1}^{S} \max
    \left\{1, 
    \ls\tfrac{8(p+2)C_\tau\tau}{\mu  r_{s-1}}\rs^2,
    \ls\tfrac{8(p+2) C_1 \sigma_1 }{ \mu r_{s-1}}\rs^2,
    \max \limits_{i = 1, \ldots, p}\left(\tfrac{8(p+2) C_i \delta_i R_{s-1}^{i-2}}{\mu} \right)^\frac{1}{i}, 
    \left(\tfrac{8(p+2) C_{p+1}L_p R^{p-1}_{s-1}}{\mu}\right)^\frac{1}{p+1}
    \right\}
    \\
    \leq  \sum \limits_{s=1}^S \ls
    1 + \ls\tfrac{8(p+2)C_\tau \tau}{\mu r_{s-1}}\rs^2 + \ls\tfrac{8(p+2) C_1 \sigma_1}{\mu r_{s-1}} \rs^2
    + 
    \textstyle{\sum \limits_{i=2}^p} \left(\tfrac{8(p+2)C_i\delta_i R_{s-1}^{i-2}}{\mu} \right)^\frac{1}{i} 
    +
    \left(\tfrac{8(p+2)C_{p+1}L_p R^{p-1}_{s-1}}{\mu}\right)^\frac{1}{p+1} \rs
    \\
    \leq 
    S 
    +
    \ls \tfrac{8(p+2) C_\tau \tau}{\mu R_0}\rs^2 4^S 
    +
    \ls \tfrac{8(p+2) C_1 \sigma_1}{\mu R_0} \rs^2 4^S
    +
    \left(\tfrac{8(p+2) C_2 \delta_2}{\mu} \right)^\frac{1}{2}S 
    + 
    \textstyle{\sum \limits_{i=3}^p}\left(\tfrac{8(p+2)C_i\delta_i R_0^{i-2}}{\mu} \right)^\frac{1}{i}
    +
    \left(\tfrac{8(p+2) C_{p+1}L_p R^{p-1}_{0}}{\mu}\right)^\frac{1}{p+1}
    \\
    \leq \log \frac{f(z_0) - f(x^*)}{\e} + \ls\tfrac{8(p+2) \tau}{\mu R_0} \rs^2\tfrac{f(z_0) - f(x^*)}{\e}
    + 
    \ls\tfrac{8(p+2) C_1 \sigma_1}{\mu R_0} \rs^2\tfrac{f(z_0) - f(x^*)}{\e}
    \\
    + \left(\tfrac{8(p+2)C_2\delta_2}{\mu} \right)^\frac{1}{2} \log \tfrac{f(z_0) - f(x^*)}{\e} + 
    \textstyle{\sum \limits_{i=3}^p}\left(\tfrac{8(p+2)C_i\delta_i R_0^{i-2}}{\mu} \right)^\frac{1}{i}
    +
    \left(\tfrac{8(p+2)L_p C_{p+1}R^{p-1}_{0}}{\mu}\right)^\frac{1}{p+1}
    .
\end{gather*}
Therefore, the total  oracle complexities are given by \eqref{eq:ACRNM_complexity_app}.
\end{proof}

For the case of stochastic high-order derivative the proof remains same w.r.t. change $\delta_i$ to $\sigma_i$.

\section{\aaa{On the solution of subproblem \eqref{eq:subproblem_2ord} in Algorithm \ref{alg:inexact_acc_detailed}}}

The subproblem \eqref{eq:subproblem_2ord} admits a closed form solution, which we will derive in the following steps. First, note that the problem \eqref{eq:subproblem_2ord} is convex. Let us write optimality condition
$$0 = \nabla \psi_t(y_t) = (\lambda_{t}  + \bk_2^{t} + \bk_3^{t} \|y_{t} - x_0\|)(y_{t} - x_0) + \sum \limits_{j=0}^{t-1} \frac{\alpha_j}{A_j} g(x_{j+1}).$$
Thus,
$$(\lambda_{t}  + \bk_2^{t} + \bk_3^{t} \|y_{t} - x_0\|)\|(y_{t} - x_0)\|  =  \| \sum \limits_{j=0}^{t-1}\frac{\alpha_j}{A_j} g(x_{j+1})\|.$$
Let us denote $r_{t} = \|(y_{t} - x_0)\|$ and $S_t = \| \sum \limits_{j=0}^{t-1}\frac{\alpha_j}{A_j} g(x_{j+1})\|$. Then, we get
$$\bk_3^{t} r_{t}^2 + (\lambda_{t}  + \bk_2^{t})r_{t} - S_t = 0.$$
Next, we get the solution of quadratic equation
$$r_{t} = \frac{\sqrt{(\lambda_{t}  + \bk_2^{t})^2 + 4 \bk_3^{t}S_t } - (\lambda_{t+1} + \bk_2^{t} )}{2\bk_3^{t}}.$$
Finally, we get explicit solution
$$y_{t} = x_0 - \frac{\sum \limits_{j=0}^{t-1}\tfrac{\alpha_j}{A_j} g(x_{j+1}) }{\lambda_t + \bk_2^t + \bk_3^t r_t}.$$
We use the explicit solution in our implementation of the Algorithm~\ref{alg:inexact_acc_detailed}.

\end{document}